\documentclass[11pt, leqno, fleqn]{amsart}
\usepackage[normalem]{ulem}
\usepackage[titletoc]{appendix}
\usepackage{lineno}

\usepackage{graphicx, psfrag, epstopdf}
\usepackage{hyperref}
\usepackage{comment}
\usepackage{dsfont}
\usepackage{amsmath}
\usepackage{amsfonts}
\usepackage[fleqn]{mathtools}
\usepackage{amssymb}
\usepackage{graphicx}
\usepackage[all,cmtip]{xy}
\usepackage{soul}
\usepackage{color}
\usepackage{yfonts}
\usepackage{paralist}

\usepackage{comment}
\usepackage{stmaryrd}

\usepackage{amsxtra}

\usepackage{mdframed}

\def\cU{\mathfrak U}

\def\cW{\mathcal W}

\def\eps{\varepsilon}

\def \kb{ {\mathfrak{b}} }

\newcommand{\norm}[1]{\| #1\|}

\def\cB{          \mathcal B}

\def\cF{          \mathcal F}
\def\cH{          \mathcal H}
\def\cN{          \mathcal N}
\def\cE{          \mathcal E}

\def\maxder{  D_1 } 

\def\maxderII{  D_2 }

\def\SL{   {\rm SL}}

\def\SO{   {\rm SO}}
\def\SU{   {\rm SU}}

\newcommand{\SectionSymbol}{\textsection}

\def\bt{        \mathbf{t}}
\def\obt{       \overline{ \mathbf{t}}}

\def \R{{\mathbb R}}

\def \Z{{\mathbb Z}}
\def \N{{\mathbb N}}

\newcommand{\abse}[1]{\left\lvert#1\right\rvert} 

\newcommand{\normBig}[1]{\Bigl\lVert#1\Bigr\rVert}

\newcommand{\T}{{\mathbb T}}
\newcommand{\prf}{{\begin{proof}}}
	\newcommand{\epf}{{\end{proof}}}

\newcommand{\bd}{\mathbf{d}}

\newcommand{\dd}{\,\mathrm{d}}

\DeclareMathOperator{\Lip}{Lip}
\DeclareMathOperator{\Gr}{Gr}

\DeclareMathOperator{\leb}{Leb}
\DeclareMathOperator{\SubP}{(S^{-})}
\DeclareMathOperator{\SupP}{(S^+)}
\DeclareMathOperator{\NCSup}{(NC^+)}
\DeclareMathOperator{\NCSub}{(NC^{-})}
\DeclareMathOperator{\OO}{O}
\DeclareMathOperator{\ini}{ini}

\newcommand{\acts}{\curvearrowright}

\newcommand{\increasing}{\boxslash}         

\newcommand{\EV}{{\mathbb E}}

\newcommand{\cP}{{\mathcal P}}

\newcommand{\cK}{{\mathcal K}}

\newcommand{\C}{{\mathbb C}}

\newcommand{\ary}{\begin{eqnarray}}
	\newcommand{\eary}{\end{eqnarray}}

\newcommand{\aryst}{\begin{eqnarray*}}
	\newcommand{\earyst}{\end{eqnarray*}}

\newcommand{\enmt}{\begin{enumerate}}
	\newcommand{\eenmt}{\end{enumerate}}

\usepackage{mathptmx}

\DeclareSymbolFont{cmletters}{OML}{cmm}{m}{it}
\SetSymbolFont{cmletters}{bold}{OML}{cmm}{b}{it}
\DeclareMathSymbol{v}{\mathalpha}{cmletters}{`v}

\usepackage{bm}

\newtheorem{thm}{\bf Theorem}[section]

\newtheorem{lemma}[thm]{\bf Lemma}
\newtheorem{prop}[thm]{\bf Proposition}
\newtheorem{defi}[thm]{\bf Definition}

\newtheorem{cor}[thm]{\bf Corollary}

\theoremstyle{definition}

\newtheorem{rema}[thm]{\bf Remark}
\newtheorem{setting}[thm]{\bf Setting}

\DeclareMathOperator{\diff}{Diff}

\DeclareMathOperator{\diam}{diam}

\theoremstyle{definition}

\def\bee{\begin{equation}}
	\def\eee{\end{equation}}

\newif\ifA

\def\supp{{\rm supp}}

\def\cL{\mathcal{L}}

\def\dist{ {\rm dist} }

\def\S{ {\mathbb{S}} }

\DeclareMathOperator{\Ad}{Ad}
\DeclareMathOperator{\Ker}{Ker}
\DeclareMathOperator{\GL}{GL}

\def \vol{  {\rm vol} }

\newcommand{\Diff}{{\rm Diff}}

		\newcommand{\pdvr}[2]
		{\dfrac{\partial^{#2} #1}{\partial \theta^{#2_1} \partial r^{#2_2}}}
		\newcommand{\pdvrs}[2]
		{\partial^{#2} #1 /\partial \theta^{#2_1} \partial r^{#2_2}}
		\newcommand{\ug}{{\underline g}}

		\newcommand{\diag}{\mathop{\rm diag}}

		\definecolor{blue}{rgb}{0,0,1}
		\definecolor{red}{rgb}{1,0,0}
		\definecolor{green}{rgb}{0,.7,0}

		\DeclareFontFamily{U}{solomos}{}
		\DeclareErrorFont{U}{solomos}{m}{n}{10}
		\DeclareFontShape{U}{solomos}{m}{n}{
			<-> s*[1.1]  gsolomos8r
		}{}

  \usepackage{tikz}
\newlength\squareheight
\begin{document}

			\title[Rigidity and equidistribution of random walks by diffeomorphisms]{Rigidity and equidistribution of random walks by diffeomorphisms near the conservative regime}
			

				\author{Timoth\'ee B\'enard}				
			\address{Timoth\'ee B\'enard, 
				      Institut Galil\'ee
				     Universit\'e Paris 13, CNRS UMR 7539,
				  	99 avenue Jean-Baptiste Cl\'ement
					93430 - Villetaneuse
			}
							\email{benard@math.univ-paris13.fr} 
			
			\author{Zhiyuan Zhang}
			
			\address{Zhiyuan Zhang,
				Department of Mathematics, 
				Imperial College London, 
				London SW7 2AZ, 
				United Kingdom
			}
			\email{zhiyuan.zhang@imperial.ac.uk} 		
 			  

\renewcommand{\subjclassname}{%
  \textup{2020} Mathematics Subject Classification}
  
			\begin{abstract}
			We consider a random walk on a closed manifold $M$ driven by a probability measure $\mu$ on the set of $C^2$-diffeomorphisms. Provided $\mu$ has compact support, satisfies certain gap, pinching, and ergodicity conditions, and is weak-$*$ close to volume-preserving, we prove $M$ carries exactly one atom-free stationary probability measure $\Upsilon_{\mu}$.
              It has finite high-dimensional Riesz energy   
            and coincides with volume in the volume-preserving setting. Moreover, for every $x\in M$, the $n$-step distribution $\mu^{*n}*\delta_{x}$ converges to  $\Upsilon_{\mu}$ unless $x$ is trapped in a finite $\mu$-invariant set. Our result applies to a variety of situations: bi-expanding random walks on surfaces, non-linear perturbations of Zariski-dense random walks on the torus $\T^d$,  on cocompact lattice quotients of $\SO(2,1), \SO(3,1)$, or on the sphere $\S^d$.
			
\end{abstract}

			\maketitle
			 
			\tableofcontents
			
			\section{Introduction} \label{sec intro}

Let $M$ be a closed smooth Riemannian manifold of dimension $d\geq2$. Let $\mu$ be a probability measure on $\Diff^2(M)$, the space of $C^2$-diffeomorphisms of $M$. Given $n\in \N$, $x \in M$, set $\mu^{*n}*\delta_{x}$ to be the law of $g_{n}\dots g_{1}x$ where the $g_{i}$ are i.i.d. variables of law $\mu$. The probability measure  $\mu^{*n}*\delta_{x}$ may be interpreted as the $n$-step distribution of a random walk on $M$ with origin $x$ and driven by random  instructions   given by  $\mu$. We refer to this process as the $\mu$-walk on $M$.

Our aim  is to study the $\mu$-walk on $M$, by classifying stationary measures and establishing equidistribution properties under  general 
 conditions on $\mu$. Recall that a probability measure $\nu$ on $M$ is $\mu$-stationary if it satisfies $\nu =\int    g_{\star}   \nu d\mu(g)$.   We equip $M$ with a volume form and write $\vol$ for the associated normalized volume measure on $M$. We denote by $\Gamma_{\mu}:=\langle \supp(\mu) \rangle_{\text{grp}}$  the group generated by the support of $\mu$. Assuming $\Gamma_{\mu}$ preserves  $\vol$, 
 and $\mu$ satisfies certain expansion and pinching conditions, we  show the following:
\begin{itemize}
\item[(a)] {\bf Rigidity}: The volume measure $\vol$ is the only atom-free $\mu$-stationary probability measure on $M$. 

\bigskip
\item[(b)] {\bf Equidistribution}:  For every $x\in M$ with infinite $\Gamma_{\mu}$-orbit, we have the weak-$*$ convergence:
$$\mu^{*n}*\delta_{x}\rightarrow \vol. $$
\end{itemize}

Our result applies in a \emph{variety of situations}, including perturbations of Zariski-dense linear random walks on the torus $\mathbb{T}^d$, on cocompact lattice quotients of $\mathrm{SO}(2,1)$ and $\mathrm{SO}(3,1)$, and on the sphere $\S^d$. We also treat the case $d=2$ for general surfaces, without requiring perturbations of underlying linear dynamics.

Our methods also apply in \emph{dissipative} settings,  when $\Gamma_\mu$ is not assumed to preserve volume.  We do require $\mu$ to be weak-$*$ close to a volume-preserving probability measure (conservative case), but this still allows $\supp(\mu)$ to contain highly  volume-disrupting elements, although with small mass. In this context, the rigidity statement must be modified as follows: there exists a unique atom-free  $\mu$-stationary probability measure $\Upsilon_\mu$ on $M$. Moreover, $ \Upsilon_{\mu}$ belongs to the Sobolev space $H^{-s}(M)$ where $s>0$ can be made arbitrarily small provided $\mu$ is close enough to the conservative regime accordingly. We recall that  $\Upsilon_{\mu}\in H^{-s}(M)$ means equivalently finiteness of the $(d-2s)$-Riesz energy when $s \in (0, d/2)$: 
$$
\iint_{M\times M}
\frac{1}
{d_M(x,y)^{d-2s}} d \Upsilon_\mu (x)\,d \Upsilon_\mu (y)<\infty.
$$
The negative Sobolev regularity obtained in the general dissipative
setting reflects a genuine possibility of dimension loss:
$\Upsilon_\mu$ need not have full Hausdorff dimension, even
arbitrarily close to the conservative regime. A family of examples
on $\T^3$, to be presented in a separate note, exhibits this
phenomenon; see Remark \ref{rem dimension loss}.
 
\bigskip

\noindent\emph{Related works}. Rigidity phenomena for stationary measures were first studied in the linear setting. In their breakthrough work \cite{BFLM}, Bourgain-Furman-Lindenstrauss-Mozes consider a random walk on the torus $\T^d$ driven by a finitely supported strongly irreducible, proximal  probability measure on $\SL_{d}(\Z)$. By means of Fourier analysis and techniques from additive combinatorics, they are able to show rigidity (a) and equidistribution (b) as stated above. Soon after, Benoist-Quint \cite{BQ1} provide another proof of (a) for linear walks on the torus,  making use of completely different tools inspired by the proof of Ratner's rigidity theorems for unipotent flows. Benoist-Quint's method has the advantage of extending to arbitrary lattice quotients of Lie groups provided the measure driving the walk satisfies a certain semisimplicity assumption \cite{BQ1, BQ2}. In \cite{BQ3}, classification of stationary measures allows Benoist-Quint to derive a weak version of (b), in the sense that convergence only holds in Ces\`aro average:
$$\frac{1}{n}\sum_{k=0}^{n-1}\mu^{*k}*\delta_{x}\rightarrow \vol. $$

Regarding non-homogeneous  settings, Benoist-Quint's approach has been influential in the  groundbreaking    work of Eskin-Mirzakhani \cite{EM}, on rigidity properties of the $\SL_{2}(\R)$-action on the moduli space of translation surfaces, and in the  landmark work of Brown-Rodriguez Hertz \cite{BRH} dealing with random walks by diffeomorphisms on closed surfaces. The latter shows that any hyperbolic ergodic stationary measure  for a random walk on a surface falls in one of  three categories: either it is uniform on a finite invariant set, or the stable direction of the walk is non-random,
 or the measure is SRB, meaning it is leafwise absolutely continuous for a notion of unstable foliation.   Building on  techniques from Tsujii \cite{T2},  Brown-Lee-Obata-Ruan  \cite{BLOR1, BLOR2} are able to promote this SRB conclusion to absolute continuity with respect to volume under certain expansion conditions.  
 Note the works \cite{BRH, BLOR1, BLOR2} are constrained to dimension $2$, and rely on techniques in the spirit of Benoist-Quint/Eskin-Mirzakhani.   For that reason, they can  derive equidistribution in Ces\`aro average only. The question of removing the Ces\`aro average is explicitly asked in \cite[Question 1.2]{BLOR2}.
\emph{Our result deals with arbitrary dimension $d\geq 2$ and shows equidistribution  without extra Ces\`aro-averaging}.

  In their recent {\it tour de force} \cite{BEFRH},
Brown--Eskin--Filip--Rodriguez Hertz establish
measure rigidity results in a general framework,
applicable in all dimensions. In the deterministic
setting, they establish additional invariance
properties of generalized $u$-Gibbs states under
a Quantitative Non-Integrability (QNI) condition;
see \cite{EPZ} for related geometric criteria.
In the random walk setting, they show that for
volume-preserving random walks satisfying uniform
expansion on all proper exterior powers of the
tangent bundle, every stationary measure is
$\Gamma_\mu$-invariant. They further show that
every atom-free ergodic stationary measure
without zero Lyapunov exponents is absolutely
continuous with respect to volume.
One advantage of our approach is that
\emph{we allow for the presence of zero
Lyapunov exponents}. While these are excluded
on surfaces under  expansion and
conservativity hypotheses, they occur
naturally in higher dimensions.  For instance, if $\mu$ is symmetric (i.e. invariant under $g\mapsto g^{-1}$) and $M$  is odd-dimensional, then every ergodic   $\Gamma_{\mu}$-invariant measure  has a zero Lyapunov exponent.  There are also examples of random walks having a non-$\Gamma_\mu$-invariant ergodic stationary measure with a zero Lyapunov exponent.

The approach we follow is in the spirit  of that of \cite{BFLM} for the torus case mentioned above, and more precisely its recent extension to the setting of lattice quotients of simple Lie groups \cite{BH1, BH2}. Noting that rigidity  (a) follows from equidistribution (b), the proof is divided into three steps: (Phase I) for $x$ with infinite $\Gamma_{\mu}$-orbit and large enough $n$, we show the distribution $\mu^{*n}*\delta_{x}$ has positive dimension above scale $\rho_{\mu,x}(n)>0$ where $\lim_{n\to +\infty}\rho_{\mu,x}(n)=0$. (Phase II) We bootstrap positive dimension to almost full dimension using a local variant of the multislicing method from \cite{BH1, BH2}. (Phase III)  We promote high dimension to equidistribution by invoking the spectral gap of the convolution operator $f\mapsto f*\mu$ on an appropriate Sobolev space. This spectral result follows from recent work of  \cite{DD, DTZ, DDZ}. 

We point out that the method of proof is quite close to providing a rate for the  convergence $\mu^{*n}*\delta_{x}\rightarrow \Upsilon_{\mu}$. Indeed, Phases II and III are quantitative; the only obstruction lies in Phase I where the function $\rho_{\mu,x}(n)$ is not explicit. In the homogeneous setting, making Phase I quantitative is possible but usually requires arithmetic assumptions (as in \cite{BFLM, BH1, BH2, BG, LMW,Y} etc.) which are not naturally available in the smooth setting. Obtaining positive-dimensional  estimates without arithmetic input is a well-known open question.

\bigskip
\begin{center}
*

*\,\,\,\,*
\end{center}
\bigskip

We now pass to the more precise description of our results. We start by presenting applications to  concrete dynamical contexts (\SectionSymbol  \ref{Sec-examples}), and then we will phrase the theorem underlying all these statements (\SectionSymbol \ref{Sec-Mainresults}).

\subsection{Examples  of application} \label{Sec-examples}

Let $M$ be a closed  connected smooth Riemannian manifold of dimension $d\geq2$. Equip $M$ with a volume form and write  $\vol$ for the associated volume  measure on $M$,  normalized so that $\vol(M)=1$. 


We denote by $\Diff^k(M)$ 
 the set of $C^k$-diffeomorphisms of $M$ endowed with the $C^k$-topology, and $\Diff^k(M, \vol)\subseteq \Diff^k(M)$ the closed subset of those  preserving the volume.  Given a Borel set $A\subset \Diff^k(M)$, we write $\cP(A)$ for the set of Borel probability measures on $A$, endowed with the weak-$*$ topology, and call $\cP_{c}(A)$ the subset of measures having compact support. The subsets of $\Diff^k(M)$  we will consider will always be Borel subsets, though sometimes we may omit to mention it. Given $\mu\in \cP(\Diff^k(M))$ we write $\Gamma_{\mu}$ for the group generated by its support.

\bigskip
	
Our main results  apply to show rigidity of stationary measures and equidistribution in various contexts, as recorded in the next theorem.

\begin{thm} \label{examplesI,II,III}
Let $\mu_{0}$ be a compactly supported probability measure on $ \Diff^{2}(M, \vol)$.  Let  $\cK \subseteq  \Diff^{2}(M)$ be a   compact    set containing $\supp(\mu_0)$.

Assume $\mu_{0}$ pertains  to either the surface case (I), or the torus case (II), or the lattice quotient case (III) described below. 

	Then there is   a neighborhood $\cU$  of $\mu_0$ in $\cP(\cK)$ with respect to the weak-$\star$ topology, such that  every  $\mu \in \cU$ has the following properties: 
	\begin{itemize}
	\item[(a)] There exists a unique atom-free $\mu$-stationary probability measure $ \Upsilon_{\mu}$ on $M$.  
    Moreover  $\Upsilon_{\mu}$ is absolutely continuous with density in the Sobolev space $H^s(M)$ for some $s=s(M,\mu_{0},\cK)>0$.
	
	\item[(b)] For every  $x \in M$, either $x$ has finite $\Gamma_{\mu}$-orbit, or $\lim_{n \to \infty} \mu^{*n} * \delta_{x} = \Upsilon_{\mu}$.   
	\end{itemize}
\end{thm}

\begin{rema} The compactness assumption on $\supp(\mu_0)$ and $\cK$  can be weakened to requiring that $\supp(\mu_0)$, $\cK$ and their inverse sets $\supp(\mu^{-1}_0)$ and $\cK^{-1}$ are bounded in $\Diff^2(M)$. Indeed, the proof works under these assumptions. We impose compactness for simplicity of exposition.
\end{rema}

\begin{rema}
If the measure $\mu$ in Theorem \ref{examplesI,II,III}   is supported on  volume-preserving diffeomorphisms,  then  $$\Upsilon_{\mu} = \vol,$$ so we recover items (a) and (b) from the introductory paragraph. 
\end{rema}
			
			To describe the  settings mentioned in the theorem, we use the following  notion of expanding measure. 	It  already appears in the literature under various names:  \lq\lq uniformly expanding in dimension $1$\rq\rq \	in \cite{BEFRH},  \lq\lq uniformly expanding on average in the future\rq\rq \  in \cite{BLOR1},  and   \lq\lq expanding on average\rq\rq  \ in \cite{DD}.	
					
	\begin{defi} \label{def expandingonaverage}
		Let $n_{0}\in \N$, $\kappa\in \R$. We say that $\mu \in \cP_c(\Diff^1(M))$ is $(n_{0}, \kappa)$-{\rm expanding} if 
		\aryst
	\inf_{x \in M, v \in T_xM \setminus \{0\}}	\int_{\Diff^1(M)} \log (\| Dg(x, v) \| / \| v \|) d\mu^{* n_0}(g) > n_{0}\kappa.
		\earyst 
	\end{defi}			
		\noindent	We also say that $\mu$ is $ \kappa$-expanding if it is $(n_{0}, \kappa)$-expanding for some $n_{0}$. Finally, we say $\mu$ is expanding if it is $0$-expanding.
			
			\bigskip
				
				We now describe the three settings addressed in Theorem \ref{examplesI,II,III}.

				\bigskip 
	 		\noindent {\bf Case I} (\emph{Perturbations of Surface maps}).
			
\noindent
\vrule width 1pt\hspace{0.4em}
\parbox{0.9\linewidth}{
Assume $M$ has dimension $d=2$, and both the measures $\mu_{0}$ and $\mu_{0}^{-1}$ are expanding. Here $\mu_{0}^{-1}$ denotes the image of $\mu_{0}$ by the inverse map $g\mapsto g^{-1}$.  
}

\begin{rema} \label{rema-dim2}
The property of being expanding  is well studied   in dimension $2$ volume-preserving   random walks.  
In this context, we have the following.
\begin{enumerate}
	\item The expanding property  holds for a weak-$*$ dense set of measures $\mu$ in $\cP(\Diff^{\infty}(M, \vol))$.  This follows from adapting the argument of Potrie \cite{P} or Elliott Smith \cite{ES}.

	\item It is equivalent to the \emph{$1$-gap property}  and to the {\it coexpanding  property} (see Definitions \ref{def gap}, \ref{defi coexpanding} and Lemma \ref{lem contractinggaptocoexpanding0}). When $\mu$ is supported in $\Diff^2(M, \vol)$, it is known to imply exponential mixing for the $\mu$-random walk on $M$  (see \cite[Theorem 1.1]{DD2} and \cite{DD}). 
	\end{enumerate}
\end{rema}

In case (I) of Theorem \ref{examplesI,II,III}, item (a) recovers
the rigidity statement of \cite[Thm.~B]{BLOR2}, while item (b)
strengthens the equidistribution statement of \cite[Thm.~C]{BLOR2}
by removing the Ces\`aro average, thereby answering
\cite[Question~1.2]{BLOR2} positively. In the conservative setting,
rigidity and Ces\`aro equidistribution for uniformly expanding
random walks on surfaces were previously established by Chung
\cite{C}. Both \cite{C} and \cite{BLOR2} ultimately rely
on drift arguments in the spirit of Benoist--Quint and
Eskin--Mirzakhani (via \cite{BRH}), whereas our approach is rooted
in Bourgain's sum-product theorem. We believe that our techniques
may also lead to progress in the study of random walks on complex
surfaces, investigated in \cite{CD,R}, and leave this direction
for future work.
 	\bigskip

 		\noindent {\bf Case II} (\emph{Perturbations of linear random walks on $\T^d$}).
		
		\noindent
\vrule width 1pt\hspace{0.4em}
\parbox{0.9\linewidth}{
Consider $M = \T^d$, and let $\vol$ denote the  normalized  Lebesgue measure on $\T^d$.
	Let $\mu_0$ be a probability measure on $\SL(d, \Z)$ whose support  is finite and generates a Zariski-dense subgroup of $\SL(d, \Z)$. 	Through the action of $\SL_{d}(\Z)$ on $\T^d$, we may view $\mu_{0}$ as a measure on  $\Diff^{2}(\T^d,\vol)$, and consider the induced volume-preserving random walk on $\T^d$.		 
}
	\bigskip		
	
	In the linear regime   $\mu=\mu_{0}$, Theorem \ref{examplesI,II,III} case (II)  is a consequence of Bourgain-Furman-Lindenstrauss-Mozes \cite{BFLM}, see also \cite{BQ1}. 
	Theorem \ref{examplesI,II,III} case (II) generalizes \cite[Theorem A]{BLOR2} which deals with conclusion (a) for the torus of  dimension $2$.

		\bigskip		
 	\noindent {\bf Case III} (\emph{Perturbations of homogeneous random walks on  $H / \Lambda$}).
	
		\noindent
\vrule width 1pt\hspace{0.4em}
\parbox{0.9\linewidth}{Let $M=H/ \Lambda$ where $H\in \{ \SO(2,1)^{\circ}, \SO(3,1)^{\circ}\}$ and $\Lambda$ is a cocompact lattice in $H$. Let $\vol$ denote the Haar probability measure on $M$.   Let $\mu'_0$ be a compactly supported Borel probability measure on $H$ whose support  generates a Zariski-dense subgroup of $H$. Define $\mu_{0}$ to be the image of $\mu_{0}'$  in $\Diff^{ 2}(M, \vol)$ via the map $H\rightarrow \Diff(M, \vol), h\mapsto L_{h}$ where  $L_h(x) = hx$. }
		 
\bigskip	
	In the linear regime   $\mu=\mu_{0}$, Theorem \ref{examplesI,II,III} case (III)   is a consequence of \cite{BH1}. 
	 	So far our proof does not cover more general homogeneous spaces.  We will comment on this in Remark \ref{disc-moregen-simplegroups}.
	 \bigskip

  \bigskip
We may also apply our method to study random walks on the sphere $M=\S^d$ driven by random diffeomorphisms close to   rotations. Here we rely on a rigidity theorem of \cite{DK} claiming that a perturbation of an equal-weight finitely supported measure on $\SO(d+1)$ spanning a dense subgroup is either conjugated to a measure on $\SO(d+1)$ or has simple Lyapunov spectrum. 	Invoking this theorem requires us to restrict weak-$*$ perturbation to pointwise perturbation and remain within the set of volume-preserving diffeomorphisms with high regularity.

		\begin{thm}[Perturbations of random rotations on $\S^d$]  \label{thm Sphere}
		Let $\ell_{0}=\ell_{0}(d)\geq 1$ be large enough depending on $d$. Let $\mu_{0}$ be a probability measure on $\SO(d+1)$ of the form $\mu_{0}=\frac{1}{k} \sum^k_{i=1}\delta_{R_i}$ where $R_{1}, \dots, R_{k}$ span a dense subgroup of $\SO(d+1)$. 
        Then there exists $\mathcal{V}$ a neighborhood of $(R_{i})_{i=1}^k$ in $\Diff^{\ell_{0}}(M, \vol)^k$   
        such that for all  $(S_{i})_{i=1}^k\in \mathcal{V}$, 
        the measure $\mu:= \frac{1}{k} \sum_{i}\delta_{S_{i}}$ satisfies conclusions 
 $(a)$ and $(b)$ from Theorem \ref{examplesI,II,III} (with $\Upsilon_{\mu}=\vol$ the Haar measure).
 \end{thm}

In the linear regime $\mu=\mu_{0}$, Theorem \ref{thm Sphere} follows from It\^o-Kawada's theorem \cite{IK}. 
In the above perturbed context, Dolgopyat-Krikorian \cite{DK} show that for  $d\in 2\N$ even, the volume measure is the only \emph{absolutely continuous} $\mu$-stationary probability on $\S^d$. This result is extended to all $d\geq 2$ in the recent work of  DeWitt, Dolgopyat and the second author \cite{DDZ}.  Theorem \ref{thm Sphere} builds on \cite{DK, DDZ} to rule out a much bigger class of measures: ergodic stationary measures are either atomic or the volume measure. The conclusion of equidistribution  (b) is also new in the non-linear regime.

\subsection{Main results} \label{Sec-Mainresults}

All the results in Section \ref{sec intro} will follow from a more general statement, Theorem \ref{thm mainDissipative}, presented below. 
We keep the notations $(M, \vol)$ of \SectionSymbol \ref{Sec-examples}. To lighten the tracking of dependences, we assume without loss of generality\footnote{We may always reduce to this setting by suitably rescaling  the metric at every point in $M$.} 
that $\vol$ is the normalized volume measure determined by the metric on $M$. 
As a shorthand, we set $G := \Diff^1(M)$, the group of  $C^{1}$ diffeomorphisms on $M$, and $G_{\vol} := \Diff^{1}(M, \vol)$, the subgroup of those preserving the volume.

	We first need to  introduce a gap condition. Informally, a measure $\mu$ on $G$ has a $b$-gap if at every point $x$, the $b$-th and $(b+1)$-th smallest Lyapunov exponents  are separated by a constant $\kappa=\kappa(\mu)>0$.
 We formulate this condition by asking that, for the action of $g \acts M$ where $g$ is chosen by $\mu$ (or an iterate),  the strongest \lq\lq contraction\rq\rq \ along any $(d-b)$-dimensional tangent subspace dominates the strongest \lq\lq expansion\rq\rq \ along its $b$-dimensional orthogonal complement.

\begin{defi}[Gap] \label{def gap} 
	Consider integers $1 \leq b \leq d-1$, $n_0 \geq 1$, and a real $\kappa > 0$. Let $\mu$ be a compactly supported Borel probability measure on $G$. We say $\mu$  has a {\em $(n_0, \kappa, b)$-gap } if for every $x \in M$ and every $V \in \Gr(T_x M, d - b)$, we have
	\aryst
	 	\int_{G}  \Big ( \log \sup_{u \in \S(V^{\perp})} \| P_{Dg(x, V)^{\perp}}(  Dg(x, u) ) \| - \log \inf_{v \in \S(V)} \| Dg(x, v) \|  \Big )  d\mu^{* n_0}(g) < - n_0 \kappa.
	\earyst

\end{defi}
	We say that $\mu$ has a {\em $b$-gap}  if there exist an integer $n_0 \geq 1$ and   $\kappa > 0$ such that $\mu$ has a $(n_{0}, \kappa,  b)$-gap. 
If $J\subseteq \N$ is a set of integers, we also say $\mu$ has a $(n_0, \kappa, J)$-gap if it has a $(n_0, \kappa, b)$-gap for all $b\in J\cap [1, d-1]$. Finally, we say $\mu$ has a $(\kappa, J)$-gap if it has a $(n_0, \kappa, J)$-gap for some $n_{0}\in \N$.
 
 The notion of $b$-gap is related to the property of being $(n_0, \kappa_1,  \kappa_2,  b)$-uniform stated in \cite[Definition 8]{Z} (see also \cite[Definition 1.1]{DDZ}). 
 
Given any   compact   set $\cK\subseteq G$, it is direct to see that  having a  $(n_0, \kappa, b)$-gap is a weak-$\star$  open condition on $\cP(\cK)$.   


\bigskip

Our first general result is the following. 
%



\begin{thm} \label{main thm 2Dissipative}
	Let $\mu_0$ be a  Borel probability measure on $  \Diff^{3}(M, \vol)$ with compact support in $\Diff^2(M)$. 
	Assume that  $\mu_0$  has a $b$-gap  for every $1 \leq b\leq d - 1$, and $\mu_{0}^{-1}$ is expanding.
	Let  $\cK \subseteq   \Diff^{2}(M)$ be a   compact    set containing $\supp(\mu_0)$.
	Let $\cU$ be a neighborhood of $\mu_0$ in $\cP(\cK)$ with respect to the weak-$\star$ topology, let $\mu \in \cU$.
	
	If $\cU$ is small enough depending on $M, \mu_{0},\cK$, then the following holds:
	
	There exists a unique atom-free $\mu$-stationary probability measure $\Upsilon_{\mu}$ on $M$; and
      for every  $x \in M$, either $x$ has finite $\Gamma_{\mu}$-orbit, or  
	$$\lim_{n \to \infty} \mu^{*n} * \delta_{x} = \Upsilon_{\mu}.$$   
        Moreover,    given any  $s>0$,  the measure $\Upsilon_{\mu}$ belongs to $H^{-s}(M)$ provided that $\cU$ is small enough depending on $M, \mu_{0},\cK, s$.
\end{thm}


	

%
 
\begin{rema}
Note that compared to Theorem \ref{examplesI,II,III},  and Theorem \ref{thm mainDissipative} below, we assume here $C^3$-regularity for $\mu_0$, instead of $C^2$-regularity. This is because our argument uses  total ergodicity of the $\mu_0$-random walk, which is given as an assumption in Theorem \ref{thm mainDissipative}, but can also be checked in the present full gap setting under $C^3$-regularity (see Theorem \ref{Thm Ergodicity0}). In the context of Theorem \ref{examplesI,II,III}, total ergodicity can be checked directly  (no need to require $C^3$-regularity there).
\end{rema}

The condition in Theorem \ref{main thm 2Dissipative} allows for any values of the Lyapunov exponents; however, they  need to have multiplicity one. This rules out applicability to perturbations of linear random walks on complex homogeneous spaces, such as quotients of $\SO(3,1)^{\circ}$ (which is isogenous to $\SL_{2}(\C)$). 
In order to allow some multiplicity in the Lyapunov spectrum, we introduce  a notion of pinching. Informally, a measure $\mu$ on $G=\Diff^1(M)$ is $(\eta, b_{0}, b_{1})$-pinched if, at every point $x\in M$, the $i$-th smallest Lyapunov exponents, where $b_{0}+1 \leq i\leq b_{1}$, are all $\eta$-close to one another. Similarly to the notion of gap, we formulate this in terms of contraction/expansion.


\begin{defi}[Pinching] \label{def pinch}
	Consider   integers $0 \leq b_0  < b_1 \leq d$,  $n_0 \geq 1$, and a real $\eta > 0$. 
	Let $\mu$ be a compactly supported Borel probability measure on $G$. We say $\mu$  is { \em $(n_0, \eta, b_0, b_1)$-pinched} if for every $x \in M$, every $V_0 \in \Gr(T_xM, d - b_0)$, and every $V_1 \in \Gr(T_xM,  d - b_1)$ with $V_1 \subset V_0$, we have
	\aryst
		\int_{G}  \Big ( \log \sup_{u \in \S(  V_1^{\perp})}  \| P_{  Dg(x, V_{1})^{\perp}}(  Dg(x, u) ) \|  - \log \inf_{v \in \S(V_0)}  \| Dg(x, v) \|    \Big ) d\mu^{* n_0}(g)   < n_0 \eta.
	\earyst
\end{defi}
  
Given $J\subseteq \N$, we have an associated notion of $(n_0, \eta, J)$-pinching which means $(n_0, \eta, b_{0},b_{1})$-pinching for all consecutive elements $b_{0},b_{1}$ of $J\cap [0, d]$. We may also drop the parameter $n_{0}$ and leave it  implicit.
  
Given any   compact   set $\cK\subseteq G$, it is direct to see that  being  $(n_0, \eta, b_{0},b_{1})$-pinched is a weak-$\star$  open condition on $\cP(\cK)$.

\bigskip

The main result of this paper is the following.

\begin{thm} \label{thm mainDissipative}
 	Let  $\kappa, \eta > 0$,  let $J = \{  0 = d_0 < d_1 <   \cdots < d_m < d_{m+1} = d \}$ be a subset of $ \{0, \cdots, d \}$   with at least $3$ elements (i.e.  $m\geq1$).  
	Let $\mu_0$ be a compactly supported Borel probability measure on $  \Diff^{2}(M, \vol)$ such that 
	(1) $\mu_0$  has  a $(\kappa, J)$-gap and is $(\eta, J)$-pinched, 
	(2)   $\mu_{0}^{-1}$ is $(-\kappa/d)$-expanding,   
	(3) $T_{\mu_0}$ is totally ergodic in $L^2(M, \vol)$.
	Let $\cK  \subseteq   \Diff^{2}(M)$ be a   compact   subset containing $\supp(\mu_0)$, and let  
	$\cU$ be a  neighborhood  of $\mu_0$ in $\cP(\cK)$ with respect to the weak-$*$ topology.  Let $\mu\in \cU$.

	If  $\eta, \cU$ are small enough depending on $M,\mu_{0}, \cK$, then the following is true:
	
		There exists a unique atom-free $\mu$-stationary probability measure $ \Upsilon_{\mu}$ on $M$; and for every  $x \in M$, either $x$ has finite $\Gamma_{\mu}$-orbit, or  
	$$\lim_{n \to \infty} \mu^{*n} * \delta_{x} = \Upsilon_{\mu}.$$   
    Moreover,    given any  $s>0$,  the measure $\Upsilon_{\mu}$ belongs to $H^{-s}(M)$ provided that $\eta, \cU$ are small enough depending on $M, \mu_{0},\cK, s$. 
\end{thm}

 In the above statement,  $T_{\mu_0}$ refers to the Markov operator associated to $\mu_{0}$, acting on $L^2(M, \vol)$ by the formula $(T_{\mu_{0}}f)(x)=\int_{G}f(gx)d\mu_0(g)$. The total ergodicity condition means that for all $k\geq 1$, the fixed points of $T^k_{\mu_{0}}$  on $L^2(M, \vol)$ are limited to  constant functions, which amounts to  $\vol$ being $\mu_0^{*k}$-ergodic.

\begin{rema}[$H^{-s}$-regularity]
Given $s>0$, recall that a Borel probability measure
$\Upsilon$ on $M$ belongs to $H^{-s}(M)$ if it defines
a continuous linear form on $H^s(M)$, meaning that
there exists $C>0$ such that, for every $f\in C^\infty(M)$,
$$
\left|\int_M f\,d\Upsilon\right|
\leq C\|f\|_{H^s(M)}.
$$
For $s\in (0, d/2)$, this  amounts equivalently to the
finiteness of its $(d-2s)$-Riesz energy:
$$
I_{d-2s}(\Upsilon):=
\iint_{M\times M}
\frac{d\Upsilon(x)\,d\Upsilon(y)}
{d_M(x,y)^{d-2s}}<\infty.
$$
In particular, the lower Hausdorff dimension of $\Upsilon$
satisfies
$$
\underline{\dim}_H(\Upsilon):=
\inf_{\Upsilon(A)>0}\dim_H(A)\geq d-2s,
$$
and $\Upsilon$ has lower local dimension at least
$d-2s$, in the sense that
\[
\liminf_{r\to 0^+}
\frac{\log\Upsilon(B_r(x))}{\log r}
\geq d-2s
\qquad\text{for $\Upsilon$-almost every $x\in M$.}
\]
\end{rema}

\begin{rema}[Dimension loss in the dissipative regime]
\label{rem dimension loss}
In a separate note in preparation, we construct a compact set
$\cK\subset\Diff^\infty(\T^3)$ and a family
$\mu_t\in\cP(\cK)$ converging to $\mu_0$ as $t\to0$, such that
$\mu_0$ is volume-preserving and satisfies the hypotheses of
Theorem \ref{thm mainDissipative} with $J=\{0,2,3\}$ and
arbitrarily small pinching parameter. For every sufficiently small
$t\ne0$, the stationary measure $\Upsilon_{\mu_t}$ has full
topological support but gives full mass to a Borel set of
Hausdorff dimension strictly smaller than $3$.
Thus the conclusion of Theorem \ref{thm mainDissipative}
cannot in general be strengthened to full Hausdorff dimension,
let alone absolute continuity.
In these examples, $\mu_0^{-1}$ is expanding but not coexpanding.
\end{rema}

\begin{rema}[$H^s$-regularity]
If $\mu_{0}^{-1}$ is coexpanding (see Definition \ref{defi coexpanding}), then Lemma \ref{Hs-stat} implies that the measure $\Upsilon_{\mu}$ in Theorems \ref{main thm 2Dissipative}, \ref{thm mainDissipative} is  absolutely continuous, with density in the Sobolev space $H^s(M)$ for some $s=s(M, \mu_{0}, \cK)>0$. This condition of coexpansion  holds automatically in the context of our application in Theorem \ref{examplesI,II,III}.
\end{rema}

\begin{rema}[Volume-preserving case]
If $\mu$ is supported on   $\Diff^2(M, \vol)$,    then   whenever the above theorems apply,   we clearly have $$\Upsilon_{\mu}=\vol.$$ 
\end{rema}

\begin{rema}[Compact support] \label{remcompsup}
We work under the assumption that $\supp(\mu_0)$ and $\cK$ are compact for convenience of terminology. In all theorems and arguments  of the paper, the compactness assumption can be weakened to requiring that  $\supp(\mu_0),\cK$ and their inverse sets $\supp(\mu^{-1}_0), \cK^{-1}$ are bounded in $\Diff^2(M)$.
\end{rema}

In \cite{BEFRH}, Brown-Eskin-Filip-Rodriguez Hertz study (among other things) stationary measures for smooth volume-preserving random walks. They show that any atom-free ergodic $\mu$-stationary measure $\sigma$ is in fact $\Gamma_{\mu}$-invariant and absolutely continuous 
if the walk is expanding  on the exterior  tangent bundle $\bigwedge^*TM$ \cite[Def 1.1.3]{BEFRH} and the measure $\sigma$ does not have  zero as a Lyapunov exponent ($\sigma$ is ``hyperbolic''). 
Our conditions in Theorems  \ref{main thm 2Dissipative},  \ref{thm mainDissipative} are stronger\footnote{As can be deduced from Proposition \ref{lem AllGapsImpliesNonConcentration}, similarly to the proof of Lemma \ref{lem finetocoexpanding2}.} than the expansion condition in \cite{BEFRH}.
Nonetheless a strength of Theorems \ref{main thm 2Dissipative}, \ref{thm mainDissipative} is that they allow us to classify stationary measures without restriction to hyperbolic ones.   
Note non-hyperbolic measures occur very naturally, e.g. if $\mu$ is symmetric (i.e. $\mu=\mu^{-1}$) and $M$ is odd-dimensional, then every ergodic $\Gamma_{\mu}$-invariant measure has zero in its Lyapunov spectrum. One can also construct non-invariant stationary examples. 
\bigskip

Most of the paper is dedicated to the proof  of Theorem  \ref{thm mainDissipative}. We will  explain how to derive Theorems  \ref{examplesI,II,III}, \ref{thm Sphere},  \ref{main thm 2Dissipative} in later sections.


\subsection{Structure of the paper}

In Section \ref{sec Strategyoftheproof}, we  reduce Theorem \ref{thm mainDissipative} to establishing three key results  (namely Propositions \ref{prop step I},  \ref{prop step II}, \ref{prop step III}), corresponding to the phases I, II, III mentioned in the introductory paragraph. 
Section \ref{sec Basic properties} is dedicated to preliminaries, exploring  consequences of  Definitions \ref{def gap} and \ref{def pinch}.
In  Section \ref{sec proof prop step I}, we prove phase I (Proposition \ref{prop step I}). Sections \ref{sec BHMT}, \ref{sec proof prop step II} establish phase II (Proposition \ref{prop step II}).  Section  \ref{sec proof prop step III} validates phase III (Proposition \ref{prop step III}).
The proofs of Theorem \ref{main thm 2Dissipative}, as well as 
Theorems \ref{examplesI,II,III}, \ref{thm Sphere} from \SectionSymbol \ref{Sec-examples} are given in Section \ref{sec Proofs123}. The paper ends with a short appendix  \ref{app-LDP} on large deviation estimates for Markov chains, which is used on several occasions throughout the paper. 

\bigskip
 
\noindent{\bf Acknowledgements.} 
The authors are grateful to Amir Algom for helpful bibliographical comments. 

\bigskip
\noindent{\bf AI disclosure.} 
All the results and proofs in this paper were discovered and written without the use of any AI tool, see the arXiv v1 which appeared in May 2026. At a later stage, we asked AI for any suggestion to polish the paper, which led to a few minor corrections regarding the presentation, and the precision of certain assumptions/conclusions.

\bigskip

\noindent{\bf Subsequent work.} After the first version of the present paper appeared, Brown--Eskin--Filip--Rodriguez Hertz \cite{BEFRH2} established new criteria for absolute continuity of invariant measures of random walks by diffeomorphisms, allowing zero Lyapunov exponents. Their results assume, in particular, absolute continuity of conditional measures along suitable partial stable manifolds, and also apply to noncompact manifolds.

\section{Notations and strategy}

 \subsection{Notations and conventions} We discuss the notational conventions in this paper.
We let $\R_{\geq0}$, $\R_{>0}$, $\N$, $\N_{\geq1}$ denote respectively the non-negative reals, positive reals, non-negative integers, positive integers.

 \bigskip

 \noindent{$\bullet$ \bf Vector spaces.} 
 We introduce some notations related to a finite-dimensional real vector space $V$ equipped with a norm induced by an inner product.

Given $r > 0$, we denote by $B^{V}_{r}$ the open ball of radius $r$ in $V$ centered at the origin.

We denote by $\S(V)$ the set of unit vectors in $V$.
Given a subspace $W \subset V$, we denote by $W^{\perp}$ the orthogonal complement of $W$ in $V$, and  denote by $P_{W}: V \to W$ the orthogonal projection to $W$.

 Let $V_0, V_1$ be two nonzero subspaces of $V$ having complementary dimensions. 
We denote by $\measuredangle(V_0, V_1)$  the angle between $V_0$ and $V_1$. More precisely, we let 
$$\measuredangle(V_0, V_1) : = \inf_{v_0 \in \S(V_0), v_1 \in \S(V_1)} \measuredangle(v_0, v_1).$$ 
This is equivalent to the angle function introduced in \cite[\SectionSymbol 1.3]{BH2}, up to a bounded multiplicative 
constant and up to a power depending on $\dim V$.

Given $k \in \{0, \cdots, \dim(V)\}$, we denote by $\Gr(V, k)$ the Grassmannian which parametrizes the set of $k$-dimensional subspaces of $V$   (in particular $\Gr(V, 0)$ is a singleton).  We write $\Gr(n, m) = \Gr(\R^n, m)$ for all integers $0 \leq m \leq n$.
 We fix a certain  Riemannian metric on $\Gr(n,m)$ defined in a canonical\footnote{We ask the metric to be canonical so that we do not need to mention how parameters depend on a choice of metric throughout the paper.} way, and write 
 ${\rm dist}$ for the induced distance on $\Gr(n,m)$.  As explained in \cite{BH2}, ${\rm dist}$ is equivalent to the $\OO(\R^n)$-invariant metric   introduced in  \cite[\SectionSymbol 1.3]{BH2}.
 \bigskip

\noindent{$\bullet$ \bf Manifold.}
As in \SectionSymbol \ref{Sec-Mainresults}, we let $M$ be a smooth, compact, connected Riemannian manifold without boundary, of dimension $d$, equipped with the normalized volume measure $\vol$ induced by the metric. We denote by $d_M$ the Riemannian distance on $M$, and by $B_r(x)$ the open ball of radius $r>0$ centered at $x\in M$.

For every $s\in\R$, we denote by $H^s(M)$ the $L^2$-based Sobolev space on $M$ of order $s$. We write $\Lip(M)$ for the space of Lipschitz continuous functions on $M$.

Given an integer $\ell\geq 1$, we denote by $C^\ell(M,M)$ the space of $C^\ell$ maps from $M$ to itself, equipped with the usual $C^\ell$ topology (see \cite[Chapter 2, Section 1]{H2}). We denote by $\Diff^\ell(M)$ the space of $C^\ell$-diffeomorphisms of $M$, which is an open subset of $C^\ell(M,M)$. Both spaces are Polish. We also recall the notations $G:=\Diff^1(M)$ and $G_{\vol}:=\Diff^1(M,\vol)$.

Fix a finite collection of smooth coordinate charts
$(V_i,\varphi_i)$ and open subsets $U_i\Subset V_i$
covering $M$. For $f\in C^\ell(M,M)$, we denote by
$\|f\|_{C^\ell}$ the supremum of the norms of the derivatives of
orders $1,\ldots,\ell$ of its coordinate expressions
$\varphi_j\circ f\circ\varphi_i^{-1}$, evaluated at $\varphi_i(x)$, for
points $x\in U_i$ such that
$f(x)\in U_j$. The boundedness of this
quantity on a family of maps is independent of
the chosen charts and subsets.

Recall that, by the Arzel\`a--Ascoli theorem, a subset $\cK\subset\Diff^\ell(M)$ is relatively compact if and only if
\begin{itemize}
\item[(a)] $\sup_{f\in\cK}\|f\|_{C^\ell}<\infty$;
\item[(b)] the derivatives of order $\ell$ of the maps $f\in\cK$ are uniformly equicontinuous in the local coordinates $(U_{i}, \varphi_{i|U_{i}})$;
\item[(c)] $\displaystyle\sup_{f\in\cK}\|Df^{-1}\|_{L^\infty}<\infty$.
\end{itemize}
In particular, $\cK$ is compact if and only if it is closed in $\Diff^\ell(M)$ and satisfies (a)--(c).

\bigskip
\noindent{$\bullet$ \bf Implicit constants.} 
(1) Given $A\in \R$, $B>0$, and a parameter $q$ (potentially multidimensional), we write $A = O_{q}(B)$ or $A \lesssim_{q} B$  to mean that there is a positive constant $C$ depending  on $M$ and $q$ only such that $|A| \leq C B$.
We also say that a statement involving $A,B$  is valid under the condition $A\lll_{q} B$ if it holds provided  $A\leq \eps B$ where $\eps>0$ is a small enough constant depending on $M$ and $q$ only. 
When the above $C,\eps$ only depend on $M$, we do not write any subscript. We insist that the notations $O(\cdot)$, $\lesssim$, $\lll$ \emph{allow implicit dependencies on $M$. }

\bigskip
(2) Given a parameter $q$, we write $C=C(q)$ to indicate that a certain constant $C$ depends on $q$ only. Contrary to the previous set of notations, we ask to \emph{indicate dependence on $M$ here}. 

For example, given parameters $\eta, \kappa>0$, a statement of the form
 $$\text{``\emph{If $\eta \leq c_{0}(d)\kappa$, then the property $P$ holds}''}$$
means that $P$ is valid for all values of $\eta,\kappa$ such that $\eta/\kappa$ is smaller than a certain constant $c_{0}$ depending on $d$. 
 
 It will be convenient to extend this notation to weak-$\star$ neighborhoods of the measure $\mu_{0}$. We will write
 $$\text{``\emph{If $\cU \subseteq \cU_{0}(q)$, then the property $P$ holds}''}$$
  to say that $P$ is valid provided that $\cU$ is included in a certain weak-$\star$  neighborhood of $\mu_{0}$ which only depends on the parameter $q$.

 \bigskip

  \subsection{Strategy of  proof} \label{sec Strategyoftheproof}

We reduce the proof of Theorem \ref{thm mainDissipative} to three steps, described as Phases I, II, and III below. We place ourselves in the setting of Theorem \ref{thm mainDissipative}, which we repeat for clarity. 

     	\begin{setting}\label{Setting-main-result}
	We consider    $\kappa, \check{\kappa}, \eta> 0$ with $\check{\kappa}\in (0, \kappa)$, and  $J = \{  0 = d_0 < d_1 <   \cdots < d_m < d_{m+1} = d \}$  a subset of $ \{0, \cdots, d \}$ with at least $3$ elements (i.e. $m\geq1$). 
	We let $\mu_0$ be a compactly supported Borel probability measure on   $\Diff^{2}(M, \vol)$   such that: 
	
	(1) $\mu_0$  has a $(\kappa, J)$-gap and is $(\eta, J)$-pinched, 
	
	(2) $\mu_{0}^{-1}$ is $(-\check{\kappa}/d)$-expanding, 
	
	 (3) $T_{\mu_0}$ is totally ergodic in $L^2(M, \vol)$.
	 
	We let   $\cK  \subseteq \Diff^{2}(M)$   be a   compact   subset containing $\supp(\mu_0)$, and 
	$\cU$ a  neighborhood  of $\mu_0$ in $\cP(\cK)$ with respect to the weak-$\star$ topology.  Finally, we consider $\mu\in \cU$. 
	\end{setting}

\begin{rema}In Theorem \ref{thm mainDissipative}, we only assume $\mu_{0}^{-1}$ to be $(-\kappa/d)$-expanding. Noting this remains true if we perturb the value of $\kappa$, the above setting (with $\check{\kappa}$) is thus equivalent to that of Theorem \ref{thm mainDissipative}, as needed. Introducing $\check{\kappa}$ will be helpful to quantify  some of the statements below.
\end{rema}



  \bigskip
  \noindent{\bf Phase I}. We show that given any starting point $x\in M$ with infinite $\Gamma_{\mu}$-orbit, any scale $\rho_{0}>0$,  the $n$-step distribution $\mu^{*n}*\delta_{x}$ of the $\mu$-random walk has positive dimension above the scale $\rho_{0}$, provided $n$ is large enough depending on $\mu, x,\rho_{0}$.

   \begin{prop}[Positive dimension] \label{prop step I} 
  Assume $\eta\leq c_{0}(d)(\kappa-\check{\kappa})$ and $\cU\subseteq \cU_{0}(M, \mu_{0}, \cK)$. 
  There exists $\varkappa_{0}=\varkappa_{0}( M,  \mu_{0}, \cK)>0$ such that for every $x \in M$ with infinite $\Gamma_{\mu}$-orbit, every $\rho_0>0$, and $n \ggg_{\mu, x, \rho_{0}} 1$, we have
\aryst
  \sup_{y \in M} \mu^{* n} * \delta_{x}(B_{\rho}(y)) \leq 2 \rho^{\varkappa_{0}}, \quad \forall \rho \geq \rho_0.
\earyst
  \end{prop}
  
    The proof of  Proposition \ref{prop step I} is given in Section \ref{sec proof prop step I} and follows the argument of  \cite[Proposition 5.1]{BH1}. 
  
  \begin{rema}
  In this step, we only require $C^1$-regularity for elements in the support of $\mu$, that is, $\supp(\mu) \subset \Diff^1(M)$ is sufficient. Moreover, we only use that the $\mu$-random walk on $M$ is expanding  forward and not too contracting backward, see Proposition \ref{prop step I gen} for a precise general statement. 
  \end{rema}

  \begin{rema}
  The proof does not allow us to track down how the implicit constant in the lower bound $n \ggg_{\mu, x, \rho_{0}} 1$ depends on $x$ or $\rho_{0}$. This is in fact the only reason why our results are not quantitative. If we could make explicit the lower bound required on $n$ depending on $x$, $\rho_{0}$, then we could obtain a rate for the convergence $\mu^{*n}*\delta_{x} \rightarrow \Upsilon_\mu$ announced in Theorem \ref{thm mainDissipative}.
  

  \end{rema}

In the course of proving Proposition \ref{prop step I}, we also obtain the following useful lemma.
\begin{lemma}\label{cor countablefiniteorbits}
	 Under the assumptions of Proposition \ref{prop step I}, for every $N \in \N$, the number of  $\Gamma_{\mu}$-orbits of cardinality at most $N$ is finite.
\end{lemma}
    
 \bigskip
   
  \noindent{\bf Phase II}.  We bootstrap the dimensional estimate obtained in Phase I, until reaching a dimension arbitrarily close to $d=\dim M$. To do so we show that acting by $\mu$-convolution on a given measure $\nu$ on $M$ improves its dimensional properties, up to ignoring a small part of $\nu$.

  We conveniently use the next terminology, extracted from \cite[Definition 4.3]{BH1}. 
  
  \begin{defi} 
  	Let $\alpha \in [0, 1]$, $\tau \geq 0$,  $I \subset \R_{ > 0}$. Let $\nu$ be a Borel measure on $M$. 
  	We say that $\nu$ is { \em $(\alpha, \cB_{I}, \tau)$-robust} if one can decompose $\nu$ into a sum of Borel measures $\nu = \nu' + \nu''$ such that 
		\aryst
  \forall \rho\in I,	\forall x \in M,\,\, 	\,\,\nu'( B_{\rho}(x) ) \leq \rho^{d\alpha} \ \mbox{ and } \  \nu''(M) \leq \tau.
  	\earyst
In the case where $I$ is a singleton $I=\{\rho\}$, we simply write that $\nu$ is $(\alpha, \cB_{\rho}, \tau)$-robust. 
  \end{defi}

The key result underlying  Phase II is the following.  Recall the parameter $\eta$ rules the pinching condition on Lyapunov exponents associated to $\mu$. 

  	   \begin{prop}[Dimensional  increment] \label{prop step II-single} 
		 	Let $\varkappa,\rho,  \eps \in (0,1/10)$,   $\tau \geq 0$, $\alpha \in [\varkappa, 1 - \varkappa]$. Let $\nu$ be a Borel measure on $M$ of mass at most $1$ and that is $(\alpha, \, \cB_{[\rho, \rho^{1/4}]},\, \tau)$-robust.      If $\cU\subseteq \cU_{0}(M, \mu_{0}, \cK, \varkappa)$,    $\eta, \eps, \rho\lll_{\cK,\mu_{0}, \varkappa}1$, then for some integer $n_{\rho} =n_{\rho}(M, \cK, \rho)\in [1, | \log \rho |]$, we have 
		$$\text{$\mu^{*n_{\rho}} * \nu$ is $(\alpha + \eps,  \,\cB_{\rho^{1/2}},\, \tau + \rho^{\eps})$-robust.}$$
		 \end{prop}

	The proof of  Proposition \ref{prop step II-single} is given in Section \ref{sec proof prop step II}. It builds on B\'enard-He's multislicing method \cite{BH1, BH2} which is recalled and developed further in Section \ref{sec BHMT}.
		 
	 \begin{rema}   
	 	 In Phase II, we only require  that the $\mu$-random walk on $M$ satisfies the gap and pinching properties required in assumption (1), and is almost volume-preserving, see  Proposition \ref{prop step II-single-gen} for a precise general statement. 
		 \end{rema}
		 
By combining Proposition \ref{prop step I} and an iterative application of Proposition \ref{prop step II-single},  we obtain that for a starting point $x\in M$ with infinite $\Gamma_{\mu}$-orbit, the distribution $\mu^{*n}*\delta_{x}$ ultimately reaches high dimension.

   \begin{prop}[Large dimension] \label{prop step II}
   Let $\varkappa, \eps, \rho\in (0,1/10)$, $n\in \N$, $x\in M$ with infinite $\Gamma_{\mu}$-orbit. If $\cU\subseteq \cU_{0}(M, \mu_{0}, \cK, \varkappa)$,   $\eta, \eps, \rho \lll_{\cK, \mu_{0},\varkappa} 1$,   and $n \ggg_{\mu, x, \varkappa, \rho} 1$, then 
$$ \text{$\mu^{*n}*\delta_{x}$ is $(1-\varkappa, \,\cB_{\rho}, \,\rho^\eps)$-robust.}$$
  \end{prop} 

Let us take a moment to check Proposition \ref{prop step II}, admitting the previous key steps. For this, we can mimic the proof of \cite[\SectionSymbol 4.5]{BHZ}.

\begin{proof}[Proof that Propositions \ref{prop step I}, \ref{prop step II-single} $\implies$ Proposition \ref{prop step II} ]

 We may assume that  $\varkappa < \varkappa_{0}/2d$  where $\varkappa_{0}$ is the dimensional constant in Proposition \ref{prop step I}. Then Proposition \ref{prop step I} guarantees that for any $C>1$, for every $\rho\lll_{C,\varkappa_0}1$ and $n\geq n_{\ini}=n_{\ini}(\mu, x, \rho^C)\in \N$, the measure
\[\text{{$\mu^{*n}*\delta_{x}$} is $(\varkappa, \cB_{[\rho^C, \,\rho^{1/C}]}, 0)$-robust}. \]

By Proposition \ref{prop step II-single},  there is some small constant $\eps=\eps(M, \cK, \mu_{0}, \varkappa)>0$, such that, up to imposing from the start $\eta\lll_{\cK, \mu_{0}, \varkappa}1$, we have for {all} $\rho\lll_{\cK, \mu_{0}, \varkappa, C}1$, all $n\geq C  |\log \rho| +n_{\ini}$ and $r\in [\rho^{C}, \rho^{4/C}]$,
\[\text{{$\mu^{*n}*\delta_x$} is $(\varkappa+2\eps, \cB_{r^{1/2}}, r^{\eps})$-robust}. \]
These estimates for single scales can be combined using \cite[Lemma 4.5]{BH1} to get under the same conditions:
\[\text{{$\mu^{*n}*\delta_x$} is $(\varkappa+\eps, \cB_{[\rho^{C/2}, \rho^{2/C}]}, O_{\eps,C}(\rho^{\eps/4C}))$-robust}. \]

The argument in the last paragraph can be applied iteratively, adding at the  $k$-th step  the value $+\eps$ to the dimension provided that the latter is not yet above $1-\varkappa$. This dimensional increment is guaranteed for the interval of scales  $[\rho^{C/2^k}, \rho^{2^k/C}]$, and  a trash component of the form $O_{\eps,C,k}(\rho^{\eps/4C})$.
Noting the value of $\eps$ is independent of $k$, we reach dimension $1-\varkappa$ in at most  $\lceil \eps^{-1} \rceil$ steps. Iterating this long is allowed provided that $C$ is chosen large enough depending on $\eps$, and up to taking $C$ even larger, the parameter $\rho$ belongs to the range of scales for which we have obtained $(1 - \varkappa)$-robustness when the iteration ends.
This concludes the proof.
\end{proof}

\bigskip

  \noindent{\bf Phase III}.  
  At the end of Phase II, we have obtained a measure $\mu^{*n}*\delta_{x}$ which has dimension very close to $d=\dim M$ above a small scale. In this last phase, we apply a few more convolutions by $\mu$ to convert this property into (quantitative) equidistribution toward the only  atom-free  $\mu$-stationary measure. The argument relies on the spectral gap of the convolution operator by $\mu$ on a suitable Sobolev space.

We quantify how close  two probability measures on $M$ are by  means of the Wasserstein distance  $\cW_1$. 
Given $f :M\rightarrow \R$, we set
$$\|f\|_{\Lip}:=  \|f\|_{L^{\infty}} + \sup_{x\neq y}\frac{|f(x)-f(y)|}{d_{M}(x,y)}. $$
The Wasserstein distance between two probability measures $\nu_1, \nu_2$ on $M$ is  defined by
\aryst
\cW_1(\nu_1, \nu_2) := \frac{1}{2}\sup_{f\,:\, \|f\|_{\Lip}\leq 1} |\nu_{1}(f)-\nu_{2}(f)|.
\earyst
     
      The main result in this step is the following.

   \begin{prop} \label{prop step III} 
    Assume that  $ \eta\leq c_{0}(d)\kappa$,  and $\cU\subseteq \cU_{0}(M, \mu_{0},\cK)$. 
    Then there exists a   $\mu$-stationary probability measure $\Upsilon_{\mu}$ on $M$ such that
    the following holds. 
    
    Let $\nu$  be a  $(1 - \eps, \cB_{\rho}, \tau)$-robust probability measure on $M$ for some $\eps, \rho\lll_{\mu_{0}, \cK}1$, and  $\tau \in \R^+$. 
  There exists $c = c(M, \cK) > 0$ such that, for every integer $n \in [c  | \log \rho|, 2c | \log \rho | ]$, we have  
  	\aryst
  	\cW_1(\mu^{* n} * \nu, \Upsilon_{\mu} ) < \rho^{\eps} + \tau.
  	\earyst
Moreover, given any $s>0$, we further have $\Upsilon_{\mu}\in H^{-s}(M)$ provided $\cU\subseteq \cU_{0}(M, \mu_{0},\cK,s)$. 
      \end{prop}

The proof of Proposition \ref{prop step III}  is  given in Section \ref{sec proof prop step III}. The spectral information we need is extracted from \cite{DD, T3, DTZ, DDZ}. The deduction of equidistribution follows an argument attributed to Venkatesh \cite{Venkatesh}, see also \cite{KK}.

 \begin{rema}   To establish Proposition \ref{prop step III}, we only need the properties that $\mu_{0}$ is coexpanding and $T_{\mu_{0}}$ is totally ergodic in $L^2(M, \vol)$,  see  Proposition  \ref{prop step III-gen} for a precise general statement. 
\end{rema}

  We may now conclude the proof of Theorem \ref{thm mainDissipative}.
  
    \begin{proof}[Proof of Theorem \ref{thm mainDissipative}]
    By combining     Propositions \ref{prop step II} and  \ref{prop step III}, we see that if $\eta\lll_{\mu_{0}, \cK}1$ and  $\cU\subseteq \cU_{0}(M,\mu_{0},\cK, s)$, then there exists a measure $\Upsilon_{\mu} \in H^{-s}(M)$ satisfying the equidistribution statements in Theorem \ref{thm mainDissipative}.
    
We now show that $\Upsilon_{\mu}$ is the unique atom-free $\mu$-stationary probability measure on $M$. 
    Let $\Upsilon_{0}$ be an arbitrary atom-free  $\mu$-stationary probability measure. By Lemma \ref{cor countablefiniteorbits},  there are only countably many finite $\Gamma_{\mu}$-orbits. Thus, almost every ergodic component of $\Upsilon_{0}$ is atom-free as well. Without loss of generality, let us assume that $\Upsilon_{0}$ itself is ergodic.
    Then for a $\Upsilon_{0}$-typical $x \in M$, we have $\frac{1}{n}\sum_{k = 0}^{n-1} \mu^{*k} * \delta_{x}$ converges to $\Upsilon_0$. By the statements in Theorem \ref{thm mainDissipative} proved above, we deduce that the $\Gamma_{\mu}$-orbit of $x$ is infinite, and that $\mu^{*k} * \delta_{x}$ converges to $\Upsilon_{\mu}$ as $k$ tends to infinity. Consequently, $\Upsilon_0 = \Upsilon_{\mu}$ is the only ergodic atom-free $\mu$-stationary measure.

%
  %
	  \end{proof}

		\section{Basic properties} \label{sec Basic properties}
		
	
	In this section, we study preliminary properties of  the random walk on $M$. We place ourselves in the following framework.
	
	\begin{setting}\label{Setting-Dimpos}
	We let $\mu\in  \cP_{c}(G)$ denote a compactly supported probability measure on $G$. We consider $\cK\subseteq G$ a   compact   set containing $\supp( \mu)$. 
	\end{setting}
	It will be convenient to use the notation
    \ary  \label{eq D1}
      \maxder := \sup_{g \in \supp(\mu)} \max( \| Dg \|_{L^{\infty}}, \| D(g^{-1}) \|_{L^{\infty}} ).
    \eary

	\subsection{Gaps and pinching of singular values} \label{sec-gap-pinch}
	We investigate how  gap/pinching assumptions on the measure $\mu$ (as defined in \SectionSymbol \ref{Sec-Mainresults}) affect the behavior of the singular values of a typical element $g\in G$ sampled by $\mu^{*n}$.

	\bigskip
			Given $x \in M$ and $g \in G$, we fix a \emph{Cartan decomposition} of  $Dg(x)$. More precisely, we write
	\begin{align*}
	Dg(x) = R'(x, g)   \diag(e^{\lambda_{1}(x, g)}, \cdots, e^{\lambda_{d}(x, g)} )  R(x, g) 
	\end{align*}
	for some reals $\lambda_{1}(x, g)  \leq \cdots \leq   \lambda_{d}(x, g)$, and some linear isometries $R'(x, g):  \R^d \to T_{g(x)}M$ and $R(x, g) :  T_x M \to \R^d$. We observe that the vector of exponents $(\lambda_{1}(x, g), \cdots, \lambda_{d}(x, g) )$ is unique:  it lists the logarithms of the singular values of $Dg(x)$. The maps $R'(x, g), R(x, g)$ are not necessarily unique.  
	

	The next lemma provides a variational characterization of singular values of $D g(x)$.

	\begin{lemma} \label{lem lambdab}
		For all  $b  \in \{1, \cdots, d \}$, and all $g\in G$,  $x \in M$, we have 
		
		
		\begin{align*}
		e^{\lambda_{b}(x, g )}  
		&\,= \, \inf_{V \in \Gr(T_x M, d -b ) } \sup_{u \in \S(V^{\perp}) } \|  P_{Dg(x, V)^{\perp}}( Dg(x, u) ) \| \\
		&\,= \, \sup_{V' \in \Gr(T_x M, d - b + 1)}  \inf_{v' \in \S(V')} \| Dg(x, v') \|.
		\end{align*}

	\end{lemma}
	
	
	\begin{proof} We show the first equality; the proof of the second one is similar. As  isometries do not affect the norm and commute with taking orthogonal complements, it is enough to show the result for the diagonal part $A\in \GL(\R^d)$ given by $A=\diag(e^{\lambda_{1}(x, g)}, \cdots, e^{\lambda_{d}(x, g)})$ instead of $Dg(x)$, and subspaces varying in $\Gr(\R^d)$ instead of $\Gr(T_{x}M)$. 
Set $W= \{0\}^{b-1}\times \R^{d-b+1}$, and note that  the action $A  \acts W$ satisfies $\|A w\|\geq e^{\lambda_{b}(x,g)}\|w\|$ for every $w \in W$. Let $V\in \Gr(d, d-b)$. Note that for dimensional reasons, there exists $w\in   A   W\cap (AV)^\perp  \smallsetminus \{0\}$.  Then  setting $u=P_{V^\perp}A^{-1}w$, and noting $P_{(AV)^\perp} AV=0$, we find
$$\|P_{(AV)^\perp} Au\|= \| P_{(AV)^\perp} A A^{-1}w\|=\|w\|.$$
Normalizing so that $\|u\|=1$, we have $\|w\|\geq e^{\lambda_{b}(x,g)}$ and the desired inequality follows. 
Equality is achieved by taking $V$   to be the span of the last $d-b$ coordinate vectors.  
\end{proof}
	
		We now use Lemma \ref{lem lambdab} to relate the gap $\lambda_{b}(x, g) -  \lambda_{b'}(x, g)$ between the $b$-th and $b'$-th smallest logarithmic singular values of $D g(x)$, to the expansion/contraction data that appear in the gap or pinching condition from \SectionSymbol \ref{Sec-Mainresults}.

	\begin{cor} \label{lem lambdab1lambdab0+1} 
		For all integers $b,b'\in \{1, \dots,d\}$, all $g\in G$, $x\in M$, we have
		\aryst
		&& e^{ \lambda_{b}(x, g)  - \lambda_{b'}(x, g)  } \\
		 &=&
		    \inf_{\substack{V \in \Gr(T_xM, \,d - b) \\ V' \in \Gr(T_xM, \,d - b'+1)} }  \left(  \sup_{u \in \S(  V^{\perp})}  \| P_{   Dg(x, V)^{\perp}}(  Dg(x, u) ) \|  /  \inf_{v' \in \S(V')}  \| Dg(x, v') \| \right).
		\earyst
	\end{cor}

\begin{proof}  
It is a direct application of Lemma  \ref{lem lambdab}.
		\end{proof}



In order to make use of Corollary \ref{lem lambdab1lambdab0+1}, we now present how a gap/pinching condition can be converted into a  moment bound for quotients of derivatives.

 
	We are now in a position to conclude the paragraph, by showing that a $b$-gap condition on $\mu$ yields a linear separation between $\lambda_{b}(x,g)$ and  $\lambda_{b+1}(x,g)$ as $x\in M$ and $g\sim \mu^{*n}$, while  on the other hand, a $(b_{0},b_{1})$-pinching condition is reflected by the proximity of  $\lambda_{b_{0}+1}(x,g)$ and $\lambda_{b_{1}}(x,g)$.

\begin{lemma} \label{cor GapWithHighProba}  

			 The following is true.
			 
			 \enmt
			 \item
			 Let  $b\in  \{1, \cdots, d -1 \}$, $\kappa >\kappa'> c > 0$, $n, n_{0}\in \N$, $x\in M$. If  $\mu$ has a $(n_{0}, \kappa, b)$-gap,   
			and  $c \lll_{\cK, n_{0}, \kappa, \kappa'} 1$, then
		   \aryst 
		   \mu^{*n} \Big (  g \in G \mid    \lambda_{b}(x, g)  + n \kappa'  <   \lambda_{b + 1}(x, g)    \Big  ) > 1 - 2e^{- c n }.
		   \earyst
		   \item
		  Let  $ b_{0}<b_{1}$ be integers in $\{0, \dots, d\}$, let $\eta'>\eta> c > 0$, $n, n_{0} \in \N$, $x\in M$.  
		   If $\mu$ is $(n_{0}, \eta, b_{0}, b_{1})$-pinched,  and $c \lll_{\cK, n_{0}, \eta, \eta'} 1$, then 		   		   \aryst 
		   \mu^{*n} \Big (  g \in G  \mid     
	   \lambda_{b_{0} + 1}(x, g) +   n \eta'   >  \lambda_{b_{1}}(x, g) \Big  ) > 1 - 2e^{-c n }.
		   \earyst
		   \eenmt
		\end{lemma}
		
		\begin{proof}
	We establish (1). The proof of (2) is similar. 

	Without loss of generality, we may assume that $\kappa - \kappa' \in (0, 1/2)$.

	Let $V\in \Gr(T_{x}M, d-b)$. By Corollary \ref{lem lambdab1lambdab0+1}, we have for every $g\in G$, 
	$$\lambda_{b}(x, g)  - \lambda_{b+1}(x, g) \leq  \log\sup_{u \in \S(  V^{\perp})}  \| P_{   Dg(x, V)^{\perp}}(  Dg(x, u) ) \|  - \log  \inf_{v \in \S(V)}  \| Dg(x, v) \|.$$
	Let $\psi(g,V)$ be the right-hand side of the above inequality. Note that $\psi:G\times \Gr(TM, d-b)\rightarrow \R$  satisfies for any $g_{1},g_{2}\in G$, $W\in \Gr(TM, d-b)$, 
	$$\psi(g_{2}g_{1},W) \leq \psi(g_{2}, g_{1}W)+ \psi(g_{1},W).$$
Moreover $\psi$ is bounded on $\supp (\mu^{*n_{0}}) \times \Gr(TM, d-b)$, and the  	$(n_{0},\kappa, b)$-gap assumption means that $$\sup_{V\in \Gr(TM, d-b)}\EV_{g\sim \mu^{*n_{0}}}(\psi(g,V))<-n_{0}\kappa.$$
	We may then apply the general large deviation estimate formulated in Proposition \ref{LDP-MC} to  $E:= (\supp(\mu^{*n_0})\cup\{\mathrm{id}\}) \times \Gr(TM, d-b)$, $f:=\psi +n_{0}\kappa$ and the Markov chain on $E$ with the transition probabilities
	$P_{(h, W)} =  \mu^{*n_{0}} \otimes \delta_{hW}$.  We obtain for all $k\geq1$,
	   \aryst 
		   \mu^{*kn_{0}} \Big (  g \in G \mid    \lambda_{b}(x, g)  -   \lambda_{b + 1}(x, g) \geq - kn_{0} (\kappa+\kappa')/2    \Big  ) <  e^{- c k n_{0}}.
		   \earyst
	where $c=c(M, \cK, n_{0}, \kappa, \kappa')>0$. Noting that for any $j\in \{1, \dots, d\}$, $g,h\in G$, $x\in M$, we have
	$$|\lambda_{j}(x, hg)- \lambda_{j}(x,g)| \leq \log\max(\|Dh(gx)\|, \|(Dh(gx))^{-1}\|),$$
	Item  (1)  in the lemma  follows  for all $n$, up to reducing $c$.
		\end{proof}

	\begin{rema}\label{lem MomentBound} In the  proof of  Lemma \ref{cor GapWithHighProba} above, we only used the second part of Proposition \ref{LDP-MC}. By applying the first part, we obtain the following complementary information: $\exists \gamma=\gamma(M, \cK, n_{0}, \kappa) >0$, $\forall n\geq0$, $\forall x\in M$, $\forall V\in \Gr(T_{x}M, d-b)$,
		\aryst
			\int  \Big ( \sup_{u \in \S(V^{\perp})} \| P_{Dg(x, V)^{\perp}}(  Dg(x, u) ) \| /  \inf_{v \in \S(V)} \| Dg(x, v) \| \Big )^{\gamma} d\mu^{* n}(g) \,\leq \,2 e^{- n \gamma \kappa/2}.
			\earyst
This estimate will play a role in the proof of Proposition \ref{lem AllGapsImpliesNonConcentration}   below. 
	\end{rema}	
		
		 	
		\subsection{Non-concentration of the contracting filtration} \label{Sec-contrac-filt}

	Given $g\in G$, $x\in M$, we introduce a flag of $T_x M$ given by
			 \aryst
	\{0\}\subsetneq  W_{1}(x, g) \subsetneq \dots  \subsetneq W_{d}(x, g)= T_x M 	\earyst
	where
	\ary \label{eq Wixg}
	W_i(x, g) = R(x, g)^{-1}(\R^{i} \times \{0\}^{d-i}), \quad i \in \{1, \cdots, d\}.
	\eary
	Roughly speaking, $W_i(x, g)$ is a \lq\lq most contracted\rq\rq  \  $i$-dimensional subspace of $T_xM$ under $Dg(x)$. For $i<d$, if $\lambda_{i}(x, g)<\lambda_{i+1}(x, g)$, then $W_i(x, g)$  is uniquely defined (independently of the choices for $R'(x, g)$, $R(x, g)$). 
	
In the next proposition, we assume that $\mu$ has a $b$-gap and examine the statistics of $W_{b}(x,g)\subseteq T_{x}M$
as $g$ varies randomly according to $\mu^{*n}$. Note that by Lemma \ref{cor GapWithHighProba}, the random subspace $W_{b}(x,g)$ is uniquely defined for most $g\sim \mu^{*n}$. We show it satisfies a certain non-concentration property at scale $\rho\geq e^{-n}$.

		 \begin{prop} \label{lem AllGapsImpliesNonConcentration}    		 
		 Let $b \in \{1, \cdots, d-1\}$,  $\kappa, c, \rho>0$, $n, n_{0}\in \N$, $x \in M$ and  $V \in \Gr(T_x M, d-b)$. Assume that $\mu$ has a $(n_{0}, \kappa, b)$-gap, and the bounds  $c, \rho \lll_{\cK, n_{0}, \kappa} 1$ and $n \geq | \log \rho |$.
		 	   Then we have 
		 	\aryst
		 	\mu^{*n}(  g \in G \mid  \measuredangle( W_b(x, g), V) < \rho  ) < \rho^{c}.
		 	\earyst
		 \end{prop}

		 \begin{proof}

		 It is sufficient to show that with probability $1-\rho^{c}$, every vector $  v'   \in T_{x}M$ of the form 
		 $v'=v+u$ where $v\in \S(V)$,  $u\in V^\perp$, $\|u\|\leq 2\rho$, satisfies
		 $$ \|D g(x,v')\| > 2 e^{\lambda_{b}(x,g)}.$$
	      
		  We consider a sequence of integers $0=k_{0}\leq k_{1}\leq \dots\leq k_{i}\leq \dots$, 
          which will be specified later depending on   $M, \cK,\kappa, \rho$.

		 Given a sequence $\ug=(g_{i})_{i\geq1}\in G^{\N^*}$, we define $(h_{i})_{i\geq 1}$ by $h_{i}=g_{k_{i}}\dots g_{k_{i-1}+1}$, and   set $x_{i}=h_{i}\dots h_{1}x$,  $V_{i}=D(h_{i}\dots h_{1})(x)V$. 
Under the splittings $T_{x_{i-1}} M = V_{i-1} \oplus V_{i-1}^{\perp}$ and $ T_{x_{i}} M = V_{i} \oplus V_{i}^{\perp}$, we may write 
		 	\aryst
		 	D h_{i}(x_{i-1})= \begin{bmatrix}
		 		A_i & C_i \\ 0 & B_i
		 	\end{bmatrix}
		 	\earyst
where $A_i \in \cL(V_{i-1},V_{i})$, $B_i \in \cL(V^\perp_{i-1},V^\perp_{i})$, $C_i \in \cL(V^\perp_{i-1},V_{i})$.
Consider $l_{i}:=k_{i}-k_{i-1}$.  By applying  Remark \ref{lem MomentBound} (conditioned on $g_{1},\dots,g_{k_{i-1}}$),  we have  for some $\gamma>0$ depending on $M, \cK, n_{0}, \kappa$,
\aryst
\EV_{\ug\sim \mu^{\otimes \N^*}}[(\| A_i^{-1} \| \| B_i \|)^{\gamma}] \leq   2 e^{ - l_{i} \gamma \kappa/2 }.
\earyst 
The Markov inequality then implies 
$$ 
\mu^{\otimes \N^*}(\ug\,:\,  \| A_i^{-1} \| \| B_i \| > e^{ - l_{i} \kappa / 4}) \leq 2e^{ - l_{i} \gamma \kappa / 4}.
$$
Let $s\in (0, 1)$ be a parameter to be specified below depending on $M, \cK,\kappa$.  
Choosing $(k_{i})_{i\geq 1}$ so that   $l_{i}\geq l_{1}+  s\, (i-1) $,  the upper bound is summable over $i$ with total sum $O_{\gamma,\kappa, s}(e^{-  l_{1}\gamma \kappa/4 })$. By  subadditivity of measures, it follows that 
\ary\label{Esize}
\mu^{\otimes \N^*}(E)\geq 1- O_{\gamma, \kappa, s}(e^{- l_{1}\gamma \kappa/4 }) 
\eary
for 
$$E:=\{\ug\,:\, \forall i\geq 1,\,\, \| A_i^{-1} \| \| B_i \| \leq e^{ - l_{i} \kappa / 4}\}.$$

Let $\ug\in E$. We now check that any  vector $v'=v+u$ as described above has the desired growth under $g_{n}\dots g_{1}$. To do so, we will rely on an iterated application of the next lemma.
		 
\begin{lemma}\label{int-lem}
Let $i\geq 1$. Let $v'_{i-1}\in T_{x_{i-1}}M$ be of the form $v'_{i-1}=v_{i-1}+u_{i-1}$ where $(v_{i-1},u_{i-1})\in V_{i-1}\times V^\perp_{i-1}$ and  $\|v_{i-1}\|=1$, $\|u_{i-1}\| \leq 2^{-1}D_{1}^{-2l_{i}}$.

Then $D(h_{i})(x_{i-1},v'_{i-1})= R_{i} (v_{i}+u_{i})$ where $R_{i}\geq 2^{-1}\|A_{i}^{-1}\|^{-1}$ is a constant and $(v_{i},u_{i})\in V_{i}\times  V^\perp_{i}$ satisfy  $\|v_{i}\|=1$, $\|u_{i}\| \leq 2 e^{-l_{i} \kappa/4} \|u_{i-1}\| $.
\end{lemma}

\begin{proof}[Proof of Lemma \ref{int-lem}]  
We can write $Dh_i(x_{i-1},v'_{i-1})= (A_{i} v_{i-1}+C_{i}u_{i-1}) + B_{i} u_{i-1}$ where the first term (in brackets) belongs to $V_{i}$ and the second term belongs to $V^\perp_{i}$. We have 
\ary\label{growthV}
\|A_{i} v_{i-1}+C_{i}u_{i-1}\| \geq \|A_{i}^{-1}\|^{-1} -\|C_{i}\| \|u_{i-1}\|\geq 2^{-1}\|A_{i}^{-1}\|^{-1}
\eary 
where the first inequality uses $\|v_{i-1}\|=1$ and the second inequality uses $\|A^{-1}_{i}\|, \|C_{i}\|\leq D_{1}^{l_{i}}$ and $\|u_{i-1}\|\leq 2^{-1}D_{1}^{-2l_{i}}.$
On the other hand, recalling $\ug\in E$, we have
\ary\label{growthW}
\frac{\|B_{i} u_{i-1}\|}{2^{-1}\|A_{i}^{-1}\|^{-1}}\leq 2 \|A_{i}^{-1}\| \|B_{i} \| \|u_{i-1}\| \leq 2e^{-l_{i}\kappa/4} \|u_{i-1}\|.
\eary 
Equations \eqref{growthV}, \eqref{growthW} together justify Lemma  \ref{int-lem}.	
\end{proof}

We may now conclude the proof of Proposition \ref{lem AllGapsImpliesNonConcentration}.   Choose $k_1=\lfloor a |\log \rho| \rfloor$ where $a>0$. Assuming $a \lll_{\cK} 1$ and $\rho\lll_{ \kappa, a} 1$, we have  $2\rho< 2^{-1}D_{1}^{-2l_{1}}$ and $2e^{-l_{1}\kappa/4} < e^{-l_{1}\kappa/8}$.  
Choose $(k_{i})_{i\geq 2}$ depending on $k_1, s$ so that   $l_{i}=l_{1}+ \lceil s (i-1)\rceil$.  Note\footnote{Even though $l_{i}$ depends on $s$, this claim can be checked  by direct computation.} that by assuming $s\lll_{\cK, \kappa}1$, we have  $e^{-s^{-1} l_{i}\kappa/8}< 2^{-1}D_{1}^{-2(l_{i}+1)}$. The  condition on $\rho$  allows us to apply Lemma \ref{int-lem} to $v'_{0}:=v'$, thus yielding a vector $v'_{1}=R^{-1}_1 Dh_1(x_0,v'_0)$, and the choice of $l_{i}$ and $s$ allows us to iterate, thus yielding $v'_{2}, v'_{3}$ etc. 
By   construction, we know that 
$$\|D(h_{j}\dots h_{1})(x,v')\| = \prod^j_{i=1}R_i \|v'_{j}\| \geq 2^{-j} \prod_{i=1}^j \|A_{i}^{-1}\|^{-1} \geq 2^{-j} e^{(\sum_{i=1}^j l_{i})\kappa/4 + \lambda_{b}(x, h_{j}\dots h_{1})}$$
where the last inequality uses 
$$\| A_i^{-1} \| \| B_i \| \leq e^{ - l_{i} \kappa / 4} \quad\,\,\text{ and } 	\,\,\quad	 \prod_{i = 1}^{j}  \| B_i \| \geq e^{\lambda_{b}(x, h_{j}\dots h_{1} )}$$
(from Lemma \ref{lem lambdab}). Choosing $j_{0}=\min \{j \,:\ l_{1}+\dots+l_{j}\geq n\}$, we get  
$$\|D(g_{n}\dots g_{1})(x,v')\|  \geq 2^{-j_{0}} D_{1}^{-2l_{j_{0}}}e^{(\sum_{i=1}^{j_{0}} l_{i})\kappa/4 + \lambda_{b}(x, g_{n}\dots g_{1})}.$$ 
	Choosing   $a, \rho\lll_{\cK, \kappa}1$   and $n\geq |\log \rho|$, we have $(\sum_{i=1}^{j_{0}} l_{i})\kappa/8 \geq 2l_{j_{0}}\log D_{1}+j_{0}\log 2$, thus yielding the desired growth:
	$$\|D(g_{n}\dots g_{1})(x,v')\|  \geq e^{(\sum_{i=1}^{j_{0}} l_{i})\kappa/8 + \lambda_{b}(x, g_{n}\dots g_{1})}.$$ 
Note that all the conditions imposed until now   allow us to choose $s$ in terms of $M,\cK,\kappa$, and $(k_{i})$ in terms of $M,\cK,\kappa, \rho$. Moreover, we only require $\rho\lll_{\cK, \kappa}1$ because $a$ can be fixed in terms of $M, \cK, \kappa$.   
Recalling \eqref{Esize}, this concludes the proof.
		 \end{proof}

		 Proposition \ref{lem AllGapsImpliesNonConcentration} will play a crucial role in the rest of the paper.
As a first application, we use it to show that gaps at two consecutive entries imply pinching.
\begin{lemma} \label{rem allgapstopinch}	Let $  b_0 \in \{0,\dots,  d-1\}$, $\kappa > 0$. 
	Assume that $\mu$ has  a $(\kappa, \{b_0, b_{0}+1,  b_{0}+2\})$-gap. 
	Then  	for every $\eps > 0$,  the measure $\mu$ is $(\eps, \{ b_0, b_0 + 1 \})$-pinched.
\end{lemma}

Note  however that in the above statement, the time $n$ needed for $\mu$ to be $(n, \eps, b_{0},b_{0}+1)$-pinched a priori goes to infinity when $\eps$ tends to $0$.

		 \begin{proof}[Proof of Lemma \ref{rem allgapstopinch}] 
		 	Choose $n_0\geq1$ such that $\mu$ has a $(n_0,\kappa,\{b_0,b_0+1,b_0+2\})$-gap. We deal with the case  $b_{0}\in \{1, \dots, d-2\}$; the proof can easily be adapted to the limiting cases $b_{0}\in \{0, d-1\}$. Without loss of generality, we may assume  $\eps < \min(\kappa, 1) $.
		 Let $n \geq 1$, set  $\rho = e^{- n \eps/10}$.
		 	
		 	We fix some  $x \in M$,  $V_0 \in \Gr(T_xM, d - b_0)$ and  $V_1 \in \Gr(T_xM,  d - b_0 - 1)$ with $V_1 \subset V_0$.
		 	We set 
		 	\aryst
		 	&& A_0 = \{ g \in G \mid  \measuredangle(  W_{b_0}(x, g), V_0) < \rho \}, \\
		 	&& A_1 = \{ g \in G \mid  \measuredangle(  W_{b_0+1}(x, g), V_1) < \rho \}.
		 	\earyst
		 	By Proposition \ref{lem AllGapsImpliesNonConcentration},  provided that $n  \ggg_{\cK, n_{0}, \kappa, \eps} 1$, we have
		 	\aryst
		 		\mu^{*n}( A_0 ) < \rho^{c}, \
		 		\mu^{*n}( A_1 ) < \rho^{c},
		 	\earyst
			where $c=c( M,  \cK, n_{0}, \kappa)>0$. 
		 	For every $g \in \supp(\mu^{* n}) \setminus (A_0 \cup A_1)$, it is direct to verify that 
		 	\aryst 
		 	&&	    \log \sup_{u \in \S( V_1^{\perp})}  \| P_{   Dg(x, V_1)^{\perp}}(  Dg(x, u) ) \|   \leq    \lambda_{b_0 + 1}(x, g)  + \eps n / 3,     \\
		 	&&	 	  \log \inf_{v \in \S(V_0)}  \| Dg(x, v) \|       \geq    \lambda_{b_0 + 1}(x, g)  - \eps n / 3.   
		 	\earyst
		 	This implies that $\mu$ is  $(n, \eps, b_0, b_0 + 1)$-pinched.
		 \end{proof}

\subsection{Almost volume-preserving estimate} 
We continue to denote by $\mu$ a compactly supported Borel probability measure on $G=\Diff^1(M)$, as in Setting \ref{Setting-Dimpos}. 

\begin{defi} 
Let $n_{0}\in \N$, $\eta>0$. We say $\mu$ is $(n_{0},\eta)$-almost volume-preserving if for every $x\in M$, 
$$\int_{G}\log|\det Dg(x)|d\mu^{*n_{0}}(g) \in (-n_{0}\eta, n_{0}\eta). $$
\end{defi}

The next lemma is a large deviation estimate for the determinant of an element $g$ selected by an almost volume-preserving measure.
\begin{lemma} \label{lem detcloseto1} 
Assume $\mu$ is $(n_{0}, \eta)$-almost volume-preserving. Let $\eta'>\eta$. Then for $c\lll_{\cK, n_{0}, \eta, \eta'}1$, every $n\geq1$, $x\in M$,  we have 
	\aryst
				\mu^{*n} \Big ( g \in G \mid | \det Dg(x)| \in ( e^{ -  n \eta'}, e^{ n  \eta'}  ) \Big ) > 1 - 2 e^{- nc}.
				\earyst

\end{lemma}

	\begin{proof}
	Set $E=G\times M$, consider $\psi : E\rightarrow \R$ the function given by $\psi(g,x)=\log | \det Dg(x) |$. Note that $\psi$ is a cocycle: $\forall g_{1},g_{2}\in G$, $x\in M$, 
	$$ \psi(g_{2}g_{1},x)=\psi(g_{2},g_{1}x)+ \psi(g_{1},x).$$
	Note also that $\psi$ is bounded on $\supp(\mu^{* n_0}) \times M$ by $O_{\cK, n_0}(1)$, and by assumption,  we have for every $x\in M$,
	$$\int_{G} \psi(g,x) \dd \mu^{*n_{0}}(g) \in (-n_{0}\eta, n_{0}\eta).$$ 
	We are thus in a position to apply the large deviation estimate formulated in Proposition \ref{LDP-MC}, for the restrictions of the functions $f_{1}, f_{2}:E\rightarrow \R$ given by $f_{1}(g,x)=\psi(g,x)-n_{0}\eta$, $f_{2}(g,x)=-\psi(g,x)-n_{0}\eta$, to $E_0:=(\supp(\mu^{*n_0})\cup\{\mathrm{id}\})\times M$, and the Markov chain on $E_0$ with transition probabilities $P_{(h, x)} = \mu^{*n_{0}} \otimes \delta_{hx}$. We obtain for all $k\geq1$, 
		\aryst
				\mu^{*kn_{0}} \Big ( g \in G \mid  | \det Dg(x)| \notin ( e^{ -  kn_{0} (\eta+\eta')/2}, e^{ kn_{0}  (\eta+\eta')/2}  ) \Big ) <  2e^{- k n_{0}c}.
				\earyst
	where $c=c(M, \cK, n_{0}, \eta, \eta')>0$. Using that 	for any $g,h\in G$, $x\in M$, 
	$$ |\log |\det D(hg)(x)| -\log |\det Dg(x)|| \leq d\log \max(\|Dh(gx)\|, \|(Dh(gx))^{-1}\|),$$ we obtain the announced result (for all $n$) up to taking $c$ smaller. 		
	\end{proof}

		 \subsection{Expansion and  coexpansion}
	Keep the notations of Setting  \ref{Setting-Dimpos}.	 
  We study the action of the $\mu$-random walk on the tangent and cotangent bundles of $M$. 

		 \begin{defi} \label{defi coexpanding}  Let $n_{0}\in \N$, $\kappa\in \R$.
We say that  $\mu$ is {\em $(n_{0},\kappa)$-expanding} if 
		 	\aryst
		 	\inf_{x\in M, v\in \mathbb{S}(T_{x}M) }\,  \int_{G} \log \| Dg(x,v)  \| \dd\mu^{* n_{0}}(g) > n_{0}\kappa.
		 	\earyst		
We say that  $\mu$ is {\em $(n_{0},\kappa)$-coexpanding} if 		
		 	\aryst
		 	\inf_{x\in M, \xi \in \mathbb{S}(T^*_{x}M) } \,  \int_{G} \log \|  \xi \circ (Dg(x))^{-1}  \| \dd\mu^{* n_{0}}(g) > n_{0} \kappa.
		 	\earyst
		 \end{defi}
		 Note that by definition $\xi \circ (Dg(x))^{-1} \in T^*_{gx}M$ is a linear form on $T_{gx}M$.
		 Note also we may allow $\kappa$ in Definition \ref{defi coexpanding} to be negative.

\bigskip

		 We  see that under the conditions of Theorem \ref{thm mainDissipative}, $\mu$ is always expanding.
		 
\begin{lemma}\label{lem finetocoexpanding2} 
Let $b\in \{1, \dots, d-1\}$, $\kappa,\kappa', \eta>0$, $n_{0}, n_{1}\in \N$. Assume that the measure  $\mu$ has a $(n_{0},\kappa, b)$-gap, is $(n_{0},\eta, b,d)$-pinched, and is $(n_{0}, \eta)$-almost volume-preserving. 
If $\kappa'<\kappa$ and $\eta\leq c_{0}(d) (\kappa-\kappa')$ and $n_{1}\ggg_{\cK, n_{0}, \kappa, \kappa', \eta} 1$, then $\mu$ is $(n_{1}, \kappa'/d)$-expanding. 

\end{lemma}

By changing the pinching condition, we can also guarantee coexpansion.

\begin{lemma}\label{lem finetocoexpanding} 
Let $b\in \{1, \dots, d-1\}$, $\kappa, \kappa', \eta>0$, $n_{0}, n_{1}\in \N$. Assume that the measure  $\mu$ has a $(n_{0}, \kappa, b)$-gap, is $(n_{0}, \eta, 0,b)$-pinched, and is $(n_{0}, \eta)$-almost volume-preserving.  If $\kappa'<\kappa$ and $\eta\leq c_{0}(d)(\kappa-\kappa')$ and $n_{1}\ggg_{\cK, n_{0}, \kappa, \kappa', \eta} 1$, then $\mu$ is   $(n_{1},\kappa'/d)$-coexpanding.  
\end{lemma}

\begin{rema} \label{rem-exp-extbundle}The proof of Lemma \ref{lem finetocoexpanding2} yields a slightly stronger statement, which guarantees expansion uniformly on any $(d-b)$-dimensional subspace of the tangent bundle, in other terms, for all large $n_{1}\in \N$, 
	\aryst
		 	\inf_{x\in M, V\in \Gr(T_{x}M, d-b)}\,  \int_{G} \log \inf_{v\in \S(V)}\| Dg(x,v)  \| \dd\mu^{* n_{1}}(g) > n_{1} \kappa'/d.
		 	\earyst
Similarly we also have uniform coexpansion on $b$-dimensional subspaces of the cotangent bundle $T^*M$  under the hypothesis of Lemma \ref{lem finetocoexpanding}. 
\end{rema}

 \begin{proof}[Proof of Lemmas \ref{lem finetocoexpanding2}, \ref{lem finetocoexpanding}]
We focus on the proof of Lemma  \ref{lem finetocoexpanding}; that of Lemma \ref{lem finetocoexpanding2} is similar. Let $x\in M$, $\xi \in T_{x}^*M$  with $\|\xi\|=1$. Fix an arbitrary subspace $U_{\xi}\in \Gr(\Ker \xi, d-b)$. For $g\in G$, note that $W_{b}(x,g)$ is the image of the\footnote{We speak loosely here as such a subspace may not be unique, in the sense that it  depends a priori on our choice of Cartan decompositions.} $b$-dimensional subspace that is most dilated by $Dg(x)^{-1}$, therefore 
\ary\label{xiDg-lb}
\|\xi \circ (Dg(x))^{-1} \| \gtrsim e^{-\lambda_{b}(x,g)} \measuredangle (U_{\xi}, W_{b}(x,g)). 
\eary

 Set $\kappa'':= (3\kappa+\kappa')/4\in (0, \kappa)$. By Lemmas \ref{cor GapWithHighProba}, \ref{lem detcloseto1}, for some  $c=c(M, \cK, n_{0}, \kappa, \kappa', \eta)>0$, we have\footnote{Here, Lemma \ref{cor GapWithHighProba} is applied with the parameters $(\kappa,\kappa', \eta, \eta')$ therein replaced by $(\kappa, \kappa'', \eta, 2\eta)$.} 
 with $\mu^{*n}$-probability at least $1- 6e^{- nc}$:
 \aryst
&& \lambda_{b}(x,g)+ n\kappa'' \leq \lambda_{b+1}(x,g), \\
&& |\lambda_{1}(x,g)-\lambda_{b}(x,g)| \leq 2n\eta, \\
&& \log |\det Dg(x)| \leq 2n\eta.
\earyst
These conditions altogether imply 
$$2n\eta \geq\sum_{i=1}^d \lambda_{i}(x,g) \geq b(\lambda_{b} - 2n\eta)+ (d-b)(\lambda_{b}+n\kappa''),$$
which leads to 
\ary\label{lyap-lb}
-\lambda_{b} \geq n \frac{(d-b) \kappa''  -2(b+1)\eta}{d} \geq n (\kappa+\kappa')/2d
\eary
where the last inequality assumes $\eta \leq c_{0}(d)(\kappa-\kappa')$.  

On the other hand, using the non-concentration estimate for $W_{b}(x,g)$ (Proposition \ref{lem AllGapsImpliesNonConcentration}), we get that for some $\gamma=\gamma(M, \cK, n_{0}, \kappa, \kappa')>0$, we have with $\mu^{*n}$-probability at least $1-2e^{-\gamma n}$:
\ary\label{ang-lb}
 \measuredangle (U_{\xi}, W_{b}(x,g))\geq e^{-n (\kappa-\kappa')/4d}.
 \eary
Combining \eqref{xiDg-lb}, \eqref{lyap-lb}, \eqref{ang-lb}, we obtain that with $\mu^{*n}$-probability at least $1-8e^{-\min(c,\gamma) n}$, we have 
$$\|\xi \circ (Dg(x))^{-1} \| \gtrsim e^{n  (\kappa+3\kappa')/4d}.$$
Observing that the $\mu^{*n}$-integral of $\log \|\xi \circ (Dg(x))^{-1} \| $ on the complementary event can be lower bounded by $-6nD_{1}e^{-\min(c,\gamma) n}$, we obtain that $\mu$ is $(n_{1}, \kappa'/d)$-coexpanding for all $n_{1}\ggg_{\cK, n_{0}, \kappa, \kappa', \eta} 1$.
 \end{proof}

		 In dimension $2$,  Lemmas \ref{lem finetocoexpanding2}, \ref{lem finetocoexpanding} tell us that if $\mu_{0}$ has a $1$-gap, then it is expanding and coexpanding. In this context, we can show that either expansion or coexpansion implies $1$-gap and obtain the following.
		 		  
		 \begin{lemma} \label{lem contractinggaptocoexpanding0} 
		 	 Assume  $d = 2$. Let $\mu_{0} \in \cP_c(G_{\vol})$. The properties that $\mu_{0}$ be  expanding, or  be coexpanding, or have a $1$-gap, are mutually equivalent.
		 \end{lemma}
		 
		 \begin{proof}
		 It is known  (see \cite[Proposition 3.11]{DD}) that the properties that $\mu_{0}$ be expanding or coexpanding are equivalent (again, for $d=2$, $\mu_{0} \in \cP_c(G_{\vol})$). In view of  Lemmas \ref{lem finetocoexpanding2}, \ref{lem finetocoexpanding}, it remains to check that if $\mu_{0}$ is expanding, then it has a $1$-gap. 
		 	Let $x \in M$,  let $T_{x}M=U \oplus V$ be an orthogonal decomposition of $T_{x}M$ into lines, say $U=\R u_{0}$, $V=\R v_{0}$,  with $\|u_{0}\|=\|v_{0}\|=1$.
		 	Then
		 	\aryst
		 	&&   \sup_{u \in \S(U)} \| P_{Dg(x, V)^{\perp}}(  Dg(x, u) ) \|  =   \|  P_{Dg(x, V)^{\perp}}(  Dg(x,  u_{0} ) )  \|,  \\
		 	&&  \inf_{v \in \S(V)} \| Dg(x, v) \|  =  \| Dg(x, v_{0}) \|.
		 	\earyst
		 	Since $|\det Dg(x)| = 1$, we have 
		 	\aryst
		 	1 =    \|  P_{Dg(x, V)^{\perp}}(  Dg(x,  u_{0} ) )  \| \| Dg(x, v_{0}) \|.
		 	\earyst
		 	Thus we have   
		 	\begin{align} \label{d=2-gapcond}
 \log \sup_{u \in \S(U)} \| P_{Dg(x, V)^{\perp}}(  Dg(x, u) ) \|  - \log \inf_{v \in \S(V)} \| Dg(x, v) \|   = - 2 \log \| Dg(x, v_{0}) \|.
\end{align}
By the assumption that $\mu_{0}$ is expanding,  there exist an integer $n_0 > 0$ and a parameter $\kappa > 0$ (depending on $M, \mu_{0}$) such that  
\aryst 
\inf_{y \in M,  u \in \S(T_y M)}	 	\int  \log \| Dg(y, u) \|  d\mu_0^{* n_0}(g)  > n_0 \kappa.
\earyst
 In view of \eqref{d=2-gapcond}, this implies that $\mu_{0}$ has a $(n_{0}, 2\kappa, 1)$-gap.
		 \end{proof}

\section{Initial dimension : Proof of Proposition \ref{prop step I}}	  \label{sec proof prop step I}

  We complete Phase I from the proof strategy, in other terms we prove Proposition \ref{prop step I}. To highlight which assumptions on $\mu$ will play a role, we place ourselves in  the following more general setting.

 \begin{setting} \label{setting-posdim}
 We let $\mu\in \cP_{c}(G)$ be a compactly supported probability measure on $G=\Diff^1(M)$. We let $\kappa_{1}>\kappa_{2}\geq 0$, $n_{0}\in \N$. We assume that $\mu$ is $(n_{0},\kappa_{1})$-expanding, and that $\mu^{-1}$ is $(n_{0},-\kappa_{2})$-expanding. We let $\cK\subseteq G$ denote a   compact   set containing the support of $\mu$.
 \end{setting}

Recall that $\mu^{-1}$ is the pushforward of $\mu$ by the map $g\mapsto g^{-1}$. 
Therefore, the assumption on $\mu$ means that the $\mu$-random walk on $TM$ is expanding forward, and not too contracting backward.

  \begin{rema} 
The property of being $(n,\kappa)$-expanding  is stable under small perturbations of $\kappa$. Therefore, it would be  equivalent to consider the situation where $\kappa_{1}=\kappa_{2}$ in Setting \ref{setting-posdim},  instead of $\kappa_{1}>\kappa_{2}$. We choose to introduce such a constant $\kappa_{2}$ to help pinpoint the roles of the various assumptions  throughout the section. 
  \end{rema}

 Recall $\Gamma_{\mu}=\langle \supp (\mu) \rangle_{\text{grp}}$ denotes the group generated by the support of $\mu$. We show the following.


 \begin{prop}[Positive dimension] \label{prop step I gen} 
  There exists $\varkappa_{0}=\varkappa_{0}(M, \cK, n_{0},\kappa_{1})>0$ such that for every $x \in M$ with infinite $\Gamma_{\mu}$-orbit, all $\rho_0 >0$,  $n \ggg_{\mu, x, \rho_{0}, \kappa_{1}} 1$, we have
\aryst
  \sup_{y \in M} \mu^{* n} * \delta_{x}(B_{\rho}(y)) \leq 2 \rho^{\varkappa_{0}}, \quad \forall \rho \geq \rho_0.
\earyst
  \end{prop}

Note that Proposition \ref{prop step I} (Phase I) follows at once.
\begin{proof}[Proof of Proposition \ref{prop step I}]
For the duration of this proof only, we use the notations from Section \ref{sec Strategyoftheproof}. 
Provided $\eta\leq c_{0}(d)(\kappa-\check{\kappa})$,  we may apply Lemma \ref{lem finetocoexpanding2} to obtain that $\mu_{0}$ is $(n_{0},(\kappa+\check\kappa)/(2d))$-expanding for some integer $n_{0}=n_{0}(M, \mu_{0},\kappa,  \check\kappa)\geq 1$. On the other hand, we know by hypothesis that $\mu_{0}^{-1}$ is $(n_{0}, -\check{\kappa}/d)$-expanding (up to increasing $n_{0}$ if necessary).
Provided that $\cU\subseteq \cU_{0}(M, \mu_{0}, \cK, \kappa, \check\kappa)$, these two properties are still true  for the perturbation $\mu$.
Now Proposition \ref{prop step I gen} applies and yields Proposition \ref{prop step I} (after noting the dependence of $\cU_{0}, \varkappa_{0}$ on $\kappa, \check\kappa$ ultimately disappears as $\kappa, \check\kappa$ can in the end be chosen as functions of $\mu_{0}$).
\end{proof}

\bigskip 
  Our proof of Proposition \ref{prop step I gen} borrows from \cite[Section 5]{BH1} which deals with the context of Zariski-dense random walks on semisimple homogeneous spaces. It consists of two parts. First, we show that for $x\in M$ with infinite $\Gamma_{\mu}$-orbit, the maximal mass among the atoms of the distribution $\mu^{*n}*\delta_{x}$ converges to $0$ as $n$ goes to infinity (Lemma \ref{lem massdecay2}). Second, we establish that the $\mu$-random walk on $M$ tends to move points apart (Lemma \ref{lem margulisineq}). In combination with part 1, this leads to positive dimension above any prescribed scale.
  
The expansion assumptions on the random walk will be used in the following form.

\begin{lemma}\label{expM}  Let  $s>0$, $n\in \N$, $x \in M$, $v \in \S(T_xM)$.
\begin{itemize}
\item[i)] For $\kappa'_{1}\in (0,\kappa_{1})$ and $s \lll_{\cK, n_{0}, \kappa_{1}, \kappa'_{1}}1$, we have
\aryst 
	\int_{G} \| Dg(x, v) \|^{-s} \dd\mu^{*n}(g) < 2 e^{- ns \kappa'_{1}}
\earyst
and 
\aryst 
	\mu^{*n}( g \mid \| Dg(x, v) \| \leq e^{n \kappa'_{1}}) < 2 e^{- ns}.
\earyst

\item[ii)] For $\kappa'_{2}>\kappa_{2}$, the same statement holds with $(\mu, \kappa_{1}, \kappa_{1}')$ replaced by $(\mu^{-1}, -\kappa_{2},  -\kappa_{2}')$.


\end{itemize}
\end{lemma}

\begin{proof}

By hypothesis, the measure $\mu$ is $(n_{0},\kappa_{1})$-expanding: 
\ary\label{eq-expavkk'}
\inf_{y \in M, w \in \S(T_y M)} \int_{G}   \log   \|   Dg(y, w) \|  \dd\mu^{* n_0}(g) > n_{0}\kappa_{1}.
\eary
Using that $e^t\leq 1+t+t^2$ for all $t\in [-1,1]$, and taking $s$ small enough depending on $(M, \cK, n_{0},\kappa_{1}, \kappa_{1}')$, the first inequality in item i) follows (alternatively, one can deduce it from \eqref{eq-expavkk'} by invoking Proposition \ref{LDP-MC}); and the second inequality in item i) follows from Markov's inequality and by possibly reducing the size of $s$.

The same argument can be used to obtain ii), by relying this time on the assumption that $\mu^{-1}$ is $(n_{0},-\kappa_{2})$-expanding.
\end{proof}

As a first corollary of Lemma \ref{expM}, we note that the back-and-forth random process induced by $\mu^{*n}$  expands tangent vectors with high probability. 
 
\begin{cor} \label{lem NonConcentrationofAngles} 

	Let $\eps > 0$,  $n\in \N$, $x \in M$, $v  \in \S(T_xM)$. If $n \ggg_{\cK,n_{0}, (\kappa_{i})_{i},\eps} 1$, then we have 
	\aryst
	(\mu^{* n})^{\otimes 2}( (g,h) \in G\times G \mid  \| D(g^{-1}h)(x, v) \| \leq 2   ) < \eps.
	\earyst
\end{cor}

\begin{proof}
It is a direct consequence of the deviation estimates in Lemma \ref{expM}, applied with $\kappa_{1}'= (2\kappa_{1}+\kappa_{2})/3$ and $\kappa_{2}'=(\kappa_{1}+2\kappa_{2})/3$.
\end{proof}

We now show that for $x\in M$ with infinite $\Gamma_{\mu}$-orbit, the maximal mass of atoms of $\mu^{*n}*\delta_{x}$ decays to $0$ as $n\to +\infty$. 

\begin{lemma} \label{lem massdecay2}
	For every $x \in M$ with infinite $\Gamma_{\mu}$-orbit, we have 
	\aryst
	\limsup_{n \to \infty} \max_{y \in M}( \mu^{* n} * \delta_{x} )(\{ y \}) = 0.
	\earyst
\end{lemma}

\begin{proof}

	For $m \in \N$,   we define $f_m : M \to \R$ by
		\aryst
		f_m(z) = \max_{y \in M} (\mu^{*m} * \delta_z)(\{ y \}).
			\earyst

	We first prove the following. 
	\begin{lemma} \label{lem AmFinite}
		For every $\eps > 0$ and every integer $m \ggg_{\mu, \eps} 1$, the set
		\aryst
		A_m = \{ z \in M \mid f_m(z) \geq \eps \}
		\earyst
		is finite.
	\end{lemma}
	
	\begin{proof}
		
		By writing $\mu^{*(m+1)}=\mu*\mu^{*m}$, it is direct to see that the sequence of functions $f_{m}$ is non-increasing: $f_{m}\geq f_{m+1}$. Thus it suffices to show that $A_m$ is finite for some $m  > 0$   depending on $M,\mu, \eps$. 
		
		Let $m$ be a large integer to be determined later. 	For each $z \in A_m$, we fix some $y_z \in M$ such that   $\mu^{*m} * \delta_z(\{ y_z \})  \geq \eps$.  We denote by $G_{z}$ the collection of $g \in \supp(\mu^{* m})$ such that   $g(z) = y_{z}$.  In particular we have $\mu^{*m}(G_{z}) \geq \eps$.
		
		Take an arbitrary $p_0 \in M$. Let $B = B_r(p_0)$ where $r > 0$ is a parameter with $r \lll_{\mu, m, \eps} 1$. In the following, we will prove that  $| B \cap A_m| \leq 2\eps^{-1}$, therefore justifying the finiteness of  $A_{m}$.
				
		Assume to the contrary that there are  distinct $z_1, \cdots, z_{K}$ in $B \cap A_m$, with $K >  2 \eps^{-1}$. By the inclusion-exclusion principle, we may find two elements, say $z_1, z_2 \in B \cap A_m$, such that $\mu^{*m}(G_{z_1} \cap G_{z_2}) > \eps^{2}/20$.
		Set $v:=\exp_{z_1}^{-1}(z_2)$, so that $0<\|v\|\leq2r$. Take two arbitrary $g, h \in G_{z_1} \cap G_{z_2}$. Then we have $g^{-1}h (z_i) = z_i$  for $i \in \{1,2\}$. Placing ourselves in normal coordinates centered at $z_1$ and using a Taylor expansion up to order $1$,   we  deduce for this vector $v$ that $D(g^{-1}h)(z_1, v) = v + o_{\mu, m}(\|v\|)$, and in particular $\|D(g^{-1}h)(z_1, v)\|\leq 2\|v\|$. Provided $m \ggg_{\mu, \eps} 1$, this  contradicts Corollary \ref{lem NonConcentrationofAngles}. 
	\end{proof}
	
	Fix some $x \in M$ so that $\Gamma_{\mu} x$ is infinite.
	Assume by contradiction that there exists $\eps > 0$  such that $ f_n(x) \geq 2 \eps$ for all integers $n \geq1$.   
	 By Lemma \ref{lem AmFinite}, we can take some large $m$ so that $A_m$ (defined for $\eps$) is finite. 
		Given $n \geq 1$, there is some $y \in M$ such that   
			\aryst
	2 \eps &\leq&	 	(\mu^{*(n + m)}* \delta_x)(\{ y \})  = \int ( \mu^{*m} * \delta_{z})(\{y\})   \dd(\mu^{*n}* \delta_x)(z) \\
	&\leq&  \int_{A_m} ( \mu^{*m} * \delta_{z})(\{y\})   \dd(\mu^{*n}* \delta_x)(z) + \eps \leq (\mu^{* n } * \delta_x)(A_m) + \eps.
	\earyst
	Consequently, for every $n \geq 1$, we have $(\mu^{* n } * \delta_x)(A_m) \geq \eps$. Let $\nu$ be a weak limit  of the sequence $N^{-1} \sum_{n = 0}^{N-1} \mu^{*n} * \delta_{x}$; note that $\nu$ is a $\mu$-stationary probability measure on $M$. As $A_{m}$ is closed, we have  $\nu(A_m)\geq \eps$.  Since $A_m$ is a finite set, there is $z \in A_m$ such that $\nu(\{z\}) > 0$, and in particular, the orbit $\Gamma_{\mu} z$ must be finite. On the other hand, it is clear that we can choose such $z$ in  $\Gamma_{\mu} x$. Thus $\Gamma_{\mu} x$ is also finite. Contradiction. 
\end{proof}

We quickly deduce Lemma \ref{cor countablefiniteorbits} as a corollary of Lemma \ref{lem AmFinite}.

\begin{proof}[Proof of Lemma \ref{cor countablefiniteorbits}]
	Assume that $x \in M$ has a finite $\Gamma_{\mu}$-orbit with cardinality at most $N$.  For every  $m > 0$, we have $f_m(x) \geq 1/N$. In other words, $x$ belongs to the set $A_m$ defined  in Lemma \ref{lem AmFinite} and for the parameter $\eps=1/N$. For $m \ggg_{ \mu, N} 1$, we know from Lemma \ref{lem AmFinite} that $A_m$ is finite. This completes the proof.
\end{proof}

We now show that the $\mu$-random walk on $M$ tends to move points apart. This can be formulated by saying that the diagonal subset is repelling for the $\mu$-random walk on $M\times M$.  
For $s > 0$, we define $\Delta_s : M \times M \to [0,\infty]$ by 
\aryst
\Delta_s(x, y) = d_M(x, y)^{-s}.
\earyst
The next result states the function $\Delta_{s}$ is contracted by the $\mu$-random walk on $M\times M$. 


\begin{lemma} \label{lem margulisineq}
	Let $C,s>0$, $n \in \N$. If  $C^{-1},s\lll_{\cK,n_{0}, \kappa_{1}}1$,  then  we have
	\aryst
	S_{\mu}^{n} \Delta_s \leq C e^{- ns\kappa_{1}/2  }\Delta_s + C
	\earyst
	where for every distinct $x, y \in M$, we denote
	\aryst
	S_{\mu}u(x, y) = \int  u(gx, gy) d\mu(g).
	\earyst
\end{lemma}

\begin{proof}
	By Jensen's inequality, it is enough to prove the result for one fixed pair $C,s$ depending on $M, \cK, n_0, \kappa_1$. Let $k\in \N$ be a parameter. By letting $r  = r(M, \cK, k) > 0$ be a small parameter, we have for  every $x, y \in M$ with $d_M(x, y) \leq r$ that there exists $v_{x,y}\in \S(T_{x}M)$  such that, for $g\in \supp( \mu^{*k})$, 
	\aryst
	d_M(g(x), g(y)) \geq \|  Dg(x, v_{x,y}) \| d_M(x, y)/2.
	\earyst
	
	Let $\kappa_{1}'=\frac{2}{3}\kappa_{1}$.  We have   for such $x,y\in M$, 
	\begin{align*}
	\int  d_{M}(g(x), g(y) )^{-s} \dd\mu^{*k}(g) 
	&\leq 2^s\int_{G} \| Dg(x, v_{x,y}) \|^{-s} \dd\mu^{*k}(g) d_M(x, y)^{-s}\\
	&\leq 2^{s+1} e^{-k s \kappa'_{1}} d_M(x, y)^{-s}\\
	&\leq e^{-k s \kappa_{1}/2} d_M(x, y)^{-s}
	\end{align*}
	where the second inequality relies on Lemma \ref{expM} i) and the condition $s\lll_{\cK,n_{0}, \kappa_{1}}1$, and where the third inequality assumes   $k\ggg_{\kappa_{1}, s}1$. 
	
	On the other hand, when $x,y\in M$ satisfy $d_M(x, y) \geq r$, we have  for all $g\in \supp (\mu^{*k})$, 
	$$d_{M}(g(x), g(y))^{-1}=O_{\cK, k}(1).$$
	
The two previous paragraphs justify that for all $x,y\in M$, we have
		\ary\label{eq-contract}
		S_{\mu}^{k} \Delta_s \leq  e^{- k s \kappa_{1}/2 }\Delta_s + O_{\cK, k}(1).
		\eary
The stated inequality follows by iteration of \eqref{eq-contract}.
\end{proof}

We deduce that acting by $\mu$-convolution on an atom-free distribution $\nu$ on  $M$ tends to reorganize the mass of $\nu$ to ultimately create some positive dimension. 

\begin{lemma} \label{lem massdecay1}
	Let   $C, s >0$  be as in Lemma \ref{lem margulisineq}. Let $\nu$ be a finite Borel measure on $M$ of total mass at most $1$. Let $n \in \N$, $\rho,r > 0$. We have
	\aryst
	\sup_{x \in M}	(\mu^{*n} * \nu)(B_\rho(x))^2 \leq C 2^s \rho^s( e^{- ns\kappa_{1}/2} r^{-s} + 1 ) + \sup_{y \in M} \nu(B_{r}(y)).
	\earyst
\end{lemma}

\begin{proof}
	By Cauchy-Schwarz's inequality, we have 
	\aryst
	(\mu^{*n} * \nu)( B_\rho(x) )^2 &\leq& \int_{G}  g_*\nu( B_\rho(x) )^2 \dd\mu^{* n}(g) \\
	&=& \int_{G} g_*( \nu \otimes \nu )(B_\rho(x) \times B_\rho(x))  \dd\mu^{* n}(g)\\
	&\leq&\mu^{* n} *(\nu \otimes \nu)( \{ \Delta_s \geq (2 \rho)^{-s} \} ).
	\earyst
	We denote by $\omega_{r}$ the restriction of $\nu \otimes \nu$ to the set $\{ (x, y) \in M \times M \mid d_{M}(x, y) \geq r \}$. Note that $(\nu \otimes \nu-\omega_{r})(M\times M)\leq\sup_{y \in M} \nu( B_{r}(y))$, whence
	\aryst
	\mu^{* n} *(\nu \otimes \nu)( \{ \Delta_s \geq (2 \rho)^{-s} \} ) \leq  ( \mu^{* n} * \omega_{r} )( \{ \Delta_s \geq (2 \rho)^{-s} \} ) + \sup_{y \in M} \nu( B_{r}(y)).
	\earyst
	By Markov's inequality and Lemma \ref{lem margulisineq},  we have 
	\aryst
	( \mu^{* n} * \omega_{r} )( \{ \Delta_s \geq (2 \rho)^{-s} \} )  &\leq& (2 \rho)^{s} \int  S_{\mu}^n \Delta_s  \dd\omega_{r} \\ 
	&\leq& C 2^s \rho^s( e^{-n s\kappa_{1} /2 } \int \Delta_s \dd \omega_{r}  + 1 ) \\
	&\leq& C 2^s \rho^s(e^{- ns\kappa_{1} /2} r^{- s} + 1).
	\earyst
\end{proof}

\begin{proof}[Proof of Proposition \ref{prop step I gen}] 

By Lemma \ref{lem massdecay2}, there exists $n_{1}=n_{1}(\mu,x, \rho_{0})\geq1$ such  that every atom in $\mu^{*n_{1}}*\delta_{x}=:\nu$  has mass at most $\rho_{0}$.  For $r=r(\mu,x, \rho_{0})>0$ small enough, we then have $\sup_{y \in M} \nu(B_{r}(y))<2\rho_{0}$. By Lemma \ref{lem massdecay1}, we deduce that for $n\ggg_{\mu, x, \rho_{0}, \cK, n_{0}, \kappa_{1}}1$, we have $\sup_{x \in M}	(\mu^{*n} * \nu)(B_\rho(x))^2 \leq  C 2^{s+1} \rho^s+2\rho_{0}$. Note the dependence on $ \cK, n_{0}$ of the lower bound for $n$ can be absorbed by that in $\mu, \kappa_{1}$. Observing that $C 2^{s+1} \rho^s+2\rho_{0}\leq \rho^{s/2}$ for $\rho_{0}\leq \rho\lll_{C,s}1$,
the proposition follows.

\end{proof}

%
%
%
%
%
%
%

		 \section{B\'enard-He's Multislicing Theorem} \label{sec BHMT}

		 We present a multislicing statement in the spirit of \cite{BH1, BH2}, but adapted to our needs. Roughly speaking, we consider a finite measure $\nu$ on $\R^d$ which has positive dimension at certain scales, and give sufficient conditions for the measure to have better dimensional properties with respect to asymmetric boxes in generic position. By generic position, we essentially mean that the partial flag of $\R^d$ which captures the geometry of the box is chosen randomly by a measure satisfying suitable projection theorems. Compared to \cite{BH2}, we allow the measure $\nu$ to be supported on a microscopic ball and have microscopic mass, as well as the side lengths of the boxes to vary. 
		 \smallskip
		 
Since $M$ is irrelevant to this section, all dependencies of implicit constants are stated explicitly in the statements in this section.
		 
	\subsection{Boxes}  
		Let $d\geq2$.  We denote by $\cF$ the full flag variety of $\R^d$, in other terms it is the collection of all the tuples  of subspaces $(W_{i})_{i=1}^{d}$ such that 
		 \aryst
		 \{0\}\subsetneq  W_{1}\subsetneq \dots  \subsetneq W_{d}=\R^d, \quad \forall i, \,\dim W_{i}=  i.
		 \earyst
		 We set
		 $$ \increasing:=\{(t_{i})_{i=1}^{d} :\,  0\leq t_{1} \leq \dots  \leq t_{d}\leq 1\}.$$
		For $\cW=(W_{i})_{i=1}^{d} \in \cF$, $\bt=(t_{i})_{i=1}^{d}\in   \increasing $, and $\rho\in (0, 1)$,  we introduce the box
		 $$B_{\rho^\bt}^{\cW} :=\sum_{i=1}^{d} B^{W_{i}}_{\rho^{t_{i}}}.$$
 		 We call $\cW$ the flag (or the filtration) carrying the box $B_{\rho^\bt}^{\cW}$.   
	
		  We will consider boxes $B_{\rho^\bt}^{\cW}$ whose exponent $\bt$ satisfies certain gap and pinching conditions. For that, given $m\in \{ 1, \cdots, d-1\}$, we set  $P_{m}(d)=\{\bd=(d_{i})_{i=1}^{m+1}\in \N^{m+1}\,:\, 0<d_{1}<\dots <d_{m+1}=d \}$. For $\bd\in P_{m}(d)$, $k_{0}\in \{1, \cdots, m\}$, $ \kappa, \eta>0$, we introduce the notation 
		 $\increasing(\bd, k_{0}, \kappa, \eta)$ to be the collection of $\bt \in \increasing$ such that 
		 \begin{itemize}
		 \item $t_{d_{2}}\geq \kappa$,
		 \item $t_{d_{k_{0} +1}}-t_{d_{k_{0}}}\geq \kappa$,
		 \item for all $k\in \{  0, \cdots, m \}$, $i,j\in \{   d_{k}+1,  \cdots, d_{k+1} \}$, one has $|t_{i}-t_{j}|\leq \eta$ (where $d_{0}=0$ by convention).
		 \end{itemize} 
		 
	 \subsection{Projection theorems} 
		
		Fix an integer $d' \in \{ 1, \cdots, d-1 \}$. The following definition depicts abstractly what it means for a measure on the Grassmannian  $\Gr(d,d')$ to satisfy a {\it subcritical projection theorem}. Given $A\subseteq \R^d$ and $\rho>0$, we denote by $\cN_{\rho}(A)$ the $\rho$-covering number of $A$, i.e. the least number of open $\rho$-balls needed to cover $A$.
		 
		 \begin{defi}  \label{Sub-crit-P} 
		 	Let $\sigma$ be a probability measure on $\Gr(d, d')$. Let $ \rho, \eps, \tau>0$.  We say $\sigma$ has the \emph{subcritical property} $\SubP$ with  parameters $(\rho, \eps,\tau)$ if  for every set $A\subseteq B^{\R^d}_{1}$, the exceptional set
		 	\begin{equation*}
		 		\begin{split}
		 			\cE := \Bigl\{\, W  \in  \Gr(d, d')  \mid  \exists A' \subseteq A \,\,&\text{ with }\,\,\cN_{\rho}(A')\geq \rho^{\eps} \cN_{\rho}(A)\\
		 			& \text{ and }\, \cN_{\rho}(P_{W^{\perp}}A')< \rho^\tau \cN_{\rho}(A)^{\frac{d-d'}{d}}\Bigr\}
		 		\end{split}
		 	\end{equation*}
		 	has measure $\sigma(\cE)\leq \rho^\eps$.
		 \end{defi}

	The quantity $d-d'$ appearing in $\cE$ should be interpreted as $\dim W^\perp$, i.e. the rank of the  projector $P_{W^{\perp}}$.	 Roughly speaking,   $\SubP$ requires that for any fixed $A\subseteq B^{\R^d}_{1}$, for most subspaces $W$ chosen by $\sigma$, for every rather large subset $A'$ of $A$, the orthogonal projection of $A'$ to $W^\perp$ has covering number at least $\rho^\tau \cN_{\rho}(A)^{\frac{\dim W^\perp}{d}}$. Here the parameter $\tau$ expresses a dimensional loss compared to the heuristic that $\cN_{\rho}(P_{W^{\perp}}A)\geq \cN_{\rho}(A)^{\frac{\dim W^\perp}{d}}$ (for which equality is achieved when $A$ is a ball).

		 The following is a useful criterion for verifying the subcritical property.

		 \begin{defi} 
		 	Let $\sigma$ be a probability measure on $\Gr(d, d')$.  Let $r, c > 0$. 
		 	We say that $\sigma$ has the \emph{non-concentration property $\NCSub
		 $} with parameters $(r, c)$ if the following holds.
		 
		 For every non-zero linear subspace $V \subseteq \R^d$, setting  
		 \aryst
		 \Sigma_{d'}(V) = \{ W' \in \Gr(d, d') \mid  \dim(V\cap W') / \dim(V) > d' / d  \},
		 \earyst
		  we have 
		 	\aryst
		 	\sigma(  W \in \Gr(d, d') \mid  {\rm dist}( W, \Sigma_{d'}(V) ) < r   ) <  r^{c}.
		 	\earyst
		 \end{defi} 
		 
		 The collection $\Sigma_{d'}(V)$ is called the $d'$-dimensional \emph{constraining Schubert variety} at $V$. Avoiding constraining Schubert varieties is a necessary condition for a measure to satisfy a subcritical projection theorem with  small dimension loss. The following result, extracted from \cite[Theorem 4.1]{BH2}, claims it is also sufficient.
		  
		 \begin{thm}  \label{lem NC-toS-} 
 		 		Let $c, \eps, \rho \in (0, 1/10)$ such that $\eps \lll_{d,c}1$ and $\rho\lll_{d,c, \eps}1$.  
		 	  If $\sigma$ satisfies $\NCSub$ with parameters $(\rho^{\sqrt{\eps}}, c)$, then $\sigma$ has the subcritical property $\SubP$ with  parameters $(\rho, \eps, D  c^{-1} \sqrt{\eps})$  for some   $D=  O_{d}(1)$.
		 \end{thm}

		 We now turn to supercritical estimates. Given $\alpha, \tau, \rho>0$, and $A\subseteq \R^d$, we set   
		 \begin{equation} \label{notation-cEaed} 
		 	\begin{split} 
		 		\cE^{(\alpha, \tau)}_{\rho}(A) := \Big \{\, W\in \Gr(d, d') \mid  \exists A' \subseteq A \,\,&\text{ with }\,\,\cN_{\rho}(A')\geq \rho^\tau \cN_{\rho}(A)\\
		 		& \text{ and }\, \,\cN_{\rho}(P_{W^{\perp}}A') < \rho^{- \alpha (d-d')  -\tau}\Big \}.
		 	\end{split}
		 \end{equation}
		 Again the quantity 	$d-d'$ appearing here should be interpreted as the rank of the  projector $P_{W^{\perp}}$. 
		 
		 The following definition depicts abstractly what it means for a measure on $\Gr(d,d')$ to satisfy a {\it supercritical projection theorem}.
		 
		 \begin{defi}  \label{SAP} 
		 	Let $\sigma$ be a  probability measure on $\Gr(d, d')$. Let $\rho, \varkappa, \tau>0$.  We say $\sigma$ has the \emph{supercritical  property} $\SupP$ with parameters $(\rho, \varkappa,\tau)$ if the following holds.
		 	
		 	Let $A\subseteq B^{\R^d}_{1}$ be any non-empty subset satisfying for some
		 	$\alpha\in[\varkappa, 1-\varkappa]$, for every $r \geq \rho$,
		 	\begin{equation}\label{nc-dim-dalpha}
		 		\sup_{v \in \R^d} \cN_\rho(A \cap (B^{\R^d}_r + v)) \leq \rho^{-\tau} r^{d \alpha } \cN_\rho(A).
		 	\end{equation}
 Then we have
		 	$$\sigma(\cE^{(\alpha, \tau)}_{\rho}(A)) \leq \rho^\tau.$$  
		 \end{defi}

				We record a useful criterion for verifying the supercritical property.

		 \begin{defi}
 	Let $\sigma$ be a probability measure on $\Gr(d, d')$.  Let $\rho,  \eps, c   > 0$. 
 	We say that $\sigma$ has the \emph{non-concentration property} $\NCSup$ with parameters $(\rho, \eps, c)$ if  for every $V \in \Gr(d, d-d')$ we have 
		 	\aryst
		 	\sigma( W \in \Gr(d, d') \mid  \measuredangle(V, W) < r  ) <  \rho^{- \eps} r^{c}, \quad r \geq \rho.
		 	\earyst
		 \end{defi}
		 
		 The following is due to Bourgain \cite[Theorem 5]{B} for projectors of rank $1$ (i.e. $d'=d-1$) and to He  \cite[Theorem 1]{H} in higher rank (i.e. $d'\leq d-2$).
		 \begin{thm} \label{lem NCtoS}
		 	Let $\varkappa,c, \eps, \rho>0$ with $\rho, \eps \lll_{d, \varkappa, c} 1$. If $\sigma$  satisfies $\NCSup$ with parameters $(\rho, \eps, c)$, then
			  $\sigma$ has the supercritical property $\SupP$ with  parameters $(\rho, \varkappa, \eps)$.
		 \end{thm} 
		  
		  We end this subsection by observing that the non-concentration condition  $\NCSup$ in the supercritical theorem is stronger than its subcritical counterpart $\NCSub$.
		  
		 \begin{lemma} \label{lem NC+toNC-}
		 	Let $c, \eps, \rho>0$ with $\rho \lll_{d} 1$ and $4\eps<c$. 	If $\sigma$ satisfies $\NCSup$ with parameters  $(\rho, \eps, c)$,
		 	then $\sigma$ satisfies $\NCSub$ with parameters   
		 	 $(\rho, c / 4)$.
		 \end{lemma}

		 \begin{proof}
		 Let $V\in \Gr(\R^d)\smallsetminus \{0\}$. We will give a lower bound for the distance between $\Sigma_{d'}(V)$ and most $W\sim \sigma$.

		 Fix an arbitrary subspace $U\subseteq V$ with 
		 $$\dim U= \min (\dim V, \,d-d').$$
Given any $W'\in  \Sigma_{d'}(V)$, observe that	$\dim(W'\cap V)+ \dim U >\dim V$, thus
		   $W'\cap U\neq \{0\}$. Enlarging $U$ if necessary, we may assume that $\dim U=d-d'$. In particular, if $\dist(W,W')< r$, then every line of $W'$ is  $O_{d}(r)$-close to a line in $W$, which leads to
		 $$ \measuredangle(W,U)=O_{d}(r).$$ 
		 It follows that
		 \aryst
		 \sigma (W   \mid \dist(W,\Sigma_{d'}(V)) <r ) \leq \sigma (W \mid \measuredangle(W,U)  =   O_{d}(r) )\leq \rho^{-\eps} (O_{d}(r))^c
		 \earyst
		 where the last upper bound relies on the  $\NCSup$ assumption. By taking $r=\rho$, and assuming $\rho \lll_{d} 1$; $4\eps<c$, this finishes the proof.		
		 \end{proof}

		\subsection{Multislicing theorem} 
		We present the main result of the section, Theorem \ref{thm Multislicing0}. Before giving the precise statement, we describe the heuristic. 
		 
		 First, we consider a random box in $\R^d$, of the form $ (B^{\cW_{\theta}}_{\rho^{\bt_{\theta} }})_{\theta\sim \sigma}$ where $\theta$ is a random variable with distribution $\sigma$, $(\cW_{\theta})_{\theta\sim \sigma}$ is a random flag, $\rho\in(0, 1)$ a scale, and $(\bt_{\theta})_{\theta\sim \sigma}$ prescribes random side length exponents. We assume that all exponent vectors $\bt_{\theta}$ belong to a  subset  $\square'\subseteq \increasing$, in which every $\bt=(t_{i})_{i=1}^d$  is pinched along a prescribed subsequence $t_{d_{1}}\leq\dots \leq t_{ d_{m+1}}$,  with $\kappa$-gap between $t_{d_{k_{0}}}$ and $t_{d_{k_{0}+1}}$ for some   $k_{0}\in \{1, \cdots, m\}$, and second minimum $t_{d_{2}}$ bounded from below by $\kappa$. For all $k \in \{1, \dots,m\}$, we assume the random subspace $(W_{\theta,d_{k}})_{\theta\sim \sigma}$ satisfies a subcritical projection theorem, and also a supercritical projection theorem for $k=k_{0}$. 
		 
		 Second, we consider a measure $\nu$ on $B^{\R^d}_{1}$ which has normalized dimension $\alpha\in (0, 1)$ at scales that appear as side lengths of the random box, i.e. the $\rho^{t_{i}}$ where $\bt\in \square'$, and also on the interval $[\rho^{t_{d_{k_{0}+1}}},\rho^{t_{d_{k_{0}}}}]$ as this is needed to exploit the supercritical assumption on $(W_{\theta,d_{k_{0}}})_{\theta\sim \sigma}$. We do not assume that $\nu$ is a probability measure. In fact, in the context of our application to random walks, the theorem below  will be applied to a measure $\nu$ supported on a ball of radius $\rho^{1/3}$ and of mass smaller than $\rho^{c_{\mu}/3}$ where $c_{\mu}>0$ depends on the measure $\mu$ driving the  walk. We let $\zeta$ be a parameter so that $\nu$ is supported on a ball of radius $\zeta$ in which all the random boxes $B^{\cW_{\theta}}_{\rho^{\bt_{\theta} }}$ fit as well. 
		 
		 The conclusion is that for most parameters $\theta$ with respect to  $\sigma$, up to throwing away a small part of the measure  $\nu$ (relative to $\nu(\R^d) + \zeta^d$  ), the measure $\nu$ has improved dimension with respect to $B^{\cW_{\theta}}_{\rho^{\bt_{\theta} }}$ and all its additive translates in $\R^d$.

		  \begin{thm}[Non-homogeneous \& local supercritical multislicing] \label{thm Multislicing0}
		 	Let $d>m \geq1$ be integers, $\bd\in P_{m}(d)$, $k_{0}\in \{1, \cdots, m\} $. 
			Consider parameters    $\rho, \eta, \eps, \eps', \varkappa, \kappa, \tau, \tau' \in (0, 1/10)$.  Let $\square'\subseteq \increasing(\bd, k_{0}, \kappa, \eta)$.   
		 	
		 	Let $(\Theta, \sigma)$ be a measurable space $\Theta$ endowed with a probability measure $\sigma$.  Consider a measurable family of flags $(\cW_{\theta} = ( W_{\theta, i} )_{i = 1}^{d} )_{\theta\in \Theta} \in \cF^{\Theta}$ and  of exponents $(\bt_{\theta} = (t_{\theta, i})_{i = 1}^{d} )_{\theta \in \Theta} \in \square'^{\Theta}$.
		 	We assume the following is true:
		 	\enmt
		 	\item for every $\bt\in \square'$, every $k\in \{1, \dots, m \}$, the distribution of $(W_{\theta, d_k})_{\theta\sim \sigma}$  satisfies $\SubP$ with parameters $( \rho^{t_{d_{k+1}}},    \eps, \tau)$, 
		 	\item for  every $\bt\in \square'$, the distribution of $(W_{\theta, d_{k_{0}} })_{\theta\sim \sigma}$  satisfies $\SupP$ with parameters $(\rho^{t_{d_{k_{0}+1}}-t_{d_{k_{0}}}}, \varkappa,\tau')$. 
		 	\eenmt

		 	Let $\nu$ be a Borel measure on $\R^d$ of mass  at most $1$,    and such that for some $\alpha\in [\varkappa, 1-\varkappa ]$,  for all  $\bt\in \square'$,     all   $r \in \{\rho^{t_{d_k}}\}_{k=1}^{m+1}\cup [  \rho^{t_{d_{k_{0}+1}}}, \rho^{t_{d_{k_{0}}}} ]$, all $v \in \R^d$,  we have
		 	$$\nu(B^{\R^d}_{r} + v)\leq\rho^{-\eps} r^{d\alpha}.$$
		 	
		 	Let $\zeta\in(0,1]$ such that $\supp\nu$ is contained in a ball of radius $\zeta$ and $\zeta \geq\sup_{\bt\in \square'}\rho^{t_{1}}$.
		 	
		 	If $\eps' \lll  \eps$, and $\eta, \eps, \tau \lll_{d, \kappa, \tau'}1$, and $\rho\lll_{d, \kappa, \tau', \eps} 1$, then there is an event $\cE \subset \Theta$ such that $\sigma(\cE) \leq \rho^{\kappa \eps'}$, and for each $\theta \in \Theta \setminus \cE$, there is a set $F_{\theta} \subseteq \R^d$ satisfying  $\nu( F_{\theta}) \leq \rho^{\kappa \eps'}(\nu(\R^d)+\zeta^{d})$, such that  for every $v \in \R^d$,  
	 	$$( \nu|_{\R^d \smallsetminus F_{\theta}} )\left( B^{\cW_{\theta}}_{\rho^{\bt_{\theta} }}+ v \right) \leq \rho^{\kappa \tau'/( 200 d) }   \leb \left(B^{\cW_{ \theta}}_{\rho^{\bt_{ \theta}}}\right)^{\alpha}.$$
		 \end{thm}


			
			\bigskip	 
		 \begin{proof} 
		  The case where the pinching parameter $\eta=0$, the support parameter $\zeta=1$, and  $\square'$ is a singleton (i.e. one considers a deterministic set of exponents) is due to \cite[Theorem 3.4]{BH2},  with slightly better dimensional gain  ($\rho^{\kappa \tau'/( 100 d) } $ instead of $\rho^{\kappa \tau'/( 200 d)}$). Our first step is to derive  Theorem \ref{thm Multislicing0} for arbitrary $\eta, \zeta$, while still assuming that $\square'$ is a singleton. This is done by tessellating $B^{\R^d}_1$ with copies of $\nu$.  Afterward, we  derive  Theorem \ref{thm Multislicing0} in full generality  by discretizing $\square'$ to reduce to the singleton case. 
		  
		  Note that throughout  the proof we may assume that $\rho$  is small enough in terms of $\eps'$ as well (not just $d, \kappa, \tau', \eps$). This is because if the conclusion of the theorem holds for a given value of $\eps'$, then it holds for all smaller values.

		 \bigskip
		 \noindent \underline{Step 1}: \emph{If $\square'$ is a singleton $\square'=\{\bt\}$ then Theorem \ref{thm Multislicing0} holds, with better dimensional gain $\rho^{\kappa \tau'/( 150 d) } $ instead of $\rho^{\kappa \tau'/( 200 d) }$.}
		 
%
		
		 One may assume $\nu(\R^d)\geq \rho^{\kappa \eps'}\zeta^{d}$ otherwise the result is trivial (taking $\cE=\emptyset$, $F_{\theta}=\R^d$). Let $Q$ be a ball of radius $\zeta$ such that $\supp (\nu) \subseteq Q$. Consider a finite family $(w_{j})_{j\in J}\in (\R^d)^J$ such that the translates $Q+ w_{j}$ are $2\zeta$-separated, included in $B^{\R^d}_{1}$, and let $\nu_{j}$ be the measure on $Q+w_{j}$ obtained by pushforward of $\nu$ under $v \mapsto v+w_{j}$. Provided $\rho^{\kappa \eps'}\lll_{d}1$, the  bounds  $1\geq \nu(\R^d)\geq \rho^{\kappa \eps'}\zeta^{d}$ allow us to choose the family $(w_{j})_{j\in J}$ so that $\tilde\nu:=\sum_{j\in J}\nu_{j}$ satisfies
		 $$1\geq \tilde\nu(\R^d)\geq \rho^{2\kappa \eps'}. $$
		 Observe that the measure $\tilde\nu$ still satisfies the required non-concentration assumption: By the $2\zeta$-separation assumption on the $(Q+w_{j})_{j}$ and the condition $ \rho^{t_{1}}\leq \zeta$, we have for all $r \in \{\rho^{t_{d_k}}\}_{k=1}^{m+1}\cup [  \rho^{t_{d_{k_{0}+1}}}, \rho^{t_{ d_{k_{0}}}} ]$, all $v \in \R^d$
	$$\tilde{\nu}(B^{\R^d}_{r}+v)\leq\rho^{-\eps} r^{d\alpha}.$$

	We write
	\aryst
	\obt:=  ( \underbrace{t_{d_1}, \dots, t_{d_1}}_{d_1}, \underbrace{t_{d_2}, \dots, t_{d_2}}_{d_2 - d_1}, \dots, \underbrace{t_{d_{m+1}}, \dots, t_{d_{m+1}}}_{d_{m+1} - d_m} ).
	\earyst		
	Namely,	$\obt$ is the set of exponents obtained from $\bt$ by collapsing each $t_{i}$ to the smallest $t_{d_{k}}$ above it.	We  apply Theorem \ref{thm Multislicing0} in the regime $\eta=0$, $\zeta=1$, $\square'=\{\,\obt\,\}$ (known by  \cite[Theorem 3.4]{BH2} and with better dimensional gain, as mentioned above). Provided $\eps' \lll  \eps$, and $\eps, \tau \lll_{d, \kappa, \tau'}1$, and $\rho\lll_{d, \kappa, \tau', \eps} 1$, we obtain the existence of $\tilde\cE\subseteq \Theta$ and $\tilde F_{\theta}\subseteq \R^d$ (for $\theta\in \Theta\smallsetminus \tilde\cE$) such that $\sigma(\tilde\cE), \tilde\nu(\tilde F_{\theta})\leq \rho^{3 \kappa \eps'}   $, and for all $v \in \R^d$, 
		\begin{align*}
		( \tilde\nu|_{\R^d \smallsetminus \tilde F_{\theta}} )\left( B^{\cW_{\theta}}_{\rho^{\obt }} + v \right) \lesssim_d \rho^{\kappa \tau'/( 100 d) }   \leb \left(B^{\cW_{ \theta}}_{\rho^{\obt}}\right)^{\alpha}.
		\end{align*}
Assuming $\eta\lll_{d,\kappa, \tau'} 1$, this inequality implies in the case of  exponents given by $\bt$:
\begin{align}\label{eq-tilde-nu-mult}
		( \tilde\nu|_{\R^d \smallsetminus \tilde F_{\theta}} )\left( B^{\cW_{\theta}}_{\rho^{\bt }} + v \right) \leq \rho^{\kappa \tau'/( 150 d) }   \leb \left(B^{\cW_{ \theta}}_{\rho^{\bt}}\right)^{\alpha}.
		\end{align}		
By the pigeonhole principle and the bounds $\tilde\nu(\tilde F_{\theta})\leq \rho^{3\kappa \eps'}\leq \rho^{\kappa \eps'} \tilde\nu(\R^d)$, there exists $j\in J$ such that $\nu_{j}(\tilde F_{\theta})\leq \rho^{\kappa \eps'}\nu_j(\R^d)$. Setting $\cE=\tilde\cE$, $F_{\theta}=\tilde F_{\theta}- w_{j}$, we find $\sigma(\cE)\leq \rho^{\kappa \eps'}$, $\nu(F_{\theta})\leq \rho^{\kappa \eps'}\nu(\R^d)$ and that Equation \eqref{eq-tilde-nu-mult}	is valid with $(\nu, F_{\theta})$ in the place of $(\tilde\nu, \tilde F_{\theta})$. This concludes the proof of Step 1. 

		 		 \bigskip

		 \noindent \underline{Step 2}: \emph{Theorem \ref{thm Multislicing0} holds (even if  $\square'$ is not a singleton).}
		 
		 Let $\{\bt^{(l)} = (t^{(l)}_i)_{i = 1}^{d} \}_{l=1}^L$ be a maximal $\eps$-separated subset of $\square'$. In particular, $L\ll_{d} \eps^{-d}$.  Fix $l\in \{1, \dots, L\}$ and, invoking Step 1, apply  Theorem \ref{thm Multislicing0} with the same random flag $(\cW_{\theta})_{\theta\sim \sigma}$
		 but with  exponents  given by the singleton $\{\bt^{(l)}\}$ and with $\eps'$ replaced by $2\eps'$. Call $\cE_{l}\subseteq \Theta$ and $F_{l, \theta}\subseteq \R^d$ (for $\theta\in \Theta\smallsetminus \cE_{l}$) the associated exceptional subsets given by Step 1. Set $\cE=\cup_{l=1}^L\cE_{l}$, and for $\theta\notin \cE$, set $F_{\theta}=\cup_{l=1}^L F_{l, \theta}$. Provided   $\eps'\lll \eps$ and $\rho\lll_{\eps'}1$, we then have $\sigma(\cE)\leq L\rho^{2\kappa \eps'}\leq \rho^{\kappa \eps'} $ and similarly $ \nu(F_{\theta})\leq \rho^{\kappa \eps'}(\nu(\R^d) + \zeta^{d})$. Moreover, for $\theta\notin \cE$, choosing $l=l(\theta)$ such that $t^{(l)}_{i}-\eps \leq t_{\theta, i}\leq t^{(l)}_{i} +\eps$ for all $i=1, \dots, d$, we find for every $v \in \R^d$ 
		 \begin{align*}
		 ( \nu|_{\R^d \smallsetminus F_{\theta}} )\left( B^{\cW_{\theta}}_{\rho^{\bt_{\theta} }}+ v \right) 
		 &\leq \rho^{-2d\eps} \sup_{w\in \R^d}( \nu|_{\R^d \smallsetminus F_{\theta}} )\left( B^{\cW_{\theta}}_{\rho^{\bt^{(l)}}}+ w \right) \\
		 &\leq \rho^{-2d\eps+\kappa \tau'/( 150 d) }   \leb \left(B^{\cW_{ \theta}}_{\rho^{\bt^{(l)}}}\right)^{\alpha}\\
		 &\leq \rho^{\kappa \tau'/( 200 d) }   \leb \left(B^{\cW_{ \theta}}_{\rho^{\bt_{\theta}  }}\right)^{\alpha}.
		  \end{align*}
where  the last inequality uses $\eps\lll_{d, \kappa, \tau'}1$. This concludes the proof.

		 \end{proof}

		 \section{Dimensional increment : Proof of Proposition \ref{prop step II}} \label{sec proof prop step II}
		 
		 
In this section, we complete Phase II, by proving Proposition \ref{prop step II}. To highlight which assumptions on $\mu$ will play a role, we place ourselves in  the following more general setting.

     	\begin{setting}\label{Setting-bootstrap}
	Let $\kappa, \eta> 0$, $n_{0}, n_{1}\in \N$, and  $J = \{  0 = d_0 < d_1 <   \cdots <  d_{m+1} = d \}$  be a subset of $ \{0, \cdots, d \}$ with at least $3$ elements (i.e. $m\geq1$). 
	Let $\mu$ be a compactly supported  probability measure on $\Diff^2(M)$. Assume  that
	$\mu$  has a  $(n_{0}, \kappa, J)$-gap, is $(n_{1},\eta, J)$-pinched, and is $(n_{1},\eta)$-almost volume-preserving.
	Let $\cK\subseteq \Diff^2(M)$ be a   compact     set containing the support of $\mu$.
	\end{setting}


		   We show the following   dimensional increment result.

		   \begin{prop}[Dimensional increment] \label{prop step II-single-gen} 
		 	Let $\varkappa,\rho,  \eps \in (0,1/10)$,   $\tau \geq 0$, $\alpha \in [\varkappa, 1 - \varkappa]$. Let $\nu$ be a Borel measure on $M$ of mass at most $1$ and that is $(\alpha, \,\cB_{[\rho, \rho^{1/4}]}, \tau)$-robust.  If    $\eta\lll_{\cK,n_{0}, \kappa, \varkappa}1$ and $\eps, \rho\lll_{\cK,n_{0}, \kappa, \varkappa, n_{1}}1$, then for some integer $n_{\rho}=n_{\rho}(M, \cK, \rho) \in [1,  | \log \rho |]$, we have
			$$
			\text{$\mu^{*n_{\rho}} * \nu$ is $(\alpha + \eps, \,\cB_{\rho^{1/2}},\, \tau + \rho^{\eps})$-robust.}
			$$
		 \end{prop}

	Note that Proposition \ref{prop step II-single} follows at once. 
	
	\begin{proof}[Proof of Proposition  \ref{prop step II-single}] 
	It is a consequence of Proposition  \ref{prop step II-single-gen} because Setting \ref{Setting-bootstrap} is looser than Setting \ref{Setting-main-result} (in view of Lemma \ref{lem detcloseto1}). Note that indicating dependences on $n_{0},\kappa, n_{1}$ is not needed here, because those parameters can be taken as functions of $M, \mu_{0}, \cK$. 
	\end{proof}

	\subsection{Linearization at local scales}	  \label{subsec	linearcharts}
	
	The proof of Proposition \ref{prop step II-single-gen} will rely on the multislicing estimate provided by Theorem \ref{thm Multislicing0}. To apply the latter, we need to place ourselves in suitable charts where the translates of balls by an element $g\in G$ look like genuine Euclidean boxes (with no curvature).  
	To this end, we let $\delta_{inj}$ be the injectivity radius of $M$, and   denote $$\varphi_{x} = \exp_x^{-1} : \exp_x( B^{T_x M}_{\delta_{inj}/2} ) \to  T_xM \simeq \R^d.$$
	Henceforth we  abusively see  $\varphi_x$ as a map from $ \exp_x( B^{T_x M}_{\delta_{inj}/2} ) $ to $\R^d$ through  an arbitrarily chosen isometry between $T_x M$ and $\R^d$. Clearly, the $C^2$-norms of $\varphi_x$ and $\varphi_x^{-1}$ are  bounded uniformly in $x$.

\begin{lemma}[Linearization] \label{lin-chart}
	There exists	$R(d) > 1$ depending only on $d$ such that the following is true.
Let $g\in \Diff^2(M)$ and let $\zeta>0$. Assume that $\zeta \lll C_0^{-R(d)}$ with   $C_0 :=2\max(\|g\|_{C^2}, \|g^{-1}\|_{C^2} )$.
 Then for any  $x, y \in M$      and any   $r\in (0, \zeta] $,   
     		\begin{align*}
		\text{$\varphi_x(B_{\zeta}(x) \cap g ( B_{r}(y) ) )$} &   \text{ can be covered by}\\
		&\text{ $R(d)$-many  translates of   $D(\varphi_x g) \left( g^{-1}(x),  B^{T_{g^{-1}(x)}M}_{r} \right)$.}
		\end{align*}
\end{lemma}



\begin{proof}
In the proof,  $R(d)$ denotes a positive constant depending on $d$  and whose value may increase from line to line.

If $B_\zeta(x)\cap g B_r(y)$ is empty, there is
nothing to prove. We assume it is not empty. This implies in particular that $gB_{r}(y)\subseteq B_{\zeta+\diam gB_{r}(y)}(x)\subseteq B_{3\zeta C_{0}}(x)$ where the second inequality uses $r\leq \zeta$ and $\|Dg\|_{L^\infty}\leq C_{0}$. Letting $\zeta \lll C_0^{-R(d)}$, we thus have $gB_{r}(y)\subseteq B_{\delta_{inj}/2}(x)$ so we may apply $\varphi_{x}$ to all of $gB_{r}(y)$. We bound the covering number of the whole image  $\varphi_x(g ( B_{r}(y) ) )$. 

Taylor expansion at $y$ gives
\begin{align}
\varphi_x(g(B_r(y)))
&\subseteq  \varphi_x(g(y)) + D(\varphi_x g)(y,B^{T_yM}_{2r}) + B^{\R^d}_{C_M C_0^{R(d)}r^2}, \nonumber\\
&\subseteq  \varphi_x(g(y)) + D(\varphi_x g)(y,B^{T_yM}_{3r})  \label{eq-Tayl}
\end{align}
where $C_M$ depends only on $M$, and the second bound uses $r\leq \zeta$ and the upper bound on $\zeta$. 

It remains to compare $D(\varphi_x g)(y,B^{T_yM}_{3r})$ with $D(\varphi_x g)(q,B^{T_qM}_{3r})$
where  $q:=g^{-1}(x)$. We have
\[
d_M(y,q)\leq r+\|Dg^{-1}\|_{L^\infty}\zeta
\leq 2C_0\zeta.
\]
Let $P_{q,y}:T_qM\to T_yM$ be parallel transport
along the minimizing geodesic from $q$ to $y$.
By the uniform $C^2$ bounds, we obtain
\[
\|D(\varphi_x g)(y)\circ P_{q,y}
     -D(\varphi_x g)(q)\|
\leq C_M C_0^{R(d)}\zeta.
\]
Consequently,
\begin{equation}\label{eq-ppQ}
D(\varphi_x g)(y,B^{T_yM}_{3r})
\subseteq
D(\varphi_x g)(q,B^{T_qM}_{3r})
+B^{\R^d}_{C_M C_0^{R(d)}\zeta r},
\end{equation}
after enlarging $C_M$ if necessary.

As  before, taking  $\zeta$ small enough so that 
$C_M C_0^{R(d)}\zeta\leq C_0^{-1}$, we deduce that
\[
D(\varphi_x g)(y,B^{T_yM}_{3r}) \subseteq D(\varphi_x g)(q,B^{T_qM}_{4r}).
\]
In view of \eqref{eq-Tayl}, this concludes the proof.
\end{proof}

 		\subsection{Proof of Proposition \ref{prop step II-single-gen}} 
		Note we may assume $\tau = 0$, i.e. $\nu$   is $(\alpha, \cB_{r}, 0)$-robust for every $r \in [\rho, \rho^{1/4}]$. We may also suppose $\rho$ is small enough in terms of $\eps$ as well (not only in terms of $M, \cK, n_{0}, \kappa, \varkappa, n_{1}$).

		Let $a_{0}>0$ be a parameter to be specified later depending on $M, \cK, \kappa$, set $n:=\lfloor a_{0}|\log \rho|\rfloor$.	We will select $a_{0}$ so that any element in $\supp(\mu^{*n})$ has controlled distortion. We introduce
		 \begin{align}\label{eq maximalderivative}
		   \maxder = \sup_{g \in \supp(\mu)} \max( \| Dg \|_{L^\infty}, \| D(g^{-1}) \|_{L^\infty} ), \quad \maxderII = \sup_{g \in \supp(\mu)} \max(  \| g \|_{C^2}, \| g^{-1} \|_{C^2} ).
		 \end{align}
		 Note that for every integer $n \geq 1$, for every $g \in \supp(\mu^{*n})$, we have 
		 		 \ary \label{eq Dginmu*g}
		\| Dg \|_{L^\infty}, \| D( g^{-1} ) \|_{L^\infty} \leq \maxder^n.
		 \eary
		 Moreover, by \cite[Lemma 6.4]{FM} and \eqref{eq maximalderivative}, there is a polynomial $P$, depending only on $M$ and the fixed charts, such that the second derivatives in those charts satisfy 
		 \ary \label{eq D2ginmu*g} 
		  \| D^2 g \|_{L^\infty}, \| D^2( g^{-1} ) \|_{L^\infty} \leq \maxder^{2n} P(n \maxderII).
		 \eary
		 Following the notations in \SectionSymbol \ref{sec-gap-pinch}, 
		 for every $x \in M$ and every $g \in  \supp(\mu^{*n})$  the differential map of $g$ at $x$ can be written
		 \aryst
		 Dg(x) = R'(x, g)   \diag(\rho^{ - s_1(x, g) }, \rho^{ - s_2(x, g)}, \cdots, \rho^{ - s_d(x, g)} )   R(x, g)
		 \earyst
		 where   $(s_i(x, g))_{i=1}^d\in \R^d$ are reals in increasing order,		 
		 and $R'(x, g):  \R^d \to T_{g(x)}M$, $R(x, g) :  T_x M \to \R^d$ are isometries. 
		As in \eqref{eq Wixg},  for  $i \in \{1, \cdots, d\}$, we set $W_i(x, g) = R(x, g)^{-1}(\R^{i} \times \{0\}^{d-i})$.

		 Choosing $0 < a_{0} \leq    \min(1,   \frac{1}{20 (\log D_{1})})$, we get by \eqref{eq Dginmu*g} that exponents of singular values are controlled by
		 \aryst
		 - 1/10 \leq s_1(x, g) \leq \cdots \leq s_d(x, g) \leq  1/10.
		\earyst
%
		
		We now aim to apply our multislicing estimate, Theorem \ref{thm Multislicing0},  to the measure  $ \nu$ restricted to a microscopic ball,   and with the covering of random  boxes given by 
		$(g^{-1}B_{\rho^{1/2}}(x))_{x\in M}$  where $g\sim \mu^{* n}$.   Set $\zeta:=\rho^{1/3}$, fix an open ball  $Q\subseteq M$  of radius $\zeta$, centered at $x_{Q}$.
	 To place ourselves in local charts,  we set $\varphi_{Q} := \varphi_{x_Q} |_Q : Q\rightarrow \R^d$, where $ \varphi_{x_Q}$ is given in \SectionSymbol \ref{subsec	linearcharts}. 
Images under $\varphi_Q$ are understood after intersection with its domain $Q$.
		  We set $\nu_{Q}:=(\varphi_{Q})_{*}\nu_{|Q}$ the pushforward of $\nu_{|Q}$ by $\varphi_{Q}$, and 
		\begin{align*}
      \cW_{Q, g} &= (W_{Q, g, i})_{i=1}^d \in \cF \quad \text{where} \quad W_{Q, g, i}= D\varphi_{Q}( W_{i}(x_{Q}, g) ),  \\
        \bt_{Q, g} &=  (t_{Q, g, i})_{i=1}^d \in [0,1]^d \quad \text{where} \quad t_{Q, g, i}= 1/2 + s_{i}(x_{Q},g).
     	\end{align*}

We will compare $B^{  \cW_{Q, g}}_{\rho^{\bt_{Q, g}}}$   with $\varphi_Q  (g^{-1}(B_{\rho^{1/2}}(x)))$ (see \eqref{eq coveringmultiplicity}).
	Observe that  the boxes  $B^{  \cW_{Q, g}}_{\rho^{\bt_{Q, g}}}$  and $D( \varphi_{Q} g^{-1})( B^{T_{g( x_{Q} )}M}_{\rho^{1/2}} )$      are commensurable, in the sense that each one is covered by $O_{d}(1)$-many translates of the other. 
 The fact that $\varphi_Q  (g^{-1}(B_{\rho^{1/2}}(x)))$ can be covered by $O_d(1)$-many translates of $D( \varphi_{Q} g^{-1})( B^{T_{g( x_{Q} )}M}_{\rho^{1/2}} )$   (seen in $\varphi_{Q}$) will be justified below using Lemma \ref{lin-chart}.

	For now, letting $g$ vary with law $\mu^{*n}$, we   check that the random box $B^{  \cW_{Q, g}}_{\rho^{\bt_{Q, g}}}$ satisfies the assumptions of Theorem \ref{thm Multislicing0}. 
	
\begin{itemize}
\item (Exponents).	
Set $\kappa_{*}:=a_{0}\kappa/2$, $\eta_{*}:=2a_{0}\eta$,  and   $\square':=\increasing(\bd, 1, \kappa_{*}, \eta_{*})\cap [2/5, 3/5]^d$   (the $\increasing(\cdot)$ notation was introduced at the beginning of Section \ref{sec BHMT}).    By Lemma \ref{cor GapWithHighProba} (applied with $\kappa'=2\kappa/3$, $\eta'=2\eta$),  we have  
   \aryst \label{eq ComplementOmegajnkappa}
    \mu^{* n}(  g \in G \mid \bt_{Q,g}\in \square'  )    > 1 - \rho^{\eps}
   \earyst
   provided  $\eps \lll_{ \cK,n_{0}, \kappa, \eta, a_{0}, n_{1}} 1$ and $\rho\lll_{\eps }1$. 
\bigskip

\item (Projection estimates along $\cW_{Q, g}$). 
     Recall that $n = \lfloor a_{0} | \log \rho| \rfloor$.
 By Proposition \ref{lem AllGapsImpliesNonConcentration},  there is a constant $c=c(M, \cK, n_{0},\kappa, a_{0})>0$ such that for all $k \in \{1, \cdots, m\}$, $V \in \Gr(d, d- d_k)$, $r \geq \rho$, we have  
   \ary \label{NC-WQg}
   \mu^{* n}(  g \in G \mid  \measuredangle( W_{Q, g, d_k}, V ) <  r )    \leq 2   r^{c}   
   \eary
   Let $\bt\in \square'$. Provided that  $\eps \lll_{\varkappa, c}1$ and  $\rho^{\kappa_*}\lll_{\varkappa, c,\eps}1$, we may apply 
    the supercritical projection Theorem \ref{lem NCtoS} to obtain that (the distribution  of)   
   $$
   \text{$(W_{Q, g, d_k})_{g\sim \mu^{*n}}$ has  the property $\SupP$ with  parameters $(\rho^{t_{d_{2}} -t_{d_{1}}},  \varkappa, \tau')$}
    $$ 
   where $\tau'=\tau'(d, \varkappa, c, \kappa)>0$.
   
   On the other hand,   \eqref{NC-WQg} and Lemma \ref{lem NC+toNC-} together   imply that  $(W_{Q, g, d_k})_{g\sim \mu^{*n}}$ satisfies the weak  non-concentration property $\NCSub$ with parameters $(  \rho^{ t_{d_{k+1}}\sqrt{\eps}}  , c/4)$.   This allows us to apply the subcritical projection Theorem \ref{lem NC-toS-} and obtain that
$$\text{$(W_{Q, g, d_k})_{g \sim \mu^{*n}}$ has  the property $\SubP$ with  parameters $( \rho^{t_{d_{k+1}}},  \eps, D c^{-1} \sqrt{\eps} )$}$$
where $D\lesssim 1$.

\bigskip

\item (Non-concentration of  $\nu_{Q}$). Observe the non-concentration assumption on $\nu$ transfers to $\nu_{Q}$, up to a multiplicative constant (depending  on $\varphi_Q$, thus only on $M$). In particular, provided $\rho\lll_{\eps}1$, we have
$$				
 \nu_{Q}(B_{r} + v) \leq \rho^{-\eps}    r^{d\alpha}, \quad
\forall  r\in [\rho, \rho^{1/4}], \forall v \in \R^d.
$$
\end{itemize}	
\bigskip
The combination of the three previous items justifies that we may\footnote{To be precise, on the event $\bt_{Q,g}\notin\square'$, we temporarily
replace $\bt_{Q,g}$ by a fixed element of $\square'$, leaving
$\cW_{Q,g}$ unchanged. We apply the theorem with $2\eps'$ in
place of $\eps'$ and adjoin this event to the resulting
exceptional set. Since its mass is at most $\rho^\eps$ and
$\rho^{2\kappa_*\eps'}+\rho^\eps\leq\rho^{\kappa_*\eps'}$
for $\rho$ sufficiently small, we obtain the stated bound \eqref{estimate-dim-inc-nuQ0}} apply Theorem \ref{thm Multislicing0} to $(B^{  \cW_{Q, g}}_{\rho^{\bt_{Q, g}}})_{g\sim \mu^{*n}}$ and $ \nu_{Q}$.
Provided we further have $\eps' \lll \eps$;  $\eta, \eps \lll_{a_{0}, \kappa, c, \tau'}1$; and $\rho\lll_{a_{0},\kappa, c, \tau', \eps}1$, this yields a subset $\cE_{Q}\subseteq G$ such that $\mu^{*n}(\cE_{Q})\leq \rho^{\kappa_{*}\eps'}$, and for $g\in G\smallsetminus \cE_{Q}$, a subset  $F_{Q,g}\subseteq \R^d$ with $\nu_Q(F_{Q,g})\leq \rho^{\kappa_{*}\eps'}(\nu(Q)+\zeta^d)$ and  
		\begin{align} \label{estimate-dim-inc-nuQ0}
		\sup_{v \in \R^d}(\nu_{Q}|_{\R^d \smallsetminus  F_{Q,g}} )\left( B^{\cW_{Q, g}}_{\rho^{\bt_{Q, g} }} + v \right) \leq \rho^{  \kappa_{*} \tau'/( 200 d) }   \leb \left(B^{\cW_{Q, g}}_{\rho^{\bt_{Q, g}}}\right)^{\alpha}. 
		\end{align}

We now  compare the right hand side in \eqref{estimate-dim-inc-nuQ0} with the volume of a ball of radius $\rho^{1/2}$. Recalling the assumption that $\mu$ is $(n_{1},\eta)$-almost volume-preserving and $n < | \log \rho |$, Lemma \ref{lem detcloseto1} implies that 
$$\mu^{*n}\left(g : |\sum_{i=1}^d s_i(x_Q,g)| >2\eta \right)\leq 2e^{-nc'}$$
where $c'=c'(M, \cK, n_{1}, \eta)>0$.   Set $\gamma :=  \kappa_{*} \tau' / (1600 d^2) $.   Assuming $\eta<\gamma/10$; $\eps' \lll_{\cK, n_{1},a_{0}, \gamma}1$; $\rho\lll_{\eps'}1$, we  get that
 for some  subset $\cH_{Q} \subset G$ with $\mu^{*n}(\cH_Q) < \rho^{\eps'}$, for  all $g \in G \setminus \cH_Q$, we have  
		 $|   \sum_{i} s_i(x_Q,g)  |  < \gamma/5$.
 In view of \eqref{estimate-dim-inc-nuQ0}, it follows that for every $g \in G \setminus (\cH_Q \cup \cE_Q)$, we have  
		\begin{align} \label{estimate-dim-inc-nuQ}
	\sup_{v \in \R^d}(\nu_{Q}|_{\R^d \smallsetminus  F_{Q,g}} )\left( B^{\cW_{Q, g}}_{\rho^{\bt_{Q, g} }} + v \right) \leq  \rho^{\kappa_{*} \tau'/( 200 d) }  \rho^{(d/2 - \gamma/5) \alpha}  \leq \rho^{d(\alpha + 3 \gamma) / 2}.
\end{align}

		\bigskip

     	We now aim to pull back  \eqref{estimate-dim-inc-nuQ} to $M$ in order to  obtain a dimensional increment estimate with respect to translates $g^{-1}B_{\rho^{1/2}}(x)$ where $x\in M$ and $g\sim \mu^{*n}$. 
 Taking $a_{0}\lll_{D_{1},D_{2}, d}1$,   the Linearization Lemma \ref{lin-chart} (applied with $(\zeta, r)=(\rho^{1/3}, \rho^{1/2})$)  and the distortion  bounds \eqref{eq Dginmu*g}, \eqref{eq D2ginmu*g},  together  imply that for all $x \in M$, all $g\in \supp ( \mu^{*n} )$,
     \begin{align} \label{eq coveringmultiplicity}
     \text{the set $\varphi_Q  (g^{-1}(B_{\rho^{1/2}}(x)))$ is covered by $O(1)$-many  translates of   $B^{\cW_{Q, g}}_{\rho^{\bt_{Q, g}}}$.}
\end{align}

	Recalling \eqref{estimate-dim-inc-nuQ} we then obtain that for  every $g\in G$, there exists $A_{Q,g}\subseteq Q$ such that 
		$$\text{the  measure $\delta_{g}*\nu_{|Q \smallsetminus A_{Q,g}}$ is $(\alpha+2\gamma, \,\cB_{\rho^{1/2}},0)$-robust}$$
		while  with $\mu^{*n}$-probability	at least $1-2\rho^{\kappa_{*}\eps'}$ in $g$, we also have  $$\nu(A_{Q,g})\leq \rho^{\kappa_{*}\eps'}(\nu(Q) +\zeta^d).$$ 
	
	Consider a fixed covering $\mathcal Q$ of $M$ by balls $Q$ of radius $\zeta$  with multiplicity bounded by a constant depending only on $d$.
	Set $\nu'_{g}:=\delta_{g}*\nu_{|M\smallsetminus \cup_{Q\in \mathcal Q}A_{Q,g}}$. For every $g\in\supp(\mu^{*n})$ and $x\in M$, we have
$$
\diam\bigl(g^{-1}B_{\rho^{1/2}}(x)\bigr)
\leq 2D_1^n\rho^{1/2}\leq\zeta,
$$
therefore
		$$\text{the  measure $\nu'_{g}$ is $(\alpha+\gamma, \,\cB_{\rho^{1/2}},0)$-robust}.$$
By integrating over $g\sim \mu^{*n}$, this remains true for $\nu_{n}':=\int_{G}\nu'_{g}\,d\mu^{*n}(g)$. Moreover, $\nu_{n}'':=\mu^{*n}*\nu-\nu_{n}'$ is a non-negative measure of mass bounded by $\sum_{Q\in \mathcal Q}\int_{G}\nu(A_{Q,g})d\mu^{*n}(g)\leq \sum_{Q\in \mathcal Q} (3\nu(Q)+\zeta^d)\rho^{\kappa_{*}\eps'} \leq \rho^{\kappa_{*}\eps'/2}$. In the end we have shown that 
	$$\text{the  measure $\mu^{*n}*\nu$ is $(\alpha+\gamma, \,\cB_{\rho^{1/2}},\, \rho^{\kappa_{*}\eps'/2})$-robust.}$$

	Tracking down dependences in the argument until here, we see we can define $a_{0}$ in terms of $(M,\cK)$, then $c$ in terms of $(M, \cK, n_{0}, \kappa)$, then $\tau', \gamma, \eta$ in terms of $(M, \cK, n_{0}, \kappa, \varkappa)$,  then $\eps, \eps'$ in terms of $(M, \cK, n_{0}, \kappa, \varkappa, n_{1})$, and finally the upper bound on $\rho$ is given in terms of $(M, \cK, n_{0}, \kappa, \varkappa, n_{1},\eps')$. As noted earlier, the dependence on $\eps'$ plays no role in the latter. 
	This concludes the proof.

	\qed

	\section{From large dimension to equidistribution : Proof of Proposition \ref{prop step III}}  \label{sec proof prop step III}
	
   We complete Phase III from the proof strategy, in other terms we prove Proposition \ref{prop step III}. To highlight which assumptions on $\mu$ will play a role, we place ourselves in  the following more general setting.


\begin{setting} 
	Let   $\mu_0$ be a compactly supported probability measure on $ \Diff^2(M, \vol)$. 
	Let $\cK\subseteq  \Diff^2(M)$ be a   compact   set containing the support of $\mu_{0}$. Let $\cU$ be a weak-$\star$ neighborhood of $\mu_{0}$ in $\cP(\cK)$, let $\mu\in \cU$. 
\end{setting}

We show the following.  

   \begin{prop} \label{prop step III-gen} 
    Assume that $\mu_{0}$ is coexpanding, $T_{\mu_0}$ is totally ergodic in $L^2(M, \vol)$, and $\cU\subseteq \cU_{0}(M, \mu_{0},\cK)$. 
    Then there exists a   $\mu$-stationary probability measure $\Upsilon_{\mu}$ on $M$ such that
    the following holds. 
    
    Let $\nu$  be a  $(1 - \eps, \cB_{\rho}, \tau)$-robust probability measure on $M$ for some $\eps, \rho\lll_{\mu_{0}, \cK}1$, and  $\tau \in \R^+$. 
  There exists $c = c(M, \cK) > 0$ such that, for every integer $n \in [c  | \log \rho|, 2c | \log \rho | ]$, we have  
  	\aryst
  	\cW_1(\mu^{* n} * \nu, \Upsilon_{\mu} ) < \rho^{\eps} + \tau.
  	\earyst
Moreover, given any $s>0$, we further have $\Upsilon_{\mu}\in H^{-s}(M)$ provided $\cU\subseteq \cU_{0}(M, \mu_{0},\cK,s)$. 
      \end{prop}

Note that Proposition \ref{prop step III} from Phase III follows at once.

\begin{proof}[Proof of Proposition \ref{prop step III}]
It is a consequence  of Proposition \ref{prop step III-gen}, which applies thanks to Lemma \ref{lem finetocoexpanding}.
\end{proof}

\subsection{Spectral estimates for the Markov operator} \label{Sec-spec}
The proof of Proposition \ref{prop step III-gen}  exploits  the spectral properties of the Markov operator $T_{\mu}$  associated to $\mu$. Recall $T_{\mu}$  acts on bounded measurable functions $f:M\rightarrow [0, +\infty]$ via the formula
$$ T_{\mu}f(x)= \int  f(gx) \dd\mu( g).$$ 
Since $\mu$ is compactly supported on $\diff^1(M)$, the chain rule and change-of-variables formula imply that  $T_{\mu}$ extends to a  bounded operator on both $L^2(M)=H^{0}(M)$ and $H^1(M)$. 
By Sobolev interpolation, $T_{\mu}$ defines a bounded operator on $H^s(M)$ for every $s \in [0,1]$.


The next theorem is essentially extracted from \cite{DD, T3, DTZ}. It states that the coexpanding property of $\mu_{0}$ guarantees an essential spectral gap  for $T_{\mu}$ on $H^s(M)$ for some (small) $s \in (0, 1]$ and $\mu$ close to $\mu_{0}$. Moreover, if $T_{\mu_{0}}$ is totally ergodic in $L^2(M, \vol)$, then this essential spectral gap can be upgraded to a genuine spectral gap. 
	
		

         \begin{thm}[\cite{DD, T3, DTZ}] \label{Thm Spectralgap} 
		 	    Assume  $\mu_{0}$ is coexpanding.
		 	Then there exists   $s_0 = s_0(M, \mu_0,   \cK)  \in  (0, 1]$ such that for all $s \in (0, s_0]$ there exist   $\cU\subseteq \cU_{0}(M, \mu_{0},\cK, s)$ and $\kappa = \kappa(M, \mu_0,   \cK,  s   ) > 0$     such that the essential spectral radius of $T_{\mu}$ on $H^s(M)$ is smaller than $1 - \kappa$. That is, there exists a finite rank operator $P : H^s(M) \to H^s(M)$ such that the spectral radius of $T_{\mu} - P$ is  smaller than $1 - \kappa$.
			
		  			 	
		 	In particular, if $T_{\mu_0}$ is totally ergodic in $L^2(M, \vol)$,  then $T_{\mu}$ has a spectral gap. More precisely, write\footnote{ The spectral projection at $1$ is well defined because $T_{\mu}{\bf 1}= {\bf 1}$ and $1$ is an isolated eigenvalue of $T_\mu$ by the first part of the theorem. Here $\bf 1$ denotes the constant function equal to $1$ on $M$. We identify $\C {\bf 1}$ with $\C$ via $z {\bf 1}\mapsto z$.}  $\Pi_{\mu, 1} : H^{s}(M)\rightarrow \C{\bf 1}\equiv \C$ for the spectral projector at $1$ for $T_{\mu}$. 
		 	Then, up to reducing $\kappa$ and  shrinking $\cU$ if necessary,  there is a  $\mu$-stationary probability measure $\Upsilon_\mu \in H^{-s}(M)$   such that 
			\aryst 
		 	\Pi_{\mu, 1} (u) \equiv \int u(y) \dd\Upsilon_{\mu}(y), \quad u \in C^{\infty}(M),
		 	\earyst 
			and the operator $T_{\mu} - \Pi_{\mu, 1} : H^s(M) \to H^s(M)$ satisfies 
        \aryst
        \sup_{\mu \in \cU} \| (T_{\mu} - \Pi_{\mu, 1})^n \|_{H^{s} \to H^{s}} \leq C(1-\kappa)^n, \quad  n \geq 1
        \earyst
        where $C=C(M,\mu_{0}, \cK,s)>0$.
		 \end{thm} 
		 
%
%
%
%
%
%
%

		 Theorem \ref{Thm Spectralgap} is essentially contained in  Theorem 1.1 and  \SectionSymbol 2.7 in \cite{DD}, but we need to go into the proof to see it.
		 We outline the argument here for the convenience of the reader.

 \begin{proof}[Proof outline of Theorem \ref{Thm Spectralgap}]
	  It is a consequence of  \cite{DTZ} that for every  $0<s\lll_{\mu_{0}, \cK}1$, for $\cU\subseteq \cU_{0}(M, \mu_{0}, \cK, s)$,  for some $\kappa, C > 0$ depending on $\mu_{0},\cK, s$, for some integer $n_0 > 0$ depending on $\mu_{0}$, we have  for all $u\in H^s(M)$: 
				 	\ary \label{eq LasotaYorke2}
		 	\|  T^{n_0}_{\mu} u \|_{H^s(M)} \leq e^{ - n_0 \kappa} \| u \|_{H^s(M)} + C \|  u \|_{L^{2}(M)}.
		 	\eary 
In fact, the proof of \eqref{eq LasotaYorke2} can essentially be found in \cite[Section 6]{DD}, but the latter presents it in the volume-preserving case and under a $C^\infty$-regularity assumption. The extension given in \cite{DTZ} adapts \cite{DD} to the non-volume-preserving setting (see also \cite{T3}) and to $C^2$ regularity, using paracomposition and pseudodifferential estimates for symbols of finite regularity \cite{T1}.

 		 	By Hennion's criterion \cite{H3},  the statement on the essential spectral gap follows  from \eqref{eq LasotaYorke2}.
	
		 	
		 	We now assume that $T_{\mu_0}$ is totally ergodic in $L^2(M, \vol)$, and establish the  claim on the spectral radius. Using that $\mu_{0}$ is volume-preserving, it is direct to see that $T_{\mu_{0}}\acts H^s(M)$ has all its eigenvalues in the closed unit disk, and those located on the unit circle form an abelian semigroup,  which is necessarily a finite abelian group in view of the bound on the essential spectral radius. Total ergodicity then imposes that $1$ is the only eigenvalue on the unit circle, and it has multiplicity one.
       
           Noting that $\| T_{\mu} - T_{\mu_0} \|_{H^s(M) \to L^2(M)} \to 0$ as $\mu$ approaches $\mu_0$ in $\cP(\cK)$, 
             
            as well as \eqref{eq LasotaYorke2}, and $T_{\mu}{\bf 1}={\bf 1}$, one can use the spectral perturbation theorem of Keller-Liverani \cite{KL} (see also \cite[\SectionSymbol 2.7]{DD}) to deduce that the picture still holds for $T_{\mu}$, possibly after shrinking $\cU$,  and the spectral gap between the eigenvalue $1$ and the remainder of the spectrum is bounded from below by a constant depending only on $\mu_{0}$ and $s$. Writing $\Pi_{\mu, 1}$ for the  spectral projector at $1$, this justifies that $T_{\mu}-\Pi_{\mu, 1}$ has spectral radius less than $1-\kappa$ up to reducing $\kappa$. 
        Here $\Pi_{\mu, 1}$ defines a linear form on $H^s(M)$   which is non-negative (because $T_{\mu}$ is) and satisfies $\Pi_{\mu, 1}({\bf 1})=1$.
By the Riesz representation theorem, it must be represented by some probability measure $\Upsilon_{\mu}$ on $M$.          
      The uniformity in the last statement follows from \eqref{eq LasotaYorke2} and the uniformity in \cite{KL}.
         
		 \end{proof}

	We stress  that in the $C^3$-setting,  the total ergodicity condition appearing in  Theorem \ref{Thm Spectralgap}  is automatic under a  $1$-gap assumption. 
	
\begin{thm}[ \cite{DDZ}] \label{Thm Ergodicity0} 
  If $\mu_{0}$ is supported on $\Diff^3(M, \vol)$ and    $\mu_{0}$ has a  $1$-gap, then $T_{\mu_{0}}$ is totally ergodic in $L^2(M, \vol)$.
\end{thm}

	\begin{proof}
	Noting the assumption on $\mu_{0}$ is stable by taking convolution powers, it is enough to check ergodicity, i.e. fixed points for $T_{\mu_{0}}\acts L^2(M, \vol)$ are constant functions. In the smooth regularity case $\supp(\mu_0) \subset \Diff^{\infty}(M, \vol)$,  this follows from \cite[Thm 1.2]{DDZ}. The passage to   $C^{3}$-regularity is explained in \cite[Remark 1.8]{DDZ}, for measures whose support is compact in $\Diff^3(M, \vol)$. In fact, this compactness requirement can be weakened down to our assumptions, where boundedness only concerns $\Diff^2(M, \vol)$. Indeed, by weak-$\star$ approximation and openness of the $1$-gap condition, we may choose a finitely supported probability measure $\sigma$, supported inside $\supp(\mu_0)$, which still has a $1$-gap. Now \cite[Remark 1.8]{DDZ} applies to $\sigma$. Moreover, volume preservation and continuity of the pullback action on $L^2(M,\vol)$ imply that every $T_{\mu_0}$-fixed function is invariant under all maps in $\supp(\mu_0)$, hence is $T_\sigma$-fixed and therefore constant. This justifies ergodicity of $T_{\mu_0}$.
	\end{proof}
	
	We also record that under the assumption that $\mu_{0}^{-1}$ is coexpanding, the $\mu$-walk has a stationary measure in $H^s(M)$. 
		 
	\begin{lemma}[\cite{DD}  Stationary measure in $H^s(M)$] \label{Hs-stat}    Assume  $\mu_{0}^{-1}$ is coexpanding and  $\cU\subseteq \cU_{0}(M, \mu_{0},\cK)$.  Then there exists  an absolutely continuous $\mu$-stationary probability measure on $M$, with density in $H^s(M)$ where $s=s(M,\mu_0,   \cK)>0$.  
 	\end{lemma}

 	\begin{proof}
	This is an adaptation of \cite[Theorem 7.9]{DD} to our context. We give the proof for convenience. We denote by $S_{\mu}$ the adjoint operator of $T_{\mu}$ on $L^2(M)$. It follows from the formula $S_{\mu}f(x)=\int f(g^{-1}x) |\det D_{x}g^{-1}| d\mu(g)$ that $S_{\mu}$ defines a bounded operator on $H^s(M)$ ($s\in [0, 1]$). Let $S^*_{\mu} \acts H^{-s}(M)$ be the dual operator. By density of smooth functions in $H^{-s}(M)$, we have that $S^*_{\mu}$ is given by the same formula as $T_{\mu}$, namely $S^*_{\mu}f=\int_{G}f(g.)d\mu(g)$.   Choosing $s=s(M,\mu_0,   \cK)>0$ and   $\cU\subseteq \cU_{0}(M, \mu_{0},\cK)$ accordingly,   the essential spectral radius of  $S^*_{\mu}  \acts H^{-s}(M)$   must be strictly below $1$ (this is true for  $S^*_{\mu_{0}} = T_{\mu_{0}}$   by \cite[Theorem 1.1]{DD} thanks to the coexpansion hypothesis, and we may transfer it to $S^*_{\mu}$ as for Theorem \ref{Thm Spectralgap}, using the analogous Lasota–Yorke estimate on \(H^{-s}(M)\)). 
    Combined with  $S^*_{\mu} 1=1$,   we get by spectral duality that $S_{\mu}\acts H^s(M)$ has $1$ in its spectrum and essential spectral radius strictly below $1$, therefore $1$ is an eigenvalue of $S_{\mu}$. Let $\phi\in H^s(M)$ be non-zero such that $S_{\mu}\phi =\phi$. After  multiplying by a complex number and restricting to the non-negative real part, this yields some $\phi'\in H^s(M)$ non-zero, non-negative with $\vol(\phi')=1$ and  $S_{\mu} \phi'\geq \phi'$. Using again that $S_{\mu}$ preserves integral against volume (due to $T_{\mu}1=1$), we get $S_{\mu}\phi'=\phi'$, so $\phi' d\vol$ is stationary and yields the desired measure.
 	\end{proof}

\subsection{Equidistribution}

 In order to exploit the previous spectral gap estimates, we will approximate the measure $\nu$ by a density on $M$.  Given $\rho\in (0, 1)$,  we define the \emph{mollification of $\nu$ at  scale $\rho$} to be the probability measure on $M$ satisfying that
for any measurable $f \colon M\rightarrow {[0, +\infty]}$,
\begin{align} \label{defnurho}
\int_M f \dd \nu_\rho = \int_M  \frac{1}{\vol(B_\rho(y))} \int_{B_\rho(y)} f \dd\vol\dd \nu (y).
\end{align} 
 Note that $\nu_{\rho}$ is absolutely continuous with respect to $\vol$, with  Radon-Nikodym derivative given by:
\[
\frac{\dd \nu_{\rho}}{\dd \vol}(x) =  \int_{B_\rho(x)}\frac{1}{\vol(B_\rho(y))}\dd \nu(y).
\]
In particular, if $\nu$ is $(1-\eps, \cB_{\rho}, 0)$-robust, then we have 
\begin{align} \label{bndpsirho}
\normBig{\frac{\dd \nu_{\rho}}{\dd \vol}  }_{L^\infty}\lesssim \rho^{-d\eps}.
 \end{align} 

We will also use the following classical result.
 \begin{lemma}\label{Hs-appr}
There is a constant $C = C(M) > 0$ such that for every $s \in [0, 1]$, for every Lipschitz function $f : M \rightarrow \R$, we have  $\| f \|_{H^s(M)}\leq  C \| f \|_{\Lip}$.
 \end{lemma}
 \begin{proof}
 	By Rademacher's theorem, a Lipschitz function $f$ is differentiable almost everywhere, with differential clearly bounded  by the Lipschitz constant of $f$.  Given $s\in [0, 1]$, it follows that $\|f\|_{H^s}\leq \|f\|_{H^1}\leq C(M) \|f\|_{\Lip}$.
 \end{proof}


	
	We are now ready to establish the equidistribution result Proposition \ref{prop step III-gen}.
	
	\begin{proof}[Proof of Proposition \ref{prop step III-gen}] 

	  Without loss of generality, we may assume $\tau=0$ and $s \in (0, s_0]$ where $s_{0}$ comes from Theorem \ref{Thm Spectralgap}.

	We let $c, \eps, \rho\in (0, 1/10)$, $n\in \N$ be parameters to be specified later. We let $\nu$ be a $(1-\eps, \cB_{\rho}, 0)$-robust measure on $M$.

   By Theorem \ref{Thm Spectralgap}, provided $\cU\subseteq \cU_{0}(M,\mu_{0}, \cK, s)$, we have  that $T_{\mu} -\Pi_{\mu, 1}$ has spectral radius at most $1-  \kappa$ on $H^{s}(M)$ where $\kappa>0$ depends only on $M, \mu_{0},  s$. 
 Moreover, we know that $\Pi_{\mu, 1}$ is given by integration against a $\mu$-stationary probability measure $\Upsilon_{\mu}\in H^{-s}(M)$.
   We let $f : M\rightarrow \R$ be  a Lipschitz function such that $\|f\|_{\Lip}\leq 1$ and $\Pi_{\mu, 1}(f)=\int_{M} f \,d \Upsilon_{\mu}=0$.  
	
We compute 
\begin{align*}
| \mu^{*n}*\nu(f)|= \abse{\int_{M} T_\mu^n f \dd \nu} 
\leq \abse{\int_{M} T_{\mu}^n f \dd \nu_\rho}  + \abse{\int_{M} T_{\mu}^n f \dd \nu_\rho - \int_{M} T_{\mu}^n f \dd \nu}.
\end{align*}
To bound the first term, we observe
\begin{align*}
\abse{\int_{M} T_\mu^n f \dd \nu_\rho}  
& \leq \norm{T_\mu^n f }_{L^1} \normBig{ \frac{\mathrm{d}\nu_\rho}{\dd\vol}}_{L^\infty}\\
&\leq \norm{T_\mu^n f }_{H^s} \normBig{ \frac{\mathrm{d}\nu_\rho}{\dd\vol}}_{L^\infty}\\
&\lesssim_{\mu_{0}, \cK,   s}  (1-\kappa)^n \|f \|_{H^s}\rho^{-d \eps}\\
&\lesssim_{\mu_{0}, \cK, s}   (1-\kappa)^n \rho^{- d \eps},
\end{align*}
where the penultimate inequality uses the uniform spectral radius estimate for $T_{\mu}-\Pi_{\mu, 1}$ on $H^s(M)$ from Theorem \ref{Thm Spectralgap},  as well as   the assumption that $\nu$ is $(1-\eps, \cB_{\rho}, 0)$-robust (via its consequence \eqref{bndpsirho}), and the last inequality uses Lemma \ref{Hs-appr} to guarantee $\|f \|_{H^s} \lesssim \|f \|_{\Lip}\leq1$.  

From the definition of $\nu_{\rho}$, the second term is bounded by
\begin{equation*}
\abse{\int_{M} T_\mu^n f \dd \nu_\rho - \int_{M} T_\mu^n f \dd \nu} 
\leq   \rho \|T_\mu^n f\|_{\Lip} \leq  \rho D^n
\end{equation*}
where  $D=D(M,\cK)\geq1$.  

Combining the above estimates, we have shown that 	
\begin{align*}
| \mu^{*n}*\nu(f)|\lesssim_{\mu_{0}, \cK,   s} (1-\kappa)^n\rho^{-d\eps} + \rho D^n.
\end{align*}
Choosing $n \leq 2c |\log \rho|$ where $c>0$ is small enough in terms of $M, \cK$, we have the upper bound $\rho D^n\leq \rho^{1/2}$. On the other hand, asking $n \geq c |\log \rho|$ and $\eps$ small enough in terms of $ \mu_{0}, c, \kappa$, we have $(1-\kappa)^n\rho^{- d \eps}\leq \rho^{2\eps}$. In the end we have shown
\begin{align*}
| \mu^{*n}*\nu(f)|\lesssim_{\mu_{0}, \cK,   s}  \rho^{2\eps}+\rho^{1/2}. 
\end{align*}
Taking $\rho\lll_{ \cK, \kappa, \eps,   s}1$, the right-hand side is smaller than $\rho^{\eps}$. The result follows.   More precisely, the first claim in Proposition \ref{prop step III-gen} is obtained by applying the above for a fixed $s$, say $s=s_{0}$, determined by $M,\mu_{0}, \cK$. The claim about $H^{-s}(M)$-regularity of $\Upsilon_{\mu}$ comes from the possibility to choose $s$ arbitrarily in $(0,s_{0}]$ and the observation that $\Upsilon_{\mu}$ does not depend on $s$ (as follows from the first claim).   
 	\end{proof}

 \section{Proofs of Theorem  \ref{main thm 2Dissipative} and applications} \label{sec Proofs123}
	
We use Theorem \ref{thm mainDissipative} to derive all the other statements collected in the introduction. We start with Theorem \ref{main thm 2Dissipative}, regarding the full gap context. 	
	
	\begin{proof}[Proof of Theorem \ref{main thm 2Dissipative}]
By Lemma \ref{rem allgapstopinch}, the full gap assumption implies full pinching. By Theorem  \ref{Thm Ergodicity0} and the   $C^3$-regularity assumption,  the full gap assumption implies that $T_{\mu_{0}}$ is totally ergodic in $L^2(M, \vol)$. We may thus apply Theorem \ref{thm mainDissipative} to obtain the desired conclusion. 
\end{proof}
	
	We now turn to applications of Theorem \ref{thm mainDissipative} in the various  situations mentioned in \SectionSymbol \ref{Sec-examples}. We start with the surface case.
	
	\begin{proof}[Proof of Theorem \ref{examplesI,II,III}, case I] 
	
 
		By Lemma \ref{lem contractinggaptocoexpanding0}, the assumption that  $\mu_0$ and $\mu_0^{-1}$ are expanding implies that they have a $1$-gap, and are also coexpanding.   By Lemma \ref{rem allgapstopinch}, the full gap assumption implies full pinching.
        Moreover, the expanding property of $\mu_0$ clearly implies that $\mu_0^{*m}$ is expanding for every integer $m>0$. This implies that  $T_{\mu_0}$ is totally ergodic (by \cite[Prop 3.1]{C} and because the set of $\Gamma_{\mu_0}$-orbits of cardinality bounded by any given threshold is finite).
         
        In particular, we may apply Theorem \ref{thm mainDissipative} to get conclusions (a) and (b) from Theorem \ref{examplesI,II,III}, except that we don't know yet that $\Upsilon_{\mu}$ is in $H^s(M)$. This point follows from Lemma \ref{Hs-stat} and the uniqueness of the atom-free stationary probability measure $\Upsilon_{\mu}$.

	\end{proof}

	We now deal with perturbations of Zariski-dense linear random walks on the torus $\T^d$.
		
	\begin{proof}[Proof of Theorem \ref{examplesI,II,III}, case II]
  It is known that  $T_{\mu_{0}}$ is totally ergodic,  see \cite[Chap I, Prop 1.5]{BekkaMayer}.     
		In view of Theorem \ref{thm mainDissipative}, and Lemmas  \ref{rem allgapstopinch}, \ref{Hs-stat}, it is sufficient to show that $\mu_{0}$ has a $b$-gap for all 
		 $b\in \{1, \dots, d-1\}$, and that $\mu_{0}^{-1}$ is both expanding and coexpanding. Noting that $\mu_{0}^{-1}$ satisfies the same assumption as $\mu_{0}$ and recalling that full gap implies expansion and coexpansion (see Lemmas \ref{rem allgapstopinch}, \ref{lem finetocoexpanding2}, \ref{lem finetocoexpanding} ), it is in fact sufficient to establish the full gap property for $\mu_{0}$. This is the content of the next lemma. 
		
		\begin{lemma}  \label{lem mu0neighborhoodfine}
			The measure $\mu_{0}$ has a $b$-gap for every $b\in \{1, \dots, d-1\}$.
		\end{lemma}
		
		As a preliminary to the proof, we recall a few standard results for the $\mu_{0}$-walk on the exterior algebra $\bigwedge^*\R^d$. 
		Write $\sigma_{1}\leq \dots\leq \sigma_{d}$ for the Lyapunov exponents for the $\mu_{0}$-walk on $\R^d$. It is known that the $\sigma_{i}$'s are mutually distinct reals \cite[Corollary 10.15]{BQ}. Note also that for every $k\in \{1, \dots, d-1\}$, the action of $\SL(d, \R)$ on  $\bigwedge^k\R^d$ is irreducible. It follows from \cite[Theorem 14.20]{BQ} that for every unit vector $w\in \bigwedge^k\R^d$, one has 
			\begin{align}\label{eq-cv-L1-LLN}
				n^{-1}\int_{\SL_{d}(\R)} \log \|gw\| \,\,\dd\mu_{0}^{*n}(g) \rightarrow \sigma_{d-k+1}+\dots +\sigma_{d},
			\end{align}
			and the convergence is uniform in $w$. By \cite[Theorem 14.19]{BQ}, the estimate also holds for the operator norm $\|g\|_{\wedge^{k}\R^d}$ in place of $ \|gw\|$: 
			\begin{align}\label{eq-cv-norm-LLN}
				n^{-1}\int_{\SL_{d}(\R)} \log \|g\|_{\wedge^{k}\R^d} \,\,\dd\mu_{0}^{*n}(g) \rightarrow \sigma_{d-k+1}+\dots +\sigma_{d}.
			\end{align}		
These formulas also hold for $k=0,d$, with the empty sum interpreted as $0$.
		
		\begin{proof}[Proof of Lemma \ref{lem mu0neighborhoodfine}] 	
				Let $b\in \{1, \dots, d-1\}$. We check that $\mu_{0}$ has a $b$-gap. 		Given a subspace $S\in \Gr(d,k)$, we write $e_{S}	\in \wedge^{k}\R^d$ for a unit vector spanning the line $\wedge^{k}S$.	Let $V\in \Gr(d, d-b)$, let $u\in V^\perp$ with norm $1$.  Observe that
			\begin{align*}
				\|g e_{(V+\R u)}\|= \|g (e_{V}\wedge u)\| =\|g e_{V}\| \|e_{gV}\wedge gu\|= \|g e_{V}\|\| P_{(gV)^{\perp}} gu  \|.
			\end{align*}
			It follows that					
			\begin{align}\label{eq-proj-rationorm}
				\| P_{(gV)^{\perp}} gu  \|\leq \frac{\|g\|_{\wedge^{d-b+1}\R^d}}{\|ge_{V}\|} 
			\end{align}
			Applying \eqref{eq-cv-L1-LLN}, \eqref{eq-cv-norm-LLN}, we deduce that uniformly in $V$, 
			\begin{align}\label{eqsupE}
				\limsup_{n} n^{-1}	\int_{\SL_{d}(\R)}  \log \sup_{u \in \S(V^{\perp})} \| P_{(gV)^{\perp}} gu  \|  \dd\mu_{0}^{* n}(g) 	 \leq  \sigma_{b}		
			\end{align}
			On the other hand, by using Cartan's decomposition, one sees that
			\begin{align}\label{eq-inf-rationorm}
				\inf_{u \in \S(V)} \| gu \| \geq \frac{\|g e_{V}\|}{\|g\|_{\wedge^{d-b-1}\R^d }}.		 
			\end{align}
			Applying \eqref{eq-cv-L1-LLN}, \eqref{eq-cv-norm-LLN}, we deduce that uniformly in $V$, 
			\begin{align}\label{eqinfE}
				\liminf_{n} n^{-1} \int_{\SL_{d}(\R)} \log \inf_{v \in \S(V)} \| g v \|   \dd\mu_{0}^{* n}(g) 		\geq \sigma_{b+1}		
			\end{align}
			Recalling $\sigma_{b}<\sigma_{b+1}$, the combination of \eqref{eqsupE} and \eqref{eqinfE} shows that $\mu_{0}$ has a $b$-gap.

		\end{proof}

	As noted earlier, Theorem \ref{examplesI,II,III} case II follows.	
	
	\end{proof}

	We now deal with  perturbations of random walks on cocompact lattice quotients of $\SO(2,1)^{\circ}$ and $\SO(3,1)^{\circ}$.

	\begin{proof}[Proof of Theorem \ref{examplesI,II,III}, case III]
In this proof, we also write $\mu_0$ for the measure $\mu'_0$ on $H$.
		  We first deal with the case where $H=\SO(2,1)^{\circ}$. The operator $T_{\mu_{0}}$ acting on $L^2(H/\Lambda, \vol)$ is totally ergodic in view of the Moore Ergodicity theorem \cite[Theorem 2.2.6]{Zimmer} (or \cite{Bekka} which even shows spectral gap). 
		 Similarly to  case II, it is thus sufficient to establish that $\mu_{0}$ has full gap: 
		
		\begin{lemma} If $H=\SO(2,1)^{\circ}$  then
			the measure $\mu_{0}$ has a $b$-gap for every $b\in \{1, 2\}$.					\end{lemma}	
		\begin{proof}
			It is the same proof as in the torus case (Lemma  \ref{lem mu0neighborhoodfine}), using that the adjoint action of $H$ on its Lie algebra $\mathfrak h$ and its dual  $\wedge^2\mathfrak{h}$ are irreducible, and that the random walk induced by $\Ad_{\star}\mu_{0}$ on $\mathfrak h$ has simple Lyapunov spectrum of the form $\sigma_{1}<\sigma_{2}=0<\sigma_{3}=-\sigma_{1}$.
		\end{proof}

	We now deal with the case where $H=\SO(3,1)^{\circ}$. In this case, $\mu_{0}$ does not have full gap, in particular we cannot apply Theorem  \ref{main thm 2Dissipative}. However  we can check the relevant pinching property and conclude by applying Theorem  \ref{thm mainDissipative}. The main step is the following.
		
		\begin{lemma}\label{caseSO(3,1)-gapinching}
		 If $H=\SO(3,1)^{\circ}$  then
			the measure $\mu_{0}$ has a $b$-gap for every $b\in \{2, 4\}$. Moreover, given any $\eta>0$, there exists $n_{0}$ such that $\mu_{0}$  is $(n_{0}, \eta, b_{0},b_{1})$-pinched for $(b_{0},b_{1})\in \{(0, 2), (2, 4), (4,6)\}$.					
		\end{lemma}
		
		\begin{proof}
		Without loss of generality, we may consider  $H=\SL_{2}(\C)$ (seen as a real Lie group) as it is a double cover of $\SO(3,1)^{\circ}$. 
			Let  $\sigma_{1}\leq \sigma_{2}\leq \dots	\leq \sigma_{6}$ denote the Lyapunov exponents for the adjoint random walk driven by $\Ad_{\star}\mu_{0}$. It is known that for some $\lambda>0$, we have $\sigma_{1}=\sigma_{2}=-\lambda$, $\sigma_{3}=\sigma_{4}=0$, and $\sigma_{5}=\sigma_{6}=\lambda$. By \cite[Proposition 3.1 (iii)]{H4}, we also know that the proximal $2$-plane satisfies a non-concentration property.   More precisely, for $h\in \SL_{2}(\C)$, write $h=R'_{h}\diag(e^{-t_{h}}, e^{t_{h}}) R_{h}$ a Cartan decomposition,  meaning $R'_{h},R_{h}\in \SU(2)$, and  $t_{h}\in \R^+$.   Then \cite[Proposition 3.1 (iii)]{H4} implies that for some constant $c=c(M, \mu_{0})>0$, for any $4$-plane $V\in \Gr(\mathfrak{sl}_{2}(\C), 4)$, $r>e^{-n}$, we have
			\begin{equation}\label{cit-He}
				\mu_{0}^{*n} \{h\,:\, \measuredangle(\Ad(R'_{h})\C E_{2,1}, V)\leq r\}\ll_{\mu_{0}} r^c.
			\end{equation}
			This result states a non-concentration property for the real $2$-plane $\Ad(R'_{h})\C E_{2,1}$ spanned by the two longest axes of the ellipsoid $\Ad(h)B^{\mathfrak{sl}_{2}(\C) }_{1}$.
			Applying this result to the image of $\mu_{0}$ by $h\mapsto h^{-1}$ instead of $\mu_{0}$, we obtain, up to taking smaller $c$:
			\begin{equation}\label{nc-2plane}
				\mu_{0}^{*n} \{h\,:\, \measuredangle(\Ad(R^{-1}_{h})\C E_{1,2}, V)\leq r\}\ll_{\mu_{0}} r^c.
			\end{equation}
			This  estimate now concerns the $2$-plane $\Ad(R^{-1}_{h})\C E_{1,2}$ which is most contracted by $\Ad h$.  
			
			Similarly, applying \eqref{cit-He} to the image of $\mu_{0}$ by $h\mapsto \, h^*:=\overline{h}^{\,t}$ (with again suitable $c$), then  taking orthogonal complements, we get for every $V'\in \Gr(\mathfrak{sl}_{2}(\C), 2)$, $r>e^{-n}$
			\begin{equation}\label{nc-4plane}
				\mu_{0}^{*n} \{h\,:\, \measuredangle(\Ad(R^{-1}_{h}) (\C E_{1,2}\oplus \C (E_{1,1}-E_{2,2})), V')\leq r\}\ll_{\mu_{0}} r^c.
			\end{equation}
			This  estimate states non-concentration for the $4$-plane $\Ad(R^{-1}_{h}) (\C E_{1,2}\oplus \C (E_{1,1}-E_{2,2}))$ which is most contracted by $\Ad h$.

			We now derive estimates for matrix norm and coefficients of the adjoint random walk, driven by  $\Ad_{\star}\mu_{0}$. First,  the law of large numbers for the Cartan projection 
			\cite[Theorem 13.17]{BQ} guarantees for every $k=1, \dots, 6$, 
			\begin{align}\label{eq-cv-L1-LLN-norm}
				n^{-1}\int_{H} \log \|\Ad h\|_{\wedge^k\mathfrak{sl}_{2}(\C)} \,\,\dd\mu_{0}^{*n}(h) \rightarrow \sigma_{6-k+1}+\dots +\sigma_{6}.
			\end{align}
			Combining this with \eqref{nc-2plane}, \eqref{nc-4plane}, we find that for   $l=1,2$, uniformly in $V\in \Gr(\mathfrak{sl}_{2}(\C), 2l)$,
			\begin{align}\label{eq-cv-L1-LLN-24}
				n^{-1}\int_{H} \log \|\Ad h (e_{V})\| \,\,\dd\mu_{0}^{*n}(h) \rightarrow \sigma_{6-2l+1}+\dots +\sigma_{6}.
			\end{align}
			Applying now \eqref{eq-proj-rationorm}, \eqref{eq-inf-rationorm} with the inequalities just obtained \eqref{eq-cv-L1-LLN-norm}, \eqref{eq-cv-L1-LLN-24}, the claim of the lemma follows at once from the definition of $b$-gap and pinching.  
		\end{proof}
		
Lemma \ref{caseSO(3,1)-gapinching} validates assumption (1) in Theorem \ref{thm mainDissipative}. For assumption (2),	observe that Lemma \ref{caseSO(3,1)-gapinching} also applies to $\mu_{0}^{-1}$ (which satisfies the same assumptions as $\mu_{0}$), and in particular, $\mu_{0}^{-1}$ is  expanding in view of Lemma \ref{lem finetocoexpanding2}. Note in passing that for the same reason,  $\mu_{0}^{-1}$ is also coexpanding by Lemma  \ref{lem finetocoexpanding}. To check assumption (3), note that the total ergodicity of $T_{\mu_{0}}$ in $L^2(H/\Lambda, \vol)$ follows from the Moore Ergodicity theorem \cite[Theorem 2.2.6]{Zimmer}, or   \cite{Bekka}. 
	
	The above allows us to use Theorem \ref{thm mainDissipative}. The remaining claim on $H^s$-regularity of $\Upsilon_{\mu}$ follows from Lemma \ref{Hs-stat}, which applies because we just checked coexpansion for $\mu^{-1}_{0}$. This concludes the proof.
	\end{proof}

\begin{rema}\label{disc-moregen-simplegroups}	For more general simple Lie groups, it is not true that the $b$-gap property holds for all $b$ such that $\sigma_{b}<\sigma_{b+1}$, in particular we cannot apply Theorem \ref{thm mainDissipative}. This fails already for $\SL_{3}(\R)$: in this case $\Ad_{\star}\mu_{0}$  satisfies $\sigma_{5}=0<\sigma_{6}$ but does not have a $5$-gap. To see why, recall from  Proposition \ref{lem AllGapsImpliesNonConcentration} that having a $5$-gap would imply that the Borel subspace of $\mathfrak{sl}_{3}(\R)$ given by
	$$\mathfrak b:=\begin{pmatrix} *&*&* \\ 0 &*&*\\0&0&*\end{pmatrix}\cap \mathfrak{sl}_{3}(\R)$$
	is transverse, in the sense that for every $W\subseteq \mathfrak{sl}_{3}(\R)$ with $\dim W=3$, there exists some $h\in \SL_{3}(\R)$ such that  
	$\Ad h(\mathfrak b)\cap W=\{0\}$. 
	However this fails:  for 
	$$W_{0}:=\left\{\begin{pmatrix} t&0&0 \\0&t&0\\ * &*&-2t\end{pmatrix}\,:\, t\in \R \right\},$$
	we have\footnote{
		Indeed, write $P$ (resp $P^-$) for the upper (resp. lower) triangular subgroup of $H=\SL_{3}(\R)$.  Note $\kb$ and $W_{0}$ are respectively invariant under $P$ and $P^-$. As they also have nontrivial intersection, we get $\Ad(h)\kb\cap W_{0}\neq \{0\}$ for all $h\in P^-P$. But $P^-P$ is Zariski-dense in $H$ and the nontrivial  intersection condition is Zariski-closed. Hence $\Ad(h)\kb\cap W_{0}\neq \{0\}$ for all $h\in H$.
	} $\Ad(h)\kb\cap W_{0}\neq \{0\}$ for all $h$.
	
\end{rema}	
	
	Finally, we deal with perturbations of linear random walks on the sphere. 

	\begin{proof}[Proof of Theorem \ref{thm Sphere}] 

				The following result follows from \cite{DK, DDZ}.

		\begin{lemma} \label{lem:rig-l} In the context of Theorem \ref{thm Sphere}, if  $\ell_{0}\in \N$ is large enough depending on $d$, and $\mathcal{V}$ is small enough depending on $\mu_{0}$, then exactly one of the following holds:
			\enmt
			 \item either there exist $h\in\Diff^1(M)$ and $A_1,\ldots,A_k\in\SO(d+1)$ such that we have $h\circ S_i\circ h^{-1}=A_i$
for all $1\leq i\leq k$,
			\item  or  $\mu$   has a $b$-gap  for every $1 \leq b \leq d - 1$. 
			\eenmt
      Moreover, in item (1),  the $A_i$ can be taken arbitrarily close
to the $R_i$ up to shrinking $\mathcal V$ accordingly.     
		\end{lemma}

		\begin{proof} 		
			The proof showing $1$-gap if item (1) fails is contained in   \cite[Prop 2.1]{DDZ}   using \cite{DK}. Although formulated for smooth diffeomorphisms, the argument applies under sufficiently high finite regularity: the conjugacy construction in \cite[Section 3.4]{DK}, together with the estimates in \cite[Section 6.1]{DK}, uses only finitely many derivatives and yields a $C^1$ conjugacy close to the identity.   We recall the argument and explain why it yields a $b$-gap for every $1\leq b\leq d-1$.

        If $\mu$ does not satisfy item (1), then \cite[Theorem 1]{DK}, together with volume preservation, implies that the top integrated Lyapunov exponent of $\mu$ with respect to $\vol$ is positive.
		By  \cite[Corollary 4]{DK}   and its proof (more precisely, (45) in \cite[Proof of Corollary 4]{DK})  applied with $r =d- b$ for every $1 \leq b \leq d-1$,  we deduce item (2). 
		\end{proof}

		We  now deduce $\mu^{*n} * \delta_x \rightarrow \vol$ for every $x$ with infinite $\Gamma_\mu$-orbit.
        In case (1), let $h$ and $A_i$ be as in
 Lemma \ref{lem:rig-l}. Up to taking $\mathcal{V}$ small enough from the start so that the $A_i$ are sufficiently
close to the $R_i$, we must have that the $A_i$ generate a dense
subgroup of $\SO(d+1)$, by openness of
dense generation in compact connected
semisimple Lie groups. By  \cite[Proposition 3.3]{BB}, it follows from density that some power of $\frac{1}{k}\sum^k_{i=1}\delta_{A_i}$ is aperiodic.
Moreover, since the $S_i$ preserve $\vol$,
the measure $h_*\vol$ is invariant under
every $A_i$, hence under $\SO(d+1)$.
It follows that $h_*\vol=\vol$.
Applying the It\^o--Kawada theorem
\cite{IK} to
$\mu'=\frac{1}{k}\sum_{i=1}^k \delta_{A_i}$
and pulling back the convergence by $h$,
we obtain the desired equidistribution
for $\mu$.

In case (2), the same dichotomy applied to $\mu^{-1}$ shows that $\mu^{-1}$ also has full gap, hence is expanding by Lemmas \ref{rem allgapstopinch} and \ref{lem finetocoexpanding2}. Since $\mu$ preserves $\vol$, we may apply Theorem \ref{main thm 2Dissipative} with $\mu$ in place of $\mu_0$, obtaining $\mu^{*n}*\delta_x\to\vol$ for every $x$ with infinite $\Gamma_\mu$-orbit.

			\end{proof}
	

\appendix
	\section{Large deviation estimate}	 \label{app-LDP}
In this appendix we record a general large deviation estimate for Markov chains. The argument is a standard adaptation of the large deviation principle for sums of i.i.d. variables and appears in the literature in several more specific settings. For convenience, we present the result here in an abstract form. In the present paper we apply it in two different contexts, and we hope that this formulation provides a convenient reference for future works.
	
\begin{prop}\label{LDP-MC}
	Let $E$ be a measurable space. Let $f:E\rightarrow \R$ be a bounded measurable function.  Let $(X_{n})_{n\geq 0}$ be a time-homogeneous Markov chain on $E$. We denote by  $\mathbb{P}_{x}$ the probability distribution for the Markov chain conditional on $X_{0}=x$, and  write $\EV_{x}$ for the corresponding expectation.
	
	 Assume that   for every $x\in E$, we have 
	$$\EV_{x}\left(f(X_{1}) \right) \leq 0.$$
	Then for every $\eps \in(0,1/2]$,  for $\gamma := \frac{\eps}{(1+\sup |f|)^2}$,  for all $x\in E$ and all integers $n\geq1$, we have
	\ary \label{LDP-mom} 
	\EV_{x}\left(\exp(\gamma  \sum_{i=1}^{n} f(X_{i})) \right)\leq \exp(\gamma \eps n).
	\eary
	In particular, 
	\ary\label{Markov-cons}
	\mathbb{P}_{x}\left( \sum_{i=1}^{n} f(X_{i}) \geq 2\eps n \right) \leq  \exp(-\gamma \eps n).
	\eary
\end{prop}

	\begin{proof}
	Note that \eqref{Markov-cons} is a direct consequence of \eqref{LDP-mom} and the Markov inequality. Therefore we focus on proving \eqref{LDP-mom}. 
	Set $R=\sup_{E} |f|<\infty$. Noting that $\gamma R<1$ and  that $e^t\leq 1+t+t^2$ for all $t\in [-1,1]$, we have for every $x\in E$, 
	\begin{align*}
	\EV_{x}\left(\exp(\gamma  f(X_{1})) \right)
	\leq 1+ \gamma \EV_{x}(  f(X_{1}) )+ \gamma^2 \EV_{x}(  f(X_{1})^2)
	\leq  1+ \gamma^2 R^2
	\end{align*}
	where the second inequality uses the assumption $\EV_{x}(  f(X_{1}) )\leq 0$ and the definition of $R$.
	As our choice of $\gamma$ guarantees $\gamma R^2\leq \eps$, we obtain 
		\begin{align*}
	\EV_{x}\left(\exp(\gamma  f(X_{1})) \right)
	&\leq \exp(\gamma \eps).
		\end{align*}
	The Markov property then yields \eqref{LDP-mom}.
	\end{proof}

				\end{document}